\documentclass[notitlepage,12pt]{article}
\usepackage{tablefootnote}
\usepackage{enumerate}
\usepackage[OT1]{fontenc}
\usepackage{multirow}
\usepackage{amssymb,mathrsfs,amsmath,amsthm}
\newtheorem{theorem}{Theorem}
\newtheorem{asm}{Assumption}
\newtheorem{tl}{Technical Lemma}
\newtheorem{lemma}{Lemma}
\newtheorem{corollary}{Corollary}
\usepackage{graphicx}
\usepackage{mathtools}
\usepackage[usenames]{color}
\usepackage[dvipsnames]{xcolor}
\usepackage{dsfont,makecell}
\usepackage{multirow,stfloats,booktabs}
\usepackage{caption}
\usepackage{subcaption}
\usepackage[colorlinks,
            linkcolor=red,
            linktoc=all,
            anchorcolor=red,
            citecolor=blue
            ]{hyperref}
\usepackage[toc,page]{appendix}
\usepackage{setspace}
\usepackage[margin=1in]{geometry}
\usepackage[numbers,square]{natbib}
\everymath{\displaystyle}
\usepackage{chngcntr}
\counterwithout{equation}{section}

\usepackage{tabularx,array}
\newcolumntype{M}{>{\centering\arraybackslash}m{0.17\textwidth}}

\DeclareMathOperator{\cov}{cov}
\DeclareMathOperator{\var}{var}

\DeclareMathOperator{\diag}{diag}

\newcommand{\calB}{\mathcal B}

\newcommand{\calS}{\mathcal S}
\newcommand{\calG}{\mathcal G}

\newcommand{\calP}{\mathcal P}
\newcommand{\calL}{\mathcal L}

\newcommand{\calI}{\mathcal I}

\newcommand{\calD}{\mathcal D}
\newcommand{\calF}{\mathcal F}
\newcommand{\calW}{\mathcal W}

\newcommand{\eps}{\epsilon}

\newcommand{\given}{\,|\,}

\newcommand{\Fo}{F^{(1)}}
\newcommand{\Fz}{F^{(0)}}
\newcommand{\Ft}{F^{(2)}}
\newcommand{\Fi}{F^{(i)}}
\newcommand{\tfo}{\widetilde{F^{(1)}}}
\newcommand{\tfz}{\widetilde{F^{(0)}}}
\newcommand{\intrd}{\int_{\mathbb R^d}}

\newcommand{\iid}{\overset{\mbox{iid}} \sim}

\usepackage{hyperref}
\def\##1\#{\begin{align}#1\end{align}}
\def\$#1\${\begin{align*}#1\end{align*}}
\usepackage{enumitem}

\def\spacingset#1{\renewcommand{\baselinestretch}%
{#1}\small\normalsize} \spacingset{1}

\newcommand{\blue}[1]{{\leavevmode\color{blue}{#1}}}

\usepackage{xr}
\makeatletter

\newcommand*{\addFileDependency}[1]{
\typeout{(#1)}
\@addtofilelist{#1}
\IfFileExists{#1}{}{\typeout{No file #1.}}
}\makeatother

\begin{document}
\title{\bf Infill Consistent Estimability of the Regression Slope Between Gaussian Random Fields Under Spatial Confounding}

\author{Abhirup Datta \\
Department of Biostatistics, Johns Hopkins University  \and  
Michael L. Stein \\
Department of Statistics, Rutgers University}
\date{}

\maketitle

\begin{abstract}
	The problem of estimating the slope parameter in regression between two spatial processes under confounding by an unmeasured spatial process has received widespread attention in the recent statistical literature. Yet, a fundamental question remains unresolved: when is this slope 
	consistently estimable under spatial confounding, with existing insights being largely empirical or estimator-specific.  
	In this manuscript,  
	we characterize conditions for consistent estimability of the regression slope 
	between Gaussian random fields (GRFs),
	the common stochastic model for spatial processes, under spatial confounding.
	Under fixed-domain (infill) asymptotics, we give sufficient conditions for consistent estimability
	in terms of the smoothness or local behavior of the exposure and confounder processes. When estimability holds, 
	we provide consistent estimators of the slope using local differencing (taking discrete differences or Laplacians of the processes of suitable order).
	Using functional analysis results on Paley-Wiener spaces, we then provide an easy-to-verify necessary condition for consistent estimability of the slope in terms of the relative 
	spectral tail decays of the confounder and exposure.
	As a by-product, we establish a novel and general spectral condition on the
	equivalence of measures on the paths of multivariate GRFs with component fields of varying smoothnesses, a result of independent importance.
	Our estimability results or estimators do not rely on specific parametric models for the covariance functions.
	We show that for many covariance classes like the Matérn, power-exponential, generalized Cauchy, and coregionalization families, the necessary and sufficient conditions become identical, except for a boundary point, thereby providing a sharp characterization of consistent estimability of the slope
	for these processes. 
	The results are extended to multivariate slopes and to accommodate measurement error using local-averaging-and-differencing-based estimators. We show that differencing-based estimation remains consistent for popular classes of non-stationary Gaussian random fields, some non-Gaussian random fields, and irregular designs.
	The finite sample behavior of the estimators is explored
	via numerical experiments. 
\end{abstract}

\section{Introduction} 
We consider the problem of estimating the regression slope of an observed spatial outcome process $Y(s)$ on an observed spatial exposure process $X(s)$ under the presence of an unmeasured process $W(s)$ which impacts both $X$ and $Y$. This problem of spatial confounding has now been the focus of a large and burgeoning literature.
We refer the readers to \cite{clayton,wakefield,reich,hodges,hanks,page,papadogeorgou,thaden,gilbert,khan,zimmerman,nobre2021effects,dupont,khan2,woodward2024instrumental,wu2025spatial} for a spectrum of contributions and opinions on this topic. Yet, a more fundamental question has remained unanswered -- {\em if the entire processes $Y(s)$ and $X(s)$ were observed in a spatial domain $\mathcal D$, under what conditions can we consistently estimate the slope
	in the presence of an unmeasured confounder process?} 
More formally, if we observe $X(s)$ and $Y(s) = X(s)\beta + W(s)$ for all $s$ in some fixed domain $\mathcal D \in \mathbb R^d$ but $W(s)$ is not observed and is possibly correlated with $X(s)$, when does there exist a consistent estimator of $\beta$? The answer to this question should provide an upper bound to the set of scenarios under which we can expect to estimate $\beta$ accurately using finite data and some analysis strategy. 

Much that is known about this problem has come from exact expressions of biases of estimators and empirical studies of these expressions
\cite{paciorek,khan2}.
These studies have broadly concluded that, when the exposure $X$ is rougher than the confounder $W,$ the slope $\beta$ can be well estimated by common spatial models or estimators, e.g., Gaussian process regression or generalized least squares (GLS). The few theoretical studies on consistent estimation of $\beta$ \citep{wang2020prediction,yu2022parametric,bolin2025spatial} have been mostly estimator-specific, e.g., studying consistency of the GLS estimator, and considering $X$ to be either a non-stochastic (fixed) function or a stochastic process that is uncorrelated with the error process $W$.
Under either assumption, the conditional mean $E(W \given X)$ is the same as its unconditional mean $E(W)$. In this scenario, even the unadjusted ordinary least squares (OLS) estimator of $\beta$, obtained from regressing $Y$ on $X$, is unbiased, precluding what is typically characterized as confounding in causal inference, where omission of the confounder in linear regression leads to bias in the unadjusted OLS estimator. However, these studies address other important challenges to identification and estimation of $\beta$, like the role of smoothness of $X$ and the impact of covariate misspecification. Section \ref{sec:setup} provides a more detailed review of this thread of literature. \citep{dupont,gilbert2025consistency} have established consistency of certain estimators for $\beta$ assuming presence of noise (non-spatial variation) in $X$. These results are not applicable when $X$ is a smooth spatial process without noise.

In this manuscript, we provide necessary and sufficient conditions for consistent estimation of $\beta$ when both the exposure $X$ and outcome $Y=X\beta+W$ are {\em Gaussian random fields (GRFs)}, the common stochastic model for spatial processes, with $W$ being an unmeasured GRF correlated with $X$. We consider infill asymptotics, i.e., the spatial domain remains fixed. Consistent estimation of parameters of GRFs is challenging under infill asymptotics, as an increase in data density within a fixed domain may not lead to an increase in information about parameters. Notable work on this topic include  \cite{stein1999interpolation,zhang2004inconsistent,anderes2010consistent,tang2021identifiability,li2023inference}. 

We first establish general sufficient conditions for consistent infill-domain estimability of the slope $\beta$ between two GRFs under spatial confounding. We show that $\beta$ can be characterized as the ratio of the principal irregular terms between the covariance functions of the outcome and exposure process. The principal irregular term dictates the local (near-zero-distances) behavior of the process. Crudely, a (cross-)covariance function having a principal irregular term of exponent $2\nu$ is $2m$ times differentiable) if $\nu > m$ \citep{stein1999interpolation}. Hence, we directly use the half-exponent $\nu$ to quantify the smoothness of a covariance function. 
We show that $\beta$ can be consistently estimated as long as the exposure process $X$ is less than $d/2$ degrees smoother
than the confounder process $W$, and the cross-covariance function between $X$ and $W$ is smoother than the covariance function of $X$. We
directly provide a consistent estimate of $\beta$ via {\em local differencing} --- using discrete differences ($d=1$) or Laplacians ($d>1$) of $Y$ and $X$ of suitable order (determined by the smoothness of $X$). The result dispels the common perception that the exposure $X$ needs to be rougher than the confounder $W$ to identify $\beta$, as we show that it can be up to $d/2$ degrees smoother. 

We then establish necessary conditions for consistent estimability of $\beta$, violation of which would lead to equivalence of measures on the paths of the bivariate $(X,Y)$ process for two different values of $\beta$. Using functional analysis results in matrix-valued  Paley-Wiener spaces, we provide a simple spectral necessary condition based on the relative rates of polynomial tail decay of the spectral densities of $X$ and $W$. In the process, we establish a novel and easy-to-verify spectral condition for equivalence of multivariate GRFs where the univariate component fields have varying smoothnesses. This is an advancement over the limited existing results for equivalence of measures on the paths of multivariate GRFs, which either require all components to have the same smoothness \citep{bachoc2022asymptotically} or are generally difficult to verify for common processes \citep{ruiz2015equivalence}. The result is thus of independent importance for studying consistent estimability of parameters for multivariate GRFs. 

Conditions on the principal irregular terms of covariance functions, which characterize our sufficient conditions, are intimately related to the tail behavior of spectral densities, which characterizes our necessary conditions. They both inform local behavior of the processes, and the equivalence of these is often established via Abelian and Tauberian theorems. We show that for several common classes of covariance functions, including Matérn, power-exponential, generalized Cauchy, and coregionalization families, our sufficient and necessary conditions are indeed identical (except at a boundary point), thus making our conditions sharp. This leads to complete characterization of estimability regions of $\beta$ as a function of the smoothnesses of the exposure and confounder for these classes of GRFs (see, e.g., Figure \ref{fig:regions}). 

We present multiple theoretical extensions of the work. We show that the results on consistent estimability remain unchanged if the outcome and the exposure are observed with measurement error; however, a different estimator will be required, which needs to first average data locally before taking local differences or Laplacians. We extend the results to multivariate slopes. We show that the differencing-based estimators remain consistent even for some non-stationary or non-Gaussian processes. We develop a spacing-weighted differencing or Laplacian-based estimator for irregular designs and show its consistency under the same conditions. We also show that, in settings where the differencing-based estimator is consistent, the popular generalized least squares (GLS) estimator can be inconsistent even with arbitrarily small misspecification of the working covariance matrix, thereby demonstrating the utility of our proposed differencing-based approach for robust estimation under model misspecifications. 
Finally, we conduct a suite of numerical experiments that explore the finite sample behavior of our proposed estimators. We conclude with a discussion on how our proposed estimation strategies can be adapted in practice
to estimate $\beta$.

\section{Setup and related work}\label{sec:setup}
We consider a pair of processes $(X,Y)=\{(X(s),Y(s)): s \in \mathcal D\}$ observed on a
domain 
$\mathcal D \subset \mathbb R^d$,
and related via a structural linear model 
\begin{equation}\label{eq:dgpy}
	Y(s) = X(s)\beta + W(s) \quad \forall s \in \mathcal D. 
\end{equation}
where $W(s)$ is an unobserved spatial process.
This spatial regression model and the estimation of $\beta$ have been the focus of much of the aforementioned spatial confounding literature. We exclude an intercept in (\ref{eq:dgpy}) without loss of generality here, as our estimators will be based on differencing and their distributional properties will be invariant to the intercept. The slope $\beta$
has the usual interpretation of the regression coefficient in multiple linear regression, i.e., $\beta$ is the expected change in $Y$ when $X$ is increased by one unit, keeping $W$ fixed. If $W$ was observed, the OLS estimator of $\beta$ obtained by regressing $Y$ on $(X,W)$ would be unbiased.

When $W$ is not observed and is correlated with $X$, regressing $Y$ on $X$ will lead to a biased estimator of $\beta$ \citep[see, e.g.,][for a formal result]{gilbert2025consistency}. We refer to this scenario of an unobserved $W$ correlated with $X$ as spatial confounding --- aligning with the standard causal inference terminology, as this induces omitted variable bias. We note that the term spatial confounding has been used more broadly to describe a collection of important identification and estimation issues related to $\beta$. Our setup of spatial confounding has been referred to as {\em data generation confounding} in \cite[][]{khan2}, who distinguished it from {\em analysis model confounding}, which refers to biases of estimators of $\beta$ caused by empirical correlation in $X$ and $W$ even if they are not correlated in the true data generation process. \cite{bolin2025spatial} 
study `spatial self-confounding', bias introduced when a smoothness-related misspecification occurs when observing $X$.

Studies on infill asymptotic properties of estimators of $\beta$ in (\ref{eq:dgpy}) are limited and have used various assumptions about the true $X$ and $W$. \cite{dupont,gilbert2025consistency} establish consistency of, respectively, the {\em spatial+} and GLS estimators of $\beta$
as long as the exposure has some non-spatial iid noise.
These results do not apply to a spatially smooth $X$.
\cite{wang2020prediction} shows that the GLS estimator $\hat \beta_{GLS}$ of $\beta$ is inconsistent when $X$ is a smooth fixed function of space in the reproducing kernel Hilbert space (RKHS) of the covariance kernel of $W$.
In \cite{bolin2025spatial},  
$\calS X$, a smoothed version of the true $X$, is assumed to be observed. They show that the degree of misspecification of the smoothness of $X$ dictates the consistency of the GLS estimator. For the case $\calS X = X$, i.e., when there is no covariate misspecification, their theory implies that $\hat\beta_{GLS}$ is consistent when $X$ is not in the RKHS of the covariance kernel of $W$, complementing the result of \cite{wang2020prediction}.
Both these studies primarily considered $X$ to be a fixed smooth function rather than a stochastic process. \cite{yu2022parametric} considers the case where $X$ is stochastic (GRF) but independent of $W$ and establishes conditions for consistently estimating $\beta$.
As we discussed earlier, when $X$ is a fixed function or is a GRF independent of $W$, there is no omitted variable bias, and even the unadjusted ordinary least squares (OLS) estimator regressing $Y$ on $X$ is unbiased (although no estimator may be consistent). Also, \cite{wang2020prediction,yu2022parametric,bolin2025spatial} study the GLS estimator using the true covariance of $W$. When $X$ is fixed or independent of $W$, this is the
maximum likelihood estimate (MLE) of $\beta$.
When $X$ is a GRF correlated with $W$, this GLS is neither the MLE nor unbiased, as 
$E(W | X) \neq E(W)$ and $Cov(W | X) \neq Cov(W)$. So, the lack of consistency of this GLS estimator
need not imply that other estimators cannot be consistent. A more fundamental notion is {\em lack of consistent estimability,} which implies that there can be no consistent estimator of $\beta$ even if the processes $(X,Y)$ are observed on all of $\calD$.

Our contribution is to characterize necessary and sufficient conditions for consistent estimability of $\beta$ when $X$ is a purely spatial stochastic process with no added noise, and is explicitly correlated with the confounder process $W$.
To our knowledge, there is no theoretical literature for this setting, although it has been extensively studied empirically \citep{paciorek,khan2}. These studies have largely concluded that when the exposure is a rougher spatial process than the confounder, $\beta$ can be estimated accurately.
Our results show that this condition is sufficient but need not be necessary, as $\beta$ can be identified and consistently estimated
even if the exposure $X$ is up to a certain degree smoother than the confounder $W$.

\section{Sufficient conditions for consistent estimability}
\label{sec:suff}

\subsection{Notation}\label{sec:notn} We denote the set of integers, natural, real, and complex numbers by $\mathbb Z, \mathbb N, \mathbb R$ and $\mathbb C$ respectively. For two sequences of real numbers $a_n$, we say $a_n = o(b_n)$ if $a_n/b_n \to 0$, $a_n = O(b_n)$ if $|a_n/b_n|$ is uniformly bounded, and $a_n \asymp b_n$ if $0 < \liminf |a_n/b_n| \leq \limsup |a_n/b_n| < \infty$. If $X_n$ denote a sequence of random variables then $X_n=o_p(a_n)$ if $X_n/a_n \to 0$ in probability, and $X_n=O_p(a_n)$ if $X_n/a_n$ is bounded in probability. For two measures $\calP_0$ and $\calP_1$ we say $\calP_0 \equiv \calP_1$ if they are equivalent (i.e., $\calP_0$ is absolutely continuous with respect to $\calP_1$ and vice versa), and  $\calP_0 \perp \calP_1$ if they are orthogonal, i.e., there exists a measurable set  $A$ such that $\mathcal P_0(A)=0$ and $\mathcal P_1(A)=1$.

\subsection{Consistent estimability in $\mathbb R$}\label{sec:r1}
We first consider the spatial domain to be in $\mathbb R$ to elaborate on the main ideas that lead to sufficient conditions for consistent estimability of $\beta$.
We present a general result on the consistent estimability of a ratio of the coefficients of the principal irregular terms of bivariate stationary GRFs.
We will then show how this general result provides sufficient conditions for consistent estimability of $\beta$ under spatial confounding. We first define the class of (cross-)covariance functions we consider. 

\begin{asm}\label{eq:K.assump} The (cross-) covariance function $K$ on $\mathbb R$ can be expressed as $K(t)=A(t) + B(t)$ where $A$ is an even analytic function and $B$
	is even, continuous on $[-L,L]$ for some $L>0$.
	The $k^{th}$ derivative $B^{(k)}$ on $(0,L]$
	satisfies, 
	as $t\downarrow 0$, 
	$B^{(k)}(t) = (\alpha)_{k}\,c\,t^{\alpha-k} + o(t^{\alpha-k})$, for a constant $c$ (non-zero when $K$ is a covariance function) and a positive exponent $\alpha$ that is not an even integer, where $B^{(0)}=B$ and $(\alpha)_k = \Gamma(\alpha+1)/\Gamma(\alpha+1 -k)$ is the falling factorial.
\end{asm}

Many common covariance functions satisfy Assumption \ref{eq:K.assump}, as we discuss in Section \ref{sec:examples}. Under Assumption \ref{eq:K.assump}, $B$ denotes the irregular part of the covariance $K$ with $ct^\alpha$ being its leading (least smooth) term, which is referred to as the {\em principal irregular term} \citep{stein1999interpolation}. 
We can extend this notion to $\alpha$ being an even integer by replacing $t^{\alpha}$ in Assumption \ref{eq:K.assump} by
$S(t)t^\alpha$ for some function $S$ slowly varying at 0 as long as $S$ is not differentiable at 0 if $\alpha$ is an even integer.
The case $t^{2m}\log t$ for $m$ a positive integer is of greatest practical interest, because it covers Mat\'ern models with integer-valued smoothness parameters. One example of how the theory extends to such covariances is provided later in Corollary \ref{cor:boundary}.

Suppose $(Z_1,Z_2)$ is a bivariate zero-mean stationary GRF on $\mathbb{R}$. Let $K_{k\ell}$ be the (cross-) covariance function of $Z_k$ and $Z_\ell$, where each $K_{k\ell}$ satisfies Assumption \ref{eq:K.assump} with some
$c_{k\ell}$, and $\alpha_{k\ell} > 0$. 
Further, let us assume that
\begin{equation}\label{eq:K12limit}
	\begin{aligned}
		\alpha_{12} &= \alpha_{11} \mbox{ and } c_{12} = \beta c_{11},
	\end{aligned}
\end{equation}
We first present a result for generic $Z_1$ and $Z_2$ satisfying (\ref{eq:K12limit}). In the context of spatial confounding, as we will show later, $Z_1$ will be $X$, $Z_2$ will be $Y = X\beta + W$, and (\ref{eq:K12limit}) will be satisfied as long as the cross-covariance between $X$ and $W$ is smoother
relative to the covariance of $X$ near the origin.   
Under (\ref{eq:K12limit}), the ratio of the principal irregular terms of $K_{12}$ and $K_{11}$ becomes $\frac{c_{12}h^{\alpha_{12}}}{c_{11}h^{\alpha_{11}}} = \frac{\beta c_{11}h^{\alpha_{11}}}{c_{11}h^{\alpha_{11}}} = \beta$. To give an intuition on how $\beta$ can be estimated leveraging this, assume for simplicity that $\alpha_{1j}$ is less than $2$, for $j=1,2$. 
Define $\gamma_{1j}^{(1)}(h)=\frac 12 \mathrm{Cov}(Z_1(s+h) - Z_1(s),Z_j(s+h) - Z_j(s))$ as the first order semi-(cross-)-variogram between $Z_1$ and $Z_j$ at distance $h$.
Then under Assumption \ref{eq:K.assump}, as $h\to 0$, we have $\gamma^{(1)}_{1j}(h) = K_{1j}(0) - K_{1j}(h) = A_{1j}(h) - A_{1j}(0) + c_{1j}h^{\alpha_{11}} + o(h^{\alpha_{11}}) = c_{1j}h^{\alpha_{11}} + o(h^{\alpha_{11}})$ as $A$ is analytic and $\alpha_{11} = \alpha_{12} < 2$. 
Thus the limiting ratio of the semi-cross-variogram between $Z_1$ and $Z_2$ and the semi-variogram of $Z_1$ at small distances is
\begin{equation}\label{eq:betaratio}
	\lim_{h \to 0} \frac{\gamma^{(1)}_{12}(h)}{\gamma^{(1)}_{11}(h)} = \beta.
\end{equation}

A similar result will hold for larger $\alpha_{11}$ using higher order semi-variograms defined by recursively taking process differences. This motivates our strategy to develop a consistent estimate of $\beta$ based on local differences of the $Z_1$ and $Z_2$ processes of suitable order.

For a stationary GRF $Z$ on $\mathbb R$ with covariance $K(h)=\mbox{Cov}(Z(s+h),Z(s))$, we define its
(scaled) first-order difference process or discrete gradient for a small distance $h \in \mathbb R$
\begin{equation}\label{eq:deriv}
	\nabla^{(1)}_h Z(s) = \frac 1h \left(
	Z(s+h) - Z(s)\right). 
\end{equation} 
Higher order differences are defined recursively, e.g., $\nabla^{(i)}_h Z(s) = \nabla^{(1)}_h (\nabla^{(i-1)}_h Z(s))$, with the convention that $\nabla_h^{(0)} Z= Z$ for all $h$.

Consider observations of a bivariate stationary GRF $(Z_1,Z_2)$ on $\mathbb R$ at $n+1$ equally spaced locations in between $0$ and $L > 0$. Write $h$ for the distance $L/n$ between neighboring observations.
For an integer $p \geq 0$, define the OLS estimator between the $p^{th}$ order spatial-first differences of $Z_1$ and $Z_2$ on this lattice as
\begin{equation}\label{eq:olsdiff}
	\mathrm{OLS}_n^{(p)}(Z_2,Z_1) = \frac{\sum_{j=0}^{n-p}\nabla_h^{(p)}Z_1(h j)\nabla_h^{(p)}Z_2(h j)}
	{\sum_{j=0}^{n-p}\{\nabla_h^{(p)}Z_1(h j)\}^2}.
\end{equation}
For $p=0$, this is simply the OLS  estimator regressing $Z_2$ on $Z_1$. For $p > 0$, this is the OLS estimator regressing the $p^{th}$ order spatial differences of $Z_2$ on those of $Z_1$. 
We now state a result on consistent estimability of $\beta$ for bivariate stationary GRF on $\mathbb R$ with covariances satisfying Assumption \ref{eq:K.assump} and Equation (\ref{eq:K12limit}).

\begin{theorem}\label{th:beta.con}
	Let $(Z_1,Z_2)$ be a bivariate
	stationary GRF on $\mathbb R$
	with covariance function $K=(K_{k\ell})_{\{1 \leq k,l \leq 2\}}$.
	Assume that each $K_{k\ell}$ satisfy Assumption \ref{eq:K.assump} for some $L > 0$, and with parameters $c_{k\ell}$ and $\alpha_{k\ell}$ which satisfy Equation (\ref{eq:K12limit}) for some $\beta$. 
	Additionally, assume $\alpha_{11}-1 < \alpha_{22} \le \alpha_{11}$.
	Then the measures on the paths of the bivariate random fields $(Z_1,Z_2)$ on $[0,L]$ are orthogonal for different values of $\beta$. In particular, if $p$ is an integer such that $\alpha_{11} < 2p$, then   
	\begin{equation}\label{eq:th1.main}
		\hat\beta_n = \mathrm{OLS}_n^{(p)}(Z_{2},Z_{1}) 
		\to \beta \mbox{ in probability as } n\to\infty.
	\end{equation}
\end{theorem}

All proofs are in the Supplementary Materials file. This result on orthogonality for the measures on the paths of the bivariate process  $(Z_1,Z_2)$ for different values of $\beta$ is, as far as we know, new. In fact, we give an explicit estimator (\ref{eq:th1.main}) of $\beta$ based on local differencing of sufficient order. 
The order of differencing $p$ is dictated by the smoothness of $Z_1$, with $p$ being an integer greater than $\alpha_{11}/2$ sufficing. Larger $p$ may lead to a less efficient estimator due to over-differencing, but will not ruin consistency. 

To see why Theorem \ref{th:beta.con} is relevant to spatial confounding, let $Y$ be generated as (\ref{eq:dgpy}) where $(X,W)$ has a bivariate covariance function matrix $K=(K_{k\ell})_{\{1 \leq k,l \leq 2\}}$ satisfying Assumption \ref{eq:K.assump} with exponents $(\alpha_{k\ell})$ for the principal irregular terms and constants $(c_{k\ell})$. Let $Z_1=X$,  $Z_2=Y$, and $K^*=(K^*_{k\ell})$ denote the covariance function of $Z=(Z_1,Z_2)$. Then each $K^*_{k\ell}$ satisfies Assumption \ref{eq:K.assump} with exponents $\alpha^*_{k\ell}$ and constants $c^*_{k\ell}$. As $Z_1=X$, we have $\alpha^*_{11}=\alpha_{11}$ and $c^*_{11}=c_{11}$. To find these terms for the cross-covariance, note that $ K^*_{12}(h)=Cov(Y(s+h),X(s)) = \beta Cov(X(s+h),X(s)) + Cov(W(s+h),X(s))$ satisfies Assumption \ref{eq:K.assump} with the irregular (non-analytic) term being $\beta c_{11} h^{\alpha_{11}} + c_{12} h^{\alpha_{12}} +$ higher order terms. If the cross-covariance between $X$ and $W$ is smoother than the covariance of $X$, i.e., $\alpha_{12} > \alpha_{11}$ and $\beta\neq 0$, then the leading irregular term is $\beta c_{11} h^{\alpha_{11}}$, yielding $\alpha^*_{12}=\alpha_{11}$ and $c^*_{12}=\beta c_{11}$. So, (\ref{eq:K12limit}) is satisfied, and an OLS of differences of $Y$ on differences of $X$ of suitable order should be consistent for $\beta$ by Theorem \ref{th:beta.con}. 
Our next result proves this formally. 

\begin{theorem}\label{th:suff}
	Consider
	$Y$ generated as in (\ref{eq:dgpy}) where $(X,W)$ is a bivariate stationary GRF on $\mathbb R$ with
	covariance function matrix $K=(K_{k\ell})_{\{1 \leq k,l \leq 2\}}$ such that $K_{11}$ and $K_{22}$ satisfy Assumption \ref{eq:K.assump} for some $L > 0$ and with $ \alpha_{11} < \alpha_{22} + 1$, and either $X$ is independent of $W$ or $K_{12}$ satisfies Assumption \ref{eq:K.assump} with $\alpha_{12} > \alpha_{11}$. 
	Then the regression slope $\beta$
	is consistently estimable on the paths of $(X,Y)$ with $\mathrm{OLS}_n^{(p)}(X,Y)$ on $[0,L]$ being a consistent estimator of $\beta$, where $p$ is an integer such that $\alpha_{11} < 2p$.
\end{theorem}

The result shows that for data generated according to (\ref{eq:dgpy}), even under spatial confounding, i.e., there being an unmeasured spatial process $W$ that is a confounder, one can identify the regression coefficient of $Y$ on $X$, as long as two conditions hold on the relative magnitudes of the exponents of the principal irregular terms:  $\alpha_{11} < \alpha_{12}$ and $ \alpha_{11} < \alpha_{22} + 1$. These exponents dictate the behavior of the process at near-zero distances and are closely related to smoothnesses of processes, with higher $\alpha_{k\ell}$ implying more smoothness \citep[page 29 of][provides a general result connecting the exponent of the principal irregular term to the degree of differentiability of the covariance function at zero]{stein1999interpolation}.
The first assumption in Theorem \ref{th:suff} is that the cross-covariance between $X$ and $W$ is smoother than the covariance of $X$, specified as $\alpha_{11} < \alpha_{12}$. For many covariance classes, this
is equivalent to assuming that the cross-spectral density $f_{XW}$ of $X$ and $W$ decays faster than the spectral density $f_{XX}$ of $X$ at high frequencies (see Section \ref{sec:examples} for examples). In other words, the ratio $f_{XW}(\omega)/f_{XX}(\omega)$,
goes to zero as $|\omega|\to\infty$.
This condition
has been used to develop spectral methods to adjust for spatial confounding \citep{guan}. 
However, our result shows that this itself may not be adequate. In addition, a second assumption is utilized, which is about the relative marginal smoothnesses of $X$ and $W$, i.e., $\alpha_{11} < \alpha_{22} + 1$, which mandates that the
exposure $X$ cannot be too much smoother than the confounder $W$. 

It has long been conjectured that the regression slope $\beta$ of $Y$ on $X$ is consistently estimable under spatial confounding if the unmeasured confounder $W$ is smoother than the exposure $X$. For example, \cite{paciorek} and \cite{khan2} provide extensive empirical evidence in favor of this. They show that the GLS estimator usually has lower bias than the OLS estimator (both based on $Y$ and $X$) when $X$ is rougher than $W$.
\cite{guan} presented examples of estimating $\beta$ when  $X$ is smoother than $W$, assuming specific parametric forms of the cross-covariance between $X$ and $W$.
Theorem \ref{th:suff} shows that consistent estimability holds under more lenient assumptions, not only when the exposure is rougher than the confounder $(\alpha_{11} < \alpha_{22})$ but even when it is somewhat smoother $(\alpha_{22} \leq  \alpha_{11} < \alpha_{22} + 1$).

Under these two conditions for estimability ($\alpha_{11} < \alpha_{12}$ and $ \alpha_{11} < \alpha_{22} + 1$), we provide an explicit consistent estimator of $\beta$ in Theorem \ref{th:suff}. The consistency result does not assume any specific parametric family of covariances beyond the general form specified in Assumption \ref{eq:K.assump}, which, as we show in Section \ref{sec:examples}, is satisfied for several common covariance families. The explicit consistent estimator we provide is non-parametric: it regresses local differences of $Y$ on local differences of $X$. It does not involve or require estimating any covariance parameters such as variances or ranges, which are often not consistently estimable, e.g.,\ in Mat\'ern covariance families \citep{zhang2004inconsistent}.

We note that local differencing of variables has been used before to mitigate spatial confounding \citep{druckenmiller2018accounting}. The rationale behind this approach was that when the confounder varies at a larger scale than the exposure, local differencing largely cancels out the confounder but retains the high-frequency variations of the exposure, which is enough to identify $\beta$. Our result shows that local differencing is more powerful, leading to consistent estimates of $\beta$ even when the exposure is somewhat smoother than the confounder, although first differences may not always suffice and the order of differencing needs to be based on the smoothness of $X$. 

\subsection{Consistent estimability in higher dimensional spatial domains}\label{sec:rd}

We extend the results on consistent estimability of $\beta$ in Theorems \ref{th:beta.con} and \ref{th:suff} to GRFs on any spatial domain
$\calD \in \mathbb R^d$ that contains an open $d$-dimensional ball. As such a domain always contains a 1-dimensional interval $\calI$, as an immediate corollary of Theorem \ref{th:suff}, we have that $\beta$ will be consistently estimable for a $(X,Y)$ process on $\calD$ whose restriction to $\calI$ satisfies the conditions of that theorem. However, we will now show that in $\mathbb R^d$, for $d > 1$, $\beta$ can be consistently estimable under weaker conditions, and that the region of consistent estimability of $\beta$, as a function of the smoothnesses of $X$ and $W$, expands with the dimension of the spatial domain. 

The consistent estimators in Section \ref{sec:r1} relied on taking differences of suitable order on a regular grid along a straight line. The order of differencing $p$ can be even or odd as long as $p > \alpha_{11}/2$. For a stationary process on a regular grid in $\mathbb R^d$, it is more natural to consider discrete Laplacians of the process of suitable order.
Crudely, discrete Laplacians can be thought of as even order differencing, 
which fully leverage the availability of data along all the $d$ directions in $\mathbb R^d$.
This is central for consistency in $\mathbb R^d$ under weaker assumptions than in $\mathbb R$. 

Let $\calG_n$ denote a $(n+1) \times (n+1) \times \ldots (n+1)$
regular grid in $[0,L]^d \in \mathbb R^d$. Then $\calG_n$ consists of $(n+1)^d$ points, and the length of each side of the hypercubic grid cell is $h=L/n$. Let
$\calG_n^{(1)}$ denote the interior $(n-1)\times(n-1)$ grid created by peeling off one layer of the outer points of $\calG_n$. Define, higher order interiors recursively as $\calG_n^{(m)}=(\calG_n^{(m-1)})^{(1)}$.

For any GRF $Z$ on $\calG_n$, define the first order Laplacian $\Delta_h (Z)$ at a location $s_j \in \calG_n^{(1)}$ as
$$\Delta_h (Z(s_i)) = \frac 1{h^2} \sum_{g = 1}^d (Z(s_i + he_g) + Z(s_i - he_g) - 2Z(s_i) )$$ where $e_g$ is the $g^{th}$ column of a $d \times d$ identity matrix. 
Define higher order Laplacians recursively as $\Delta_h^{(m)}(Z)=\Delta_h\left(\Delta_h^{(m-1)}(Z)\right)$ where $\Delta_h^{(1)}=\Delta_h$.

For a bivariate zero-mean stationary GRF $(Z_1,Z_2)$ on the grid $\calG_n$, let
$Z_i^{(m)}$ denote $\Delta_h^{(m)}(Z_i)$ restricted to $\calG_n^{(m)}$, for $i=1,2$, and define the OLS estimator on Laplacians as  
\begin{equation}\label{eq:lap}
	\mbox{LAP}_n^{(m)}(Z_1,Z_2) = 
	OLS_n(Z_1^{(m)},Z_2^{(m)})= \frac{Z^{(m)T}_1Z^{(m)}_2}{ Z^{(m)T}_1Z^{(m)}_1}.
\end{equation}
On $\mathbb R$, $\mbox{LAP}_n^{(m)}(Z_2,Z_1) = \mbox{OLS}_n^{(2m)}(Z_2,Z_1)$, i.e., the $m^{th}$-order Laplacian corresponds to $2m^{th}$ order differencing. So on $\mathbb R^d$ with $d >1$,
the estimator in (\ref{eq:lap}) corresponds to only even-order differencing. 

We specify the following regularity condition for the covariance function and its derivatives in $\mathbb R^d$, generalizing Assumption \ref{eq:K.assump}  to covariance functions in $\mathbb R^d$. 

\begin{asm}\label{as:rd} Let
	$C(u)=(C_{k\ell}(u))_{1 \leq k,l \leq 2}$ be a $2 \times 2$ matrix-valued 
	stationary covariance function on $\mathbb R^d$ with symmetric cross-covariance function, i.e., $C_{12}=C_{21}$, where $C_{k\ell}(u)=K_{k\ell}(\|u\|)+ r_{k\ell}(u)$
	for $1\leq k,\ell\leq2$.
	Here $K_{k\ell}(h)$ is a function that can be extended to be supported on $\mathbb R$ by assuming it is even and  $K_{k\ell}(h)$ satisfies Assumption \ref{eq:K.assump} on $[0,L]$ for some $L>0$, and with constants $c_{k\ell}$ and $\alpha_{k\ell}$.
	The remainder term $r_{kl}(u)$ is such that  $r_{k\ell}(u) = o(\|u\|^{\alpha_{k\ell}})$, and for any $u \neq 0$, the $p^{th}$ order mixed partial derivatives of $r_{k\ell}$ at $u$ are all $o(\|u\|)^{\alpha_{k\ell}-p}$. 
\end{asm}

Under Assumption \ref{as:rd}, we have the following result on consistent estimation of $\beta$
in $\mathbb R^d$. 

\begin{theorem}\label{th:suff.rd}
	Let $Y$ be generated as in (\ref{eq:dgpy}) where $(X,W)$ is a bivariate stationary GRF on $\mathbb R^d$ with
	covariance function $C=(C_{k\ell})_{\{1 \leq k,l \leq 2\}}$
	satisfying Assumption \ref{as:rd}
	with $\alpha_{11} < \alpha_{12}$ and $ \alpha_{11} < \alpha_{22} + d$. Then $\beta$ is consistently estimable on the paths of $(X,Y)$ in $[0,L]^d$ and $\mathrm{LAP}_n^{(m)}(X,Y) \to \beta$ in probability, where $m$ is an integer such that $\alpha_{11} < 4m$.
\end{theorem}

Theorem \ref{th:suff.rd} generalizes the results of Theorem \ref{th:suff} from processes on a line to processes in Euclidean domains of any dimension. 
As the exponent $\alpha_{k\ell}$ often
equals twice the smoothness of the process for many parametric covariance families (see Section \ref{sec:examples} for examples), the sufficient condition $\alpha_{11} < \alpha_{22} + d$ implies that the exposure $X$ is allowed to be up to $d/2$ degrees smoother than the confounder $W$ for estimability of $\beta$ to hold, implying  that the region of consistent estimability
increases with increasing dimension (see Figure \ref{fig:regions}).
We will also show in the next Section that this gap of $d/2$ in the smoothnesses is not only sufficient but also necessary, thereby providing sharpness to our results. We discuss this apparent `blessing of dimensionality' in Section \ref{sec:matern}. 

To our knowledge, Theorem \ref{th:suff.rd} gives the broadest conditions under which one can guarantee consistent estimability of the slope under spatial confounding while considering stochasticity of the exposure process $X$ and without assuming any specific parametric form of the covariance functions of $X$ and $W$. We also provide a non-parametric estimator of $\beta$ that
simply uses discrete Laplacians of the observed processes, and does not require knowledge or estimation of the covariance parameters.

\section{Necessary conditions for consistent estimability}\label{sec:nec}

\subsection{Background on Fourier analysis and Paley-Wiener spaces}\label{sec:notation}

We will now establish necessary conditions for consistent estimability of $\beta$ in model (\ref{eq:dgpy}) based on observations of the $(X,Y)$ process.
We will obtain conditions under which different values of $\beta$ lead to equivalent measures on the paths of the bivariate GRF $(X,Y)$ for some specified domain.
In this case, there can be no consistent estimator for $\beta$ even if the $(X,Y)$ field were observed everywhere on this domain. There are existing results on the equivalence and orthogonality of measures on $Y$ for different
values of $\beta$. 
For example, when $X$ is a fixed function of space,
\cite{bolin2025spatial} studied the consistency of the correctly specified GLS estimator based
on the true covariance of $W$ using the Feldman--H\'ajek theorem \cite{da2014stochastic}.
This case essentially reduces to comparing two univariate Gaussian measures with
the same covariance but different means, and a necessary and sufficient condition is that the difference of means lies in the reproducing kernel
Hilbert space (RKHS) of $\mathrm{Cov}(W)$. Many applications involve the SPDE
characterization of the Mat\'ern process, either on a Riemannian manifold or
on a bounded domain with homogeneous Neumann or Dirichlet boundary conditions, where
the RKHS has a well-known characterization
as a Sobolev space, and the eigenvalues grow according to Weyl's law. However, this does not
cover the stationary Mat\'ern GRF on a compact bounded domain obtained
by restricting the stationary Mat\'ern field on $\mathbb{R}^d$. Similar
equivalence and orthogonality results are used to study the correctly specified
GLS estimator when $X$ is independent of $W$ in \cite{yu2022parametric}, in which case, conditioning on a realization of $X$ does not change the law of $W$.

When $X$ is a GRF, $\beta$ features in the covariance function of $(X,Y)$, and consistent estimability of $\beta$ thus relates to equivalence/orthogonality results for two Gaussian measures with different covariance functions.
Existing sufficient conditions to establish this are either challenging to verify \citep[e.g., the condition (iii) of the Feldman-Hajek Theorem in][]{bolin2025spatial}, or are only provided for the univariate case \citep[e.g., the easy-to-verify spectral conditions of][]{ibragimov2012gaussian,skorokhod1973absolute,stein1999interpolation}. When $X$ is a GRF correlated with $W$, and we observe $(Y,X)$, we have to use the equivalence result on the path of this bivariate GRF (or more generally, multivariate GRF, if $X$ has more than one covariate). 
There are relatively fewer available results for the equivalence of multivariate Gaussian random fields. \cite{ruiz2015equivalence} provides some conditions that are generally challenging to verify for common multivariate covariance families. \cite{bachoc2022asymptotically} provides sufficient conditions on the equivalence of multivariate Gaussian random fields that are easier to verify but assume the univariate components of the multivariate fields all have the same smoothness or rate of decay of the spectral density at high frequencies. As we saw in Section \ref{sec:suff}, consistent estimability of $\beta$ in (\ref{eq:dgpy})
is fundamentally tied to the relative differences in smoothness between the exposure $X$ and the confounder $W$. 
Hence, these existing results do not apply to this setting. 

Our first theoretical contribution in this section is to provide a novel and simple sufficient spectral condition for the equivalence of two multivariate GRFs
that accommodates differing smoothnesses or spectral tail behaviors of the component fields. This will be central to its application in establishing necessary conditions for consistent estimability of $\beta$.

We first present some notation and background.
Let $C=(C_{ij})$ denote a $p \times p$ matrix-valued stationary covariance function on a bounded domain $\calD \in \mathbb R^d$. In this section, we take the domain to be $\calD=[-T,T]^d$ for some $T > 0$, i.e., symmetric about 0. This is convenient here for the Fourier analysis used for the theory and, because of stationarity, equivalence on $[-T,T]^d$ implies equivalence on $[0,2T]^d$.

We consider covariance function matrices $C$ for which there exists a spectral density matrix $F=(F_{ij})$,
i.e., 
\begin{equation}\label{eq:bochner}
	C_{ij}(h) = \int_{\mathbb R^d} \exp(\iota h^T \omega) F_{ij}(\omega) d\omega,
\end{equation}
and all the spectral densities
are integrable. Here $\iota$ denotes the complex square root of $-1$. 

For a function $g: \mathbb R^d \to \mathbb R$, let $\calF(g)$ denote its Fourier transform i.e., 
\begin{equation}\label{eq:fourier}
	\calF(g)(\omega) = \frac 1{(2\pi)^d} \int_{\mathbb R^d} \exp(-\iota h^T \omega)g(h)dh.
\end{equation}

Let $\calW_\calD$ denote the set of functions from $\mathbb R^d \to \mathbb C^p$ such that if $u=(u_1,\ldots,u_p)^T \in \calW_\calD$ then each $u_i$ can be expressed as a Fourier transform $\calF(g_i)$ of a square integrable function $g_i$ on $\mathbb R^d$ ($g_i \in \calL_2(\mathbb R^d)$) that vanishes outside $\calD$. We denote this as $u=\calF(g)$ where $g=(g_1,\ldots,g_p)^T$. The space $\calW_\calD$ is a multivariate Paley-Wiener space \citep{iosevich2015exponential}. By the Plancherel theorem \citep[see][]{skorokhod1973absolute}, $\intrd |u_i(\omega)|^2 d\omega < \infty$ for any such $u_i$. For a complex matrix, we use the $^*$ notation to denote its Hermitian. For a complex vector, $^*$ indicates the transposed complex conjugate.
For a matrix valued function $F=(F_{ij})$ such that $F(\omega)$ is Hermitian and positive definite for all $\omega \in \mathbb R^d$ and $\sup_\omega \|F(\omega)\| \leq M$ for some $M>0$, define $\calW_\calD(F)$ to be the closure of $\calW_\calD$ in the metric
\begin{equation}\label{eq:metric}
	\|u\|_F^2
	= \int_{\mathbb R^d} u(\omega)^*F(\omega)u(\omega)d\omega.
\end{equation}
Then $\calW_\calD(F)$ is a complex, separable Hilbert space.

Let $\calW_\calD^2$ be the space of $p\times p$ matrix-valued functions $B(\mu,\omega)=(b_{ij}(\mu,\omega))$ such that each $b_{ij}(\mu,\omega)$ can be represented as
\begin{equation}\label{eq:bij}
	b_{ij}(\mu,\omega) = \frac 1{(2\pi)^{2d}} \int_\mathbb R^d \int_\mathbb R^d \exp(-\iota a^T\mu + \iota h^T \omega)\rho_{ij}(a,h) da\, dh
\end{equation}
for some $\rho_{ij}$ in $\calL_2(\mathbb R^d \times \mathbb R^d)$ that is zero outside  $\calD \times \calD$. Then $\calW_\calD^2$ is also a Paley-Wiener space, now for matrix-valued functions. For a $p \times p$ spectral density matrix $F$ as above, define $\calW^2_\calD(F)$ to be the closure of $\calW_\calD^2$ based on the inner product
\begin{equation}\label{eq:w2}
	\langle B_1,B_2 \rangle_{2,F} = \int_\mathbb R^d \int_\mathbb R^d \mbox{trace} \left[ B_1(\mu,\omega)F(\omega)B_2(\mu,\omega)^*  F(\mu)\right] d\mu\, d\omega.
\end{equation}
Note that for any $u,v \in \calW_\calD(F)$, $B(\mu,\omega)=u(\mu)v^*(\omega) \in \calW^2_\calD(F)$ with $\|B\|_{2,F}^2=\|u\|_F^2 \|v\|^2_F$. 

\subsection{Equivalence of multivariate Gaussian random fields}\label{sec:equiv}

We now state our main result on the equivalence of measures on the paths of multivariate Gaussian random fields with component univariate fields of possibly unequal smoothnesses. 

\begin{theorem}\label{th:equiv} Let $\calP_0$ and $\calP_1$ denote two measures on the paths of a $p$-dimensional stationary zero-mean Gaussian random field on $\calD$. Let $C^{(i)}$ and $F^{(i)}$ denote their respective covariance functions and spectral densities under $\calP_i$. 
	Suppose there exists positive constants $c_1$, $c_2$, and $p$ real-valued positive 
	functions $\phi_1, \ldots, \phi_p$ with $(\phi_1, \ldots, \phi_p)^T \in \calW_\calD$ such that $\sup_{j,\omega} \phi_j(\omega) \leq M$ for some $M > 0$ and with $\Phi(\omega)=\diag\left(\phi^2_1(\omega),\ldots,\phi^2_p(\omega)\right)$ we have
	\begin{equation}\label{eq:cond1}
		c_1 \Phi(\omega) \leq F^{(i)}(\omega) \leq c_2 \Phi(\omega)\, \forall \omega \in \mathbb R^d,  i=0,1.
	\end{equation}
	Then $\calP_0 \equiv \calP_1$ if
	\begin{equation}\label{eq:equivgeneral}
		\int_{\mathbb R^d} \Big\| \Phi(\omega)^{-1/2}\left(F^{(1)}(\omega)-F^{(0)}(\omega)\right)\Phi(\omega)^{-1/2}\Big\|^2 d\omega < \infty.
	\end{equation}
\end{theorem}

Condition (\ref{eq:cond1}) states that the spectral density matrices corresponding to the two measures are uniformly bounded from below and above by a multiplier of a diagonal spectral density matrix $\Phi=\mbox{diag}(\phi_j^2)$, where each component $\phi_i$ is a Fourier transform of a square-integrable positive compactly supported function. Then each  $\phi_j$ is an entire function (holomorphic on the entire complex plane). The condition implies that the spectral densities $\Fz$ and $\Fo$, bounded by multipliers of $\Phi$ on both sides, are regular. A similar assumption has been used to derive equivalence results in the univariate setup in \cite{skorokhod1973absolute}. For the multivariate setup of \cite{bachoc2022asymptotically}, a more stringent condition was used, where $\Phi(\omega)$ was assumed to be of the form $\gamma^2(\omega) I$ for some $\gamma(\omega)$ that is a Fourier transform of an integrable compactly supported function. This restricted the scope of the sufficiency result in \cite{bachoc2022asymptotically}, ruling out even simple cases where, say, for example, $\Fz$ (or $\Fo$) is a diagonal matrix, with the two spectral densities having different tail decays. Our condition (\ref{eq:cond1}) is more general, allowing component spectral densities to have different tail decays.
From (\ref{eq:equivgeneral}), using $\Phi^{-1} \asymp F^{(0)-1}$
and the equivalence of $\ell_2$ and Frobenius norms for fixed-dimensional matrices, we can obtain the following sufficient condition 
\begin{equation}\label{eq:equivdirect}
	\intrd \mbox{trace}\left[ \left(F^{(1)}(\omega)F^{(0)}(\omega)^{-1} - I_{p \times p} \right)^2 \right] d\omega < \infty.
\end{equation}
Condition (\ref{eq:equivdirect}) is a simpler sufficient condition for the equivalence of two measures on the paths of a multivariate GRF with components of varying smoothnesses. The condition can be directly evaluated using the two spectral density matrices $\Fz$ and $\Fo$.

\subsection{Spectral necessary conditions for consistent estimability of the slope}\label{sec:converse}
We now apply Theorem \ref{th:equiv} to obtain a simple sufficient spectral condition for equivalence of measures on the paths of $(X, Y = X\beta + W)$, where $(X, W)$ is a bivariate GRF.
Violation of this condition will be necessary for consistent estimability of $\beta$. The condition, stated in the following result, is expressed in terms of the spectral densities of $X$ and $W$ and shows that when $Y=X\beta + W$ and $X$ is much smoother than $W$,
$\beta$ is not consistently estimable even if there is no confounding ($W$ and $X$ are independent). 

\begin{theorem}\label{th:nonidgen}
	Let $Y(s) = X(s) \beta + W(s)$ for $s \in \calD$, a bounded subset of $\mathbb R^d$ that contains an open $d$-dimensional ball. Let $X$ and $W$ 
	be independent stationary GRFs with spectral densities $f_X(\omega)$ and $f_W(\omega)$ which are positive, continuous, bounded away from $0$ and $\infty$
	as $\omega \to 0$, and satisfies
	\begin{equation}\label{eq:bound}
		\begin{aligned}
			c_1 \phi^2_X(\omega) \leq f_X(\omega) \leq c_2 \phi^2_X(\omega),\;
			c_1 \phi^2_W(\omega)
			\leq
			f_W(\omega)
			\leq c_2 \phi^2_W(\omega).
		\end{aligned} 
	\end{equation}
	for all $\omega$,  some universal positive constants $c_1$ and $c_2$, and real-valued positive functions $\phi_X(\omega)$ and $\phi_W(\omega)$
	which are
	Fourier transforms of functions that are in $\calL_2(\calD^*)$ for some bounded subset $\calD^*$ containing $\calD$, and are zero outside $\calD^*$.
	If $\sup_{\omega \in \mathbb R^d} f_X(\omega)/f_W(\omega) < \infty$ and
	\begin{equation}\label{eq:nec}
		\intrd \frac{f_X(\omega)}{f_W(\omega)} d\omega < \infty,
	\end{equation}
	then the measures on the paths of $(X,Y)$ are equivalent for different values of $\beta$, i.e., there can be no consistent estimator of $\beta$ even if the entire $(X,Y)$ process were observed on $\calD$.
\end{theorem}

For measures corresponding to two different values of $\beta$ on the paths of the bivariate GRF $(X,Y)$, Theorem \ref{th:nonidgen} offers a direct, simple spectral condition that implies their equivalence. It is thus necessary for the integral in (\ref{eq:nec}) to diverge for $\beta$ to be consistently estimable.
To our knowledge, this spectral condition (\ref{eq:nec}), guaranteeing the lack of estimability of the slope between two Gaussian random processes, is new. The inconsistency results for $\beta$ in  \cite{wang2020prediction} and \cite{bolin2025spatial} assume $X$ to be a fixed function and not a stochastic process, and focus only on the GLS estimator. We emphasize that our result not only implies a lack of consistency of any specific estimator, but it is a stronger impossibility result saying that there cannot exist any consistent estimator of $\beta$ when (\ref{eq:nec}) holds, even if the entire $(X,Y)$ field is observed.
The result on equivalent measures for GRFs for two different values of $\beta$ in Proposition 7.3.4 of \cite{yu2022parametric} is based on sample path properties of the $X$ process.
These can be challenging to verify
for covariance families like the generalized Cauchy or powered exponential.
Our result does not make any parametric assumptions on the covariance functions of Gaussian random fields and provides a simple necessary spectral condition for consistent estimability based only on the relative tail spectral decays of $X$ and $W$. 
Indeed, condition (\ref{eq:nec}) should be easy to verify for any pair of spectral densities as long as their tail behavior is known, as we show
in Section \ref{sec:examples}. 

We remark that our result does not cover analytic processes like a GRF with a squared exponential covariance, because (\ref{eq:bound}) is not satisfied for such processes. Such analytic stochastic processes are of less relevance in geosciences, as it is unlikely that a physical process can be perfectly predicted at a location just by knowing its values in some closed neighborhood not containing that location.

\section{Theory for common covariance families}\label{sec:examples}

The theoretical results in the previous two sections do not assume any specific covariance families. The sufficient conditions in Theorems \ref{th:suff} and \ref{th:suff.rd} are based on the behavior of the bivariate covariance function matrix of $(X,Y=X\beta +W)$ near zero distances, specifically on the exponents of the principal irregular terms. On the other hand, the necessary condition implied by (\ref{eq:nec}) for consistent estimability of $\beta$
is based on the ratio of spectral tail decays of $W$ and $X$.
Principal irregular terms and behavior of covariances near zero distances are closely related to decay rates of spectral densities at high frequencies via Abelian and Tauberian theorems \citep[see, e.g.,][for general results]{bingham1972}. In this Section, we show that the sufficient conditions for estimability
from Theorems \ref{th:suff} and \ref{th:suff.rd} and the necessary condition implied by (\ref{eq:nec}) coincide (except at a boundary point) for many 
common covariance families, thereby yielding sharp conditions for consistent estimability of $\beta$ under spatial confounding. 

\subsection{Matérn covariance}\label{sec:matern}

We first consider the case where the exposure and the confounder $(X,W)$ are jointly distributed as a GRF with the bivariate Matérn covariance function matrix \citep{gneiting2010matern,apanasovich2010cross}.
The Matérn covariance  between locations $ s $ and $ s' $ is given by
\[
C(h) = \sigma^2 \cdot \frac{2^{1 - \nu}}{\Gamma(\nu)} \left( \frac{\sqrt{2\nu} h}{\rho} \right)^\nu K_\nu\left( \frac{\sqrt{2\nu} h}{\rho} \right),
\]
where $ h = \|s - s'\| $ is the Euclidean distance, $ \sigma^2 $ is the marginal variance, $ \rho $ is the range parameter, $ \nu $ controls smoothness, $ \Gamma(\cdot) $ is the gamma function, and $ K_\nu(\cdot) $ is the modified Bessel function of the second kind. In a bivariate Matérn GRF, both the marginal covariance functions and the cross-covariance function are from the Matérn family. The following result provides a near-complete characterization of
consistent estimability of the slope $\beta$ between $Y$ and $X$ under unmeasured spatial confounding by $W$ for Matérn processes.  

\begin{corollary}[Matérn covariance]\label{cor:matern}
	Let $Y(s) = X(s) \beta + W(s)$ for $s \in \calD$, a bounded subset of $\mathbb R^d$ that contains an open $d$-dimensional ball. Let $(X,W)$ be a non-degenerate bivariate Matérn GRF with
	marginal smoothnesses $\nu_X$ and
	$\nu_W$ respectively, and cross-smoothness $\nu_{XW}$.
	\begin{enumerate}[label=(\alph*)]
		\item If $\nu_X < \nu_W + d/2$, then $\beta$ is consistently estimable based on $(Y,X)$ as long as
		$\nu_{XW} > \nu_X$. If $d=1$, a consistent estimator of $\beta$ is given by $OLS^{(p)}_n(X,Y)$ as in (\ref{eq:olsdiff}), the OLS estimator between $p^{th}$ order differences of $Y$ on those of $X$ on any 1-dimensional regular lattice, where $p$ is any integer exceeding $\nu_X $. If $d >1$,  a consistent estimator of $\beta$ is given by the OLS estimator between $m^{th}$ discrete Laplacians of $Y$ and $X$ on a regular $d$-dimensional lattice, where $m$ is any integer exceeding $\nu_X/2$.
		
		\item If $\nu_X > \nu_W + d/2$, 
		then $\beta$ is not consistently estimable by any estimator based on $\{(Y(s),X(s)) : s \in \calD\}$.
	\end{enumerate}
\end{corollary}

Corollary \ref{cor:matern} proves that the estimability of the slope $\beta$ between $Y=X\beta + W$ and $X$ depends solely on the smoothness parameters $\nu_X$, $\nu_W$, and $\nu_{XW}$ and not on the marginal variances, the spatial ranges, or the intra-site correlation between $X$ and $W$. This is a new and important finding as the spatial range parameter in Matérn covariances has been previously empirically shown to be driving spatial confounding bias \citep{paciorek}. Our result proves that while the ranges may dictate finite sample bias and, possibly, efficiency of different estimators, it plays no role in asymptotic consistency, where the only relevant notion of `scale' are the smoothnesses. As long as the cross-correlation function is smoother than the marginal correlation function of $X$, $\beta$ is consistently estimable when the difference $\nu_X - \nu_W$ is less than $d/2$. When consistent estimation is possible, the minimum order of differencing $p$ (or the order of discrete Laplacian $m=p/2$ for $d>1$) needed to obtain a consistent estimator is solely dictated by the smoothness of the observed covariate $X$.
On the other hand, when $\nu_X - \nu_W$ is greater than $d/2$, then $\beta$ cannot be consistently estimable using any method or estimator based on only observing $Y(s)$ and $X(s)$ for all $s \in \calD$. 
Figure \ref{fig:regions} demonstrates the region of consistent estimability as a function of $\nu_X$ and $\nu_W$ along with the minimal order of differencing or Laplacian needed to get a consistent estimator.  

\begin{figure}[h!]
	\centering
	\includegraphics[width=0.9\linewidth]{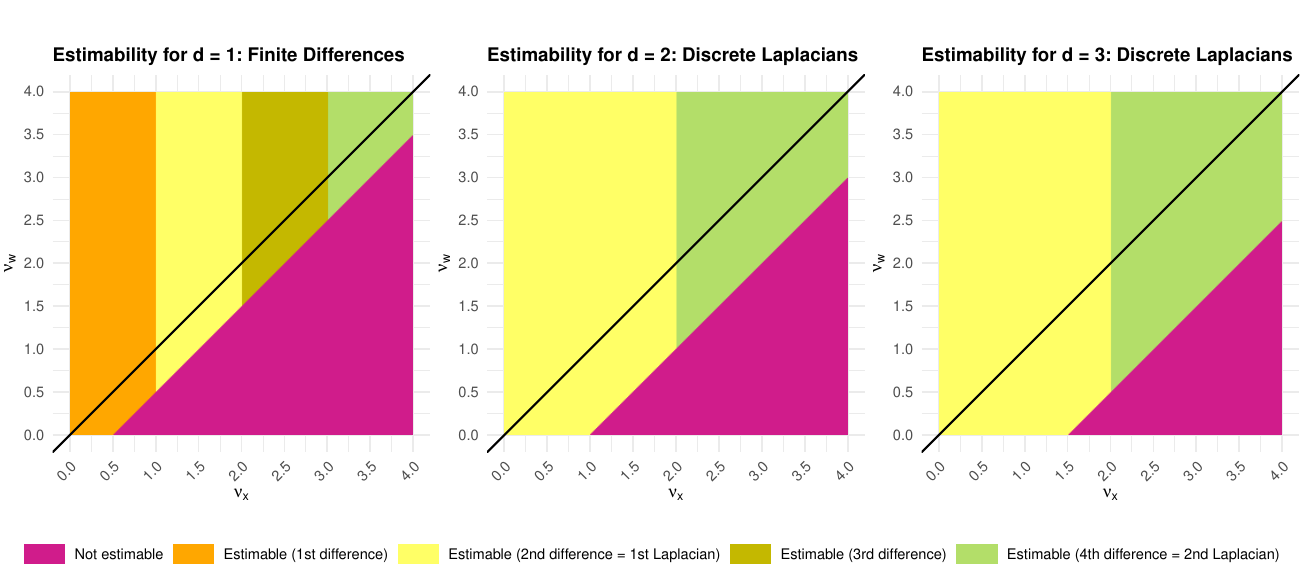}
	\caption{Region of consistent estimability of the slope $\beta$ in regression of $Y=X\beta + W$ on $X$ under spatial confounding in $\mathbb R^d$ for Matérn processes, as concluded from Corollary \ref{cor:matern}. Here $(X,W)$ is jointly a bivariate Matérn process with smoothnesses $\nu_X$ and $\nu_W$ and cross-smoothness $\nu_{XW} > \nu_X$. The region where $\beta$ is consistently estimable is color-coded by the minimum order of differencing/Laplacian needed for a consistent estimator.}
	\label{fig:regions}
\end{figure} 
We provide some intuition to this apparent `boon of dimensionality' proved in Corollary \ref{cor:matern} and seen in Figure \ref{fig:regions}, where the region of consistent estimability of $\beta$ expands with increasing dimension $d$ of the spatial domain, allowing $W$ to be up to $\frac d2$ degrees smoother than $X$. Our threshold of $\nu_X-\nu_W = d/2$ is consistent with the existing results on estimability of the slope in spatial regression and relates to the gap between the smoothness of sample paths of Matérn GRF and those of functions in the RKHS of the covariance kernel. Formally, if $X$ is a Mat\'ern GRF with covariance function
$C$ and smoothness parameter $\nu_X$, then typical sample paths have Sobolev
smoothness just below $\nu_X$, whereas the RKHS $\mathcal H_C$ coincides
(up to norm equivalence) with the Sobolev space of order $\nu_X+d/2$. Hence
$\mathcal H_C$ is $d/2$ smoother than the sample paths.  
See \cite{kanagawa2018gaussian} for a comprehensive review on this topic. \cite{wang2020prediction,bolin2025spatial}  
showed that the consistency of the GLS estimator using the true covariance of $W$ ($C_W$) depends on whether $X$ lies in the RKHS of $C_W$. \cite{yu2022parametric} also obtains the same threshold of $d/2$ where $X$ is a GRF independent of $W$. 

Also, it is well known that even for univariate processes, consistent estimability of GRF parameters depends on the dimensionality of the domain. For example, for univariate Matérn GRF, variance and range parameters are not infill consistently estimable for $d \leq 3$ \citep{stein1999interpolation,zhang2004inconsistent}, whereas they are consistently estimable when $d \geq 5$ \citep{anderes2010consistent} (the case $d=4$ remains open). A rough intuition is that, as $d$ increases, there are more directions to look for implied differences in the high-frequency spectral properties of the process. 
Thus, based on observations restricted to a bounded domain, the higher the dimension, the smoother $X$ can be relative to $W$ while still allowing consistent estimation of $\beta$. A similar discussion is in \cite{anderes2010consistent} whose estimators for spatial variance and range parameters in univariate Matérn GRF are also based on local differences. Heuristically, taking increments along enough distinct directions makes the summands involved in the estimators sufficiently
decorrelated, so the estimator converges as the sample size grows. This allows consistent estimability for a larger range of settings for higher $d$, as the discrete Laplacians use increments of the fields along $d$ orthogonal directions. 
A mathematical insight is provided in \cite{li2023inference} in the context of covariance parameter estimation in univariate Matérn GRF on Riemannian manifolds, who show that the necessary and sufficient condition for orthogonality is determined by Weyl's law on the growth of the eigenvalues of the covariance operators, which is dependent on $d$.

For the boundary case where $\nu_X=\nu_W + d/2$,  
the following Corollary proves that for $d=1$, when $\alpha_{11}$ is an even integer, the differencing based estimator of $\beta$ is still consistent.

\begin{corollary}\label{cor:boundary} If $(X,W)$ is a bivariate stationary GRF in $\mathbb R$ with covariance $K=(K_{k\ell})$, where each $K_{k\ell}$ satisfies a modified version of Assumption \ref{eq:K.assump}, where $\alpha$ is an even integer and $B(t)=t^\alpha \log t + o(t^\alpha \log t)$ with the derivatives $B^{(k)}(t)=t^{\alpha-k} \log t + o(t^{\alpha-k} \log t)$.
	Assume that Equation (\ref{eq:K12limit}) is also satisfied. Then the limiting result (\ref{eq:th1.main}) holds when $\alpha_{11}=2p$ for an integer $p$ and $\alpha_{22}=\alpha_{11}-1$. 
\end{corollary}

In the Matérn case, the corollary gives the exact sharp results that when the smoothness $\nu_X$ of $X$ is an integer and $\nu_{XW} > \nu_X$, $\beta$ can be consistently estimated if and only if $\nu_X \leq \nu_W +1/2$. However, the variance of the numerator of the differencing-based estimator shrinks to zero at a very slow (logarithmic) rate. 
Similar consistency results should hold for the Laplacian-based estimator when $d>1$ and $\nu_X$ is still an integer. For the more general case, where $\nu_X$ is not an integer, we conjecture  that
$\beta$ will still be consistently estimable when $\nu_X=\nu_W + d/2$ (as long as $\nu_{XW} > \nu_X$).
However, the difference or Laplacian-based estimator may not be consistent, and some other estimator may be needed.

\subsection{Power exponential family}\label{sec:powexp} We consider the case where $X$ and $W$ are GRFs on $\mathbb R$, with covariances from the power exponential family given by 
$K(t) = \sigma^2 \exp(-\phi |t|^\delta)$ for $0 < \delta < 2$. We exclude the case $\delta=2$ from our study, 
although it is a valid covariance function,
as it corresponds to an analytic process (see discussion at the end of Section \ref{sec:converse}). The next result characterizes the consistent estimability of $\beta$ under spatial confounding for the power exponential family.

\begin{corollary}[Power exponential covariance]\label{cor:powexp}
	Let $Y(s) = X(s) \beta + W(s)$ for $s \in \calD$, a bounded subset of $\mathbb{R}$ that contains an open interval. Let $(X,W)$ 
	be a non-degenerate stationary bivariate GRF such that $X$ and $W$  marginally have power exponential covariances with exponent parameters  $\delta_X$ and $\delta_W$ respectively, and the cross-covariance satisfies Assumption \ref{eq:K.assump} for some $\alpha_{12}>\delta_X$.  
	Then, 
	\begin{enumerate}[label=(\alph*)]
		\item If $\delta_X < \delta_W + 1$, then $\beta$ is consistently estimable. A consistent estimator of $\beta$ is given by $OLS^{(1)}_n(X,Y)$ as in (\ref{eq:olsdiff}), i.e., the OLS estimator between the first
		order differences of $Y$ on those of $X$ on a $1$-dimensional regular lattice.
		\item If $\delta_X > \delta_W + 1$, 
		then $\beta$ is not consistently estimable.
	\end{enumerate}
\end{corollary}

Like the Matérn family, the power exponential covariance family also has the near-zero asymptotic expansion as in Assumption \ref{eq:K.assump} with a principal irregular term $O(t^\alpha)$ where $\alpha=\delta$. Hence, Theorem \ref{th:suff} implies that $\delta_X < \delta_W + 1$ is sufficient for consistent estimability of $\beta$. On the other hand, when $\delta_X > \delta_W + 1$, we prove equivalence of measures for different values of $\beta$ by applying Theorem \ref{th:nonidgen}.
The spectral densities of power exponential covariance functions
are generally not available in closed forms except for special cases (e.g., $\delta=1$). Instead, we rely on characterization of the powered exponential covariance as the characteristic function of the Lévy stable distribution \citep{zolotarev1986one}. We then use asymptotic expansions of the probability density function (pdf) of these stable distributions near zero and infinity \citep{garoni2002levy,nolan2020univariate} to establish that the power spectral density varies asymptotically like $O(\omega^{-\delta-1})$ at high frequencies and is well-behaved at low frequencies. This justifies applying Theorem \ref{th:suff} to prove the equivalence of measures on the paths of $(Y,X)$ for two values of $\beta$ when $\delta_X > \delta_W + 1$. Thus, the line $\delta_X = \delta_W + 1$ provides a sharp boundary for the region of consistent estimability.

\subsection{Generalized Cauchy covariance}\label{sec:cauchy}

We next consider the four-parameter generalized Cauchy correlation family in $\mathbb R^d$ given by $C(s-s')=\sigma^2(1+\phi\|s-s'\|^\delta)^{-\kappa}$, for $\delta \in (0,2)$ and $\kappa > 0$ \citep{gneiting2004stochastic}. The following result provides conditions for the estimability of $\beta$ in (\ref{eq:dgpy}) when the marginal covariance functions of both $X$ and $W$ are from this family. As for the power exponential family, the case $\delta=2$ for the generalized Cauchy family also corresponds to
an analytic process and is hence not considered in this study. We have the following result on the estimability of $\beta$ for processes with the generalized Cauchy covariance family.

\begin{corollary}[Generalized Cauchy]\label{cor:cauchy}
	Let $Y(s) = X(s) \beta + W(s)$ for $s \in \calD$, a bounded subset of $\mathbb R^d$ that contains a $d$-dimensional open ball. Let $(X,W)$ 
	be a non-degenerate stationary bivariate GRF such that $X$ and $W$  marginally have generalized Cauchy covariance functions with parameters $(\delta_X,\kappa_X)$ and $(\delta_W,\kappa_W)$ respectively with $\delta_X \kappa_X > d$ and $\delta_W \kappa_W > d$. Also, the cross-correlation satisfies Assumption \ref{eq:K.assump} for some $\alpha_{12}>\delta_X$.  
	Then, 
	\begin{enumerate}[label=(\alph*)]
		\item If $d=1$ and $\delta_X < \delta_W + 1$, then $\beta$ is consistently estimable. A consistent estimator of $\beta$ is given by $OLS^{(1)}_n(X,Y)$ as in (\ref{eq:olsdiff}), i.e., the OLS estimator between the first
		order differences of $Y$ on those of $X$ on a $1$-dimensional regular lattice.
		\item If $d=1$ and $\delta_X > \delta_W + 1$, 
		then $\beta$ is not consistently estimable.
		\item If $2>d>1$, $\beta$ is always consistently estimable. A consistent estimator of $\beta$ is given by LAP$^{(1)}_n(X,Y)$ as in (\ref{eq:lap}), i.e., the OLS estimator between the first
		order Laplacians of $Y$ on those of $X$ on a $d$-dimensional regular lattice.
	\end{enumerate}
\end{corollary}

For $d=1$,
$\beta$ is not identified if $\delta_X > \delta_W +1$. This is expected. For $s-s' \to 0$, $C(s-s') = \sigma^2 - \sigma^2\phi\kappa \|s-s'\|^\delta + o(\|s-s'\|^\delta)$. Hence, $\delta$ is the exponent of the principal irregular term. Additionally, the spectral density $f$ satisfies $f(\omega) = O(\|\omega\|^{-\delta-d})$ as $\|\omega\| \to \infty$ \citep{lim2009gaussian}.
Thus,
$\delta$ determines both the exponent of the principal irregular term and the algebraic rate of decay of the spectral density at high frequencies, implying that the sufficient and necessary conditions for estimability of $\beta$ from Sections \ref{sec:suff} and \ref{sec:nec} are identical and sharp. 

Part (c) proves that $\beta$ can always be identified in $\mathbb R^2$ or higher dimensions as long as the cross-covariance between $X$ and $W$ is smoother than the covariance of $X$. This is because as both $\delta_X,\delta_W \in (0,2)$, $\delta_X$ is always less than $\delta_W + d$ for $d \geq 2$ which guarantees consistent estimability from Theorem \ref{th:suff.rd}. The conditions $\delta_X \kappa_X > d$ and $\delta_W \kappa_W > d$ are used in Corollary \ref{cor:cauchy} to ensure that the spectral density is continuous and convergent at low frequencies \citep{lim2009gaussian}, which is needed for the regularity conditions of Theorem \ref{th:nonidgen}.

\subsection{Linear model of coregionalization}\label{sec:lmc}

Correlated GRFs are often perceived to be formed by linear combinations of independent processes, with the weights determining the extent of correlation. This model is often termed the {\em linear model of coregionalization} \citep{gelfand2003proper,wackernagel2003multivariate}. 
The following result provides conditions for consistent estimability of $\beta$ under spatial confounding when $(X,W)$ are generated as a linear model of coregionalization. 

\begin{corollary}\label{cor:lmc}
	Let $U_1, \ldots, U_r$ denote $r$ independent univariate GRFs each from either the Matérn, power-exponential, or generalized Cauchy family in $\mathbb R^d$ with different exponent parameters $\delta_1, \delta_2, \ldots, \delta_r$ (if $U_r$ has a Matérn covariance, then $\delta_i = 2\nu_i$, $\nu_i$ being the smoothness parameter). Let $X = \sum_i a_{i}U_i$ and $W = \sum_i b_{i}U_i$ where $a_{i}$ and $b_{i}$ are real numbers. Let $ 
	\delta_X=\min \{\delta_i : a_{i} \neq 0\}$, $
	\delta_W=\min \{\delta_i : b_{i} \neq 0\}$, $
	\delta_{XW}=\min \{\delta_i : a_{i}b_{i} \neq 0\}$ and $Y(s) = X(s) \beta + W(s)$ for $s \in \calD$, a bounded subset of $\mathbb R^d$ that contains a $d$-dimensional open ball. Then, $\beta$ is consistently estimable if $\delta_X < \delta_W + d$ and $\delta_{XW} > \delta_X$ with $OLS^{(p)}_n(X,Y) \to \beta$ for any $p > \nu_X$ if $d=1$, or LAP$^{(m)}_n(X,Y) \to \beta$ for $m > \nu_X/2$ if $d>1$. When $\delta_X > \delta_W + d$, and $X$ is independent of $W$, then $\beta$ is not consistently estimable. 
\end{corollary}

Corollary \ref{cor:lmc} proves that when $(X,W)$ is based on a linear model of coregionalization, $\beta$ is consistently estimable if $X$ contains at least one factor $U_i$ that is not in $W$ (this ensures that $\delta_{XW} > \delta_X$), and when the roughest factor of $X$ (which determines its smoothness) is not more than $d/2$ degrees smooth than the roughest factor of $W$. 

\subsection{Covariances with nugget}\label{sec:error}

Spatial processes are often observed with some added noise (nugget).
Theorem 6 of \cite{stein1999interpolation} shows that equivalence or orthogonality of two measures on the paths of univariate GRFs is not affected by the addition of spatially independent noise. \cite{tang2021identifiability} showed that the measurement error variance (nugget) can be consistently estimable for univariate GRFs.
In the setting of spatial confounding, it is important to study whether measurement error impacts the consistency of estimators. The local differencing/Laplacian-based estimators we have proposed may no longer work when the processes have measurement error, as differencing noise amplifies it relative to the continuous component of the process. While this does not necessarily imply a lack of consistent estimability, it illustrates that the consistency of specific estimators may rely on the absence of noise. The following result characterizes consistent estimability of $\beta$ in the presence of measurement error.

\begin{theorem}
	\label{thm:error}
	Let $(X,W)$ denote a bivariate stationary GRF on $\mathbb R^d$ whose covariance $K=(K_{k\ell})$ belongs to one of the classes considered in Corollaries \ref{cor:matern}, \ref{cor:powexp}, \ref{cor:cauchy}, and \ref{cor:lmc},
	where each $K_{k\ell}$ satisfies Assumption \ref{eq:K.assump} ($d=1$) or 
	\ref{as:rd} ($d>1$) for some $L>0$ and with $\alpha_{k\ell}$ as the exponent of the principal irregular term. Let $Y(s)=X(s)\beta+W(s)$. If $\alpha_X < \alpha_W + d$ and $\alpha_{XW} > \alpha_X$, and we observe $Z(s_i)=Y(s_i)+\epsilon(s_i)$ and $\tilde X(s_i) = X(s_i) + \varepsilon(s_i)$ for $s_i$ on the regular lattice $\calG_n$ in $[0,L]^d$ where $\{\epsilon(s_i)\}$ and $\{\varepsilon(s_i)\}$ are iid zero mean Gaussian random variables independent of each other and of $(X,W)$, then $OLS(M_n Z, M_n \tilde X) \to \beta$ in probability as $n \to \infty$ where $M_n = D_n A_n$, $A_{n}= \left(\frac 1{k_n}I(s_j \in B_n(s_i))\right)$ an averaging matrix over local neighborhoods $B_n(s_i)$ of size $k_n$ around $s_i$ in $\calG_n$, and $D_n$ is a differencing ($d=1$) or discrete Laplacian ($d>1$) matrix of suitable order. 
	If $\alpha_X > \alpha_W + d$, and we observe $Z(s_i)=Y(s_i)+\epsilon(s_i)$ and $\tilde X(s_i) = X(s_i) + \varepsilon(s_i)$ for any countable set of locations $\{s_i\}$ in any bounded subset of $\mathbb R^d$, $\beta$ cannot be consistently estimated based on $\{(\tilde X(s_i),Z(s_i)) \mid i=1,2, \ldots\}$. 
\end{theorem}

Theorem \ref{thm:error} ensures that if consistent estimability is feasible in the noiseless case, then we can still estimate $\beta$ consistently in the presence of measurement error. We state this as a theorem, as the results do not directly follow from previous results. Particularly, the technique of taking local differences (Theorem \ref{th:suff}) or discrete Laplacians (Theorem \ref{th:suff.rd}) used to obtain a consistent estimator for $\beta$
does not work when there is measurement error.
We address this problem by using a {\em local-averaging-and-differencing} technique where we first average nearby observations around each grid point and then use differencing or discrete Laplacians at those grid points. The idea is illustrated in Figure \ref{fig:lad}. Averaging first a sufficiently large number of observations makes the noise variance small enough to yield a consistent estimator for $\beta$ by subsequent differencing. The formal details are in the proof. 

\begin{figure}[t]
	\centering
	\includegraphics[width=0.5\linewidth]{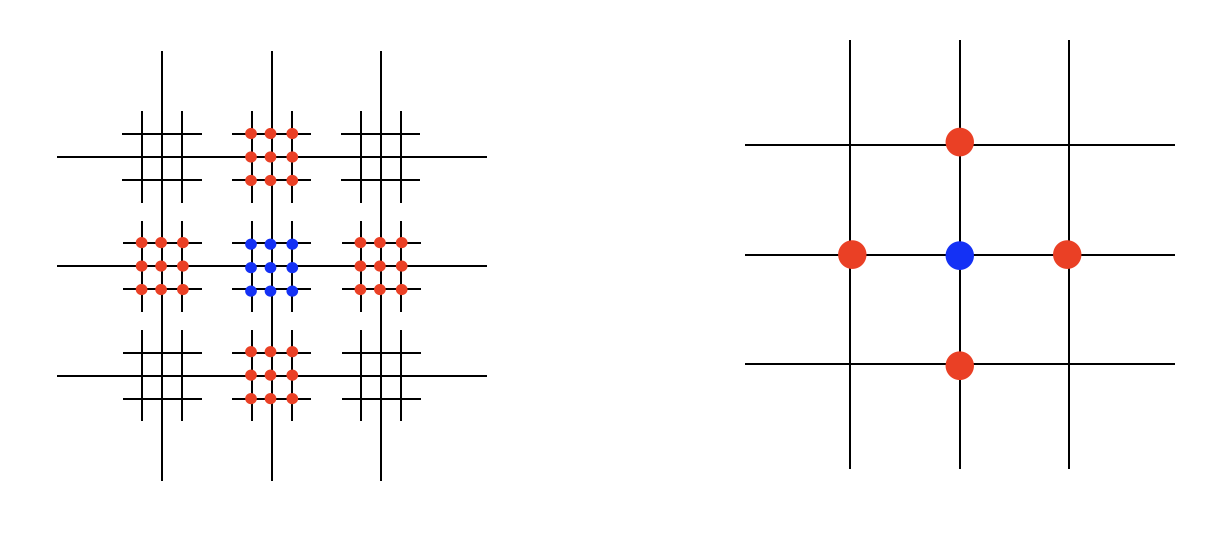}
	\caption{Local-averaging-and-differencing based estimation of $\beta$ when we observe $(\tilde X,Z)$, a measurement error contaminated version of $(X,Y)$.
		The left figure corresponds to the averaging part. For each point in the coarser grid, the bivariate $(\tilde X,Z)$ process over the finer sub-grid around it is averaged to create an averaged process for that point. The right figure corresponds to taking differences/discrete Laplacians.
		This is done by summing over the differences in the averaged process at the blue point and each of its neighbors (red points).}
	\label{fig:lad}
\end{figure}

\section{Non-stationary and non-Gaussian random fields}\label{sec:nsandng}
We now discuss conditions for consistent estimability of $\beta$ under spatial confounding when $X$ and/or $W$ is non-stationary or
non-Gaussian. For simplicity, we present results for a one-dimensional spatial domain. They can be extended to higher dimensions akin to the theory in Section \ref{sec:rd}.

\subsection{A general class of non-stationary covariances}\label{sec:nsgen}
We first provide a general result for a class of non-stationary covariance functions that subsumes popular examples. Let $K_\alpha$ be a covariance function on $[0,L]^2$ for some $L>0$, which admits the expansion
\begin{equation}\label{eq:nsclass}
	\begin{aligned}
		K_\alpha(s,s')
		&= \Gamma_{0}(s,s') + \Gamma_{2}(s,s')(s-s')^{2}
		+ \Gamma_{\alpha}(s,s')\,|s-s'|^{\alpha} + r(s,s'),
	\end{aligned}
\end{equation}
where the remainder $
r(s,s') \asymp \,\rho\big(\,|s-s'|\big)$ for some function $\rho$
which satisfies \[
\rho(t) = o(t^{\alpha'}), 
\qquad 
\rho^{(k)}(t) = o\big(t^{\alpha'-k}\big)
\quad\text{as } t\to 0,
\quad k = 1,2,\ldots
\]
with $\alpha' = \min(\alpha,2)$. 
Here $\Gamma_0,\Gamma_2$, and $\Gamma_\alpha$
are some real analytic functions on $[0,1]^{2}$. This class of covariances $K_\alpha$ generalizes the class in Assumption \ref{eq:K.assump} to non-stationary settings.

The following theorem proves consistency of the first differencing-based estimator for this non-stationary class. 

\begin{theorem}\label{th:ns} Let $Y=X\beta +W$ where $(X,W)$ is a zero-mean bivariate GRF on $[0,L]$ with covariance $K=(K_{k\ell})$ such that each $K_{k\ell}$ is of the form (\ref{eq:nsclass}) with exponent parameter $\alpha_{k\ell}$. If the function $\Gamma_{\alpha_{11}}$ satisfies $\int_{0}^{1} \Gamma_{\alpha_{11}}(s,s)\,ds \neq 0$ and $\alpha_{11} < \min\{\alpha_{12},\alpha_{22}+1,2\}$, then the OLS estimator regressing first differences of $Y$ on those of $X$ on the regularly spaced grid in $[0,L]$ is consistent for $\beta$. 
\end{theorem}

Consistent estimability holds for this class of non-stationary processes under very similar conditions on the principal irregular terms as the stationary case, $\alpha_{11} < \alpha_{12}$ and $\alpha_{22}+1$. The additional $\alpha_{11} < 2$ can be relaxed by replacing first-order differences with higher-order ones, as before. The condition $\int_{0}^{1} \Gamma_{\alpha_{11}}(s,s)\,ds \neq 0$ is essentially the generalization of $c \neq 0$ condition in Assumption \ref{eq:K.assump}. 

\subsection{Location-warped non-stationary GRF}\label{sec:warped}
We first consider the popular class of non-stationary processes of \cite{sampson1992nonparametric} created from stationary processes by warping of locations (see also \citep{zammit2022deep} for a modern take on location warping using deep neural networks). 
The following result presents necessary and sufficient conditions for consistent estimability of $\beta$ for such warped non-stationary processes in terms of equivalence or orthogonality of the measures on the paths of the original stationary processes with respect to different values of $\beta$. 

\begin{corollary}\label{cor:ns_warped} Let $(X(s),W(s))$ be a bivariate stationary GRF on $\mathbb R$ with covariance $K_{k\ell}$ satisfying Assumption \ref{eq:K.assump} for some $L>0$ and with exponents $\alpha_{k\ell}$. Let $(X^*(s),W^*(s)) = (X(f(s)),W(f(s)))$
	where $f: [0,L] \to [0,L]$ is a fixed function.
	Define $Y^*(s)=\beta X^*(s) + W^*(s)$.
	Then:
	\begin{enumerate}[label=(\alph*)]
		\item{\em Consistency:}
		If 
		$\alpha_{11} < \min\{\alpha_{12},\alpha_{22} + 1, 2\}$ and $f$ is a Lipschitz function that is piecewise analytic and monotonic with $f$ not being a constant on at least one of the pieces, then the OLS estimator based on first differences of $Y^*$ and $X^*$
		is consistent for $\beta$. 
		\item{\em Equivalence:} If $\alpha_{11} > \alpha_{22} + 1$, and $X$ and $W$ have spectral densities that satisfy the regularity conditions of Theorem \ref{th:nonidgen} with high frequency decay $\asymp |\omega|^{-(\alpha_{11}+1)}$ and $\asymp |\omega|^{-(\alpha_{22}+1)}$ respectively,
		then the measures on the paths of $(Y^*(s),X^*(s))$ are equivalent for any two different values of $\beta$.
	\end{enumerate}
\end{corollary}

The result proves that even under non-stationarity due to location warping, where the warping function is allowed to be irregular (only piecewise smooth), the same differencing-based estimator is consistent for $\beta$ under the same assumptions for the underlying unwarped stationary fields. Here piecewise analytic means that there exist \(0=\tau_0<\tau_1<\cdots<\tau_m=L\) such that, for each \(r=1,\ldots,m\), \(f|_{[\tau_{r-1},\tau_r]}\) extends to an analytic and monotone function on an open neighborhood of \([\tau_{r-1},\tau_r]\). The proof is based on showing that the warped covariance belongs to the class of non-stationary processes defined in (\ref{eq:nsclass}) and leverages Theorem \ref{th:ns}. We note that the differencing-based estimator does not require knowing or estimating the warping function.

\subsection{Variance and range non-stationarity}\label{sec:paciorek}
Next, we consider another popular class of non-stationary covariance functions \citep{paciorek2006spatial} given by 
\begin{equation}\label{eq:pac}
	\begin{aligned}
		K^{NS}(s, s')
		&=
		\frac{\sigma(s)\sigma(s')\, \left(|\Phi(s)|\, |\Phi(s')|\right)^{1/4}}{
			\left| \frac{\Phi(s) + \Phi(s')}{2} \right|^{1/2}} K(\sqrt{Q(s,s')}),  
		\mbox { where } \\
		Q(s,s') &= 
		(s - s')^T
		\left( \frac{\Phi(s) + \Phi(s')}{2} \right)^{-1}
		(s - s').
	\end{aligned}
\end{equation}
Here, both the variance $\sigma^2(s)$ and the range $\Phi(s)$ are positive spatially varying functions. 

For multivariate processes, the generalization we consider is
\begin{equation}\label{eq:kleiber}
	\begin{aligned}
		K_{k\ell}^{NS}(s, s')
		&=
		\frac{\sigma_k(s)\sigma_{\ell}(s')\, \left(|\Phi_{k\ell}(s)|\, |\Phi_{\ell k}(s')|\right)^{1/4}}{
			\left| \frac{\Phi_{k \ell}(s) + \Phi_{\ell k}(s')}{2} \right|^{1/2}} K_{k\ell}(\sqrt{Q_{k\ell}(s,s')}),  
		\mbox { where } \\
		Q_{k \ell}(s,s') &= 
		(s - s')^T
		\left( \frac{\Phi_{k \ell}(s) + \Phi_{\ell k}(s')}{2} \right)^{-1}
		(s - s'), 
	\end{aligned}
\end{equation}
Each $K_{k\ell}$ satisfies Assumption \ref{eq:K.assump} with the corresponding principal irregular term exponent equaling $\alpha_{k\ell}$.
Examples of valid multivariate covariances of the form (\ref{eq:kleiber}) are given in \cite{kleiber2012nonstationary}. 

The following result shows that the differencing-based OLS estimators are consistent for $\beta$  for this class of non-stationary covariances as long as the local range and variance functions are reasonably well-behaved (piecewise analytic) and the exponent parameters in the principal irregular terms of $K_{k\ell}$ satisfy the conditions of Theorem \ref{th:suff}.

\begin{corollary}\label{cor:nspaciorek} If $(X,W)$ is a bivariate zero-mean GRF on $\mathbb R$ generated from a valid bivariate covariance function matrix $(K^{NS}_{k\ell})$ where each $K^{NS}_{k\ell}$ has the form (\ref{eq:kleiber}) and $K_{k\ell}$ is a stationary covariance satisfying Assumption \ref{eq:K.assump} on some $L > 0$, and with exponent $\alpha_{k\ell}$. Let $Y=X\beta + W$. 
	\begin{enumerate}[label=(\alph*)]
		\item Consistency: If $\alpha_{11} < \min\{\alpha_{12},\alpha_{22} + 1, 2\}$, all $\sigma_k(s)$ and $\Phi_{k \ell}(s)$ functions are positive, continuous, and piecewise analytic in $[0,L]$, the OLS estimator based on first differences of $X$ and $Y$ on a regular grid $jh$, $j=0,\ldots,n$, $h=L/n$ in $[0,L]$ is consistent for $\beta$. 
		\item Equivalence: If $\alpha_{11} > \alpha_{22}+1$, the class of functions for $\Phi_{k \ell}$ and $\sigma_k(s)$ includes the constant functions, and the spectral densities corresponding to $K_{11}$ and $K_{22}$ satisfy the regularity conditions of Theorem \ref{th:nonidgen} with high frequency decay $\asymp |\omega|^{-(\alpha_{11}+1)}$ and $\asymp |\omega|^{-(\alpha_{22}+1)}$ respectively, then $\beta$ is not estimable on the paths of $(X,Y)$.
	\end{enumerate}
\end{corollary}

As in Section \ref{sec:warped} with the warping function $f$, the differencing-based estimator consistently estimates $\beta$ without having to know or estimate the local variance and range functions.

\subsection{Heavy-tailed processes}\label{sec:heavy}

We now give sharp results on when $\beta$ can and cannot be consistently estimated, assuming a common type of non-Gaussian processes -- heavy-tailed processes based on scale mixture of GRFs.

\begin{corollary}\label{cor:heavy}
	Let $Z^*=(X^*,W^*)$ be a bivariate stationary GRF on $\mathbb R$ with covariance function matrix 
	$K=(K_{k\ell})$ where $K_{k\ell}$ satisfies Assumption \ref{eq:K.assump} for some $L>0$, and with exponent
	$\alpha_{k\ell}$, $k,\ell\in\{1,2\}$. Let 
	$\mathrm{Var}(X^*(s))=\mathrm{Var}(W^*(s))=1$
	and $\sigma_X\in(0,\infty)$ and $\sigma_W\in(0,\infty)$ be positive processes on $\mathbb R$, independent of $Z^*$ and not depending on $\beta$, with
	almost surely piecewise analytic paths. Let 
	$
	X(s)=\sigma_X(s) X^*(s)$, $W(s)=\sigma_W(s) W^*(s)$, $Y(s)=X(s)\beta+W(s)$. Then, 
	
	\begin{enumerate}[label=(\alph*)]
		\item Consistency: If $\alpha_{11} < \min\{\alpha_{12},\alpha_{22} + 1, 2\}$,
		the OLS estimator based on first differences of $X$ and $Y$ on a regular grid $jh$, $j=0,\ldots,n$, $h=1/n$ in $[0,L]$ is consistent for $\beta$. 
		\item Equivalence: If $\alpha_{11} > \alpha_{22}+1$
		and $X^*$ and $W^*$ have spectral densities that satisfy the regularity conditions of Theorem \ref{th:nonidgen} with high frequency decay $\asymp |\omega|^{-(\alpha_{11}+1)}$ and $\asymp |\omega|^{-(\alpha_{22}+1)}$ respectively, then $\beta$ is not consistently estimable on the paths of $(X,Y)$.
	\end{enumerate}
\end{corollary}

Corollary \ref{cor:heavy} gives sharp conditions for consistent estimation of $\beta$
for a variety of heavy-tailed non-Gaussian random fields obtained by
multiplying the base GRF $(X^*,W^*)$ with positive scale fields.
With 
$\sigma_X^2(s)\sim\mathrm{Inv\text{-}Gamma}(\kappa_X(s)/2,\kappa_X(s)/2)$ and
$\sigma_W^2(s)\sim\mathrm{Inv\text{-}Gamma}(\kappa_W(s)/2,\kappa_W(s)/2)$, at each location $s$, $X(s)$ and $W(s)$ will marginally have univariate Student's $t$ distributions with location-specific degrees of freedom
$\kappa_X(s)$ and $\kappa_W(s)$ respectively. 
To construct such a scale process with analytic sample paths and Inverse-Gamma marginals, let $G$ be a GRF on $\calD$ which has almost surely analytic sample paths (e.g., one with squared exponential covariance)
and define $\sigma_X^2(s) = F^{-1}_{\kappa(s)}(\Phi(G(s)))$, where $F_\kappa$ is the $\mathrm{Inv\text{-}Gamma}(\kappa/2, \kappa/2)$ CDF; the resulting path of $\sigma_X$ is almost surely analytic because $\Phi$ is analytic, $\kappa$ is analytic and bounded away from $0$, and $(\kappa, u) \mapsto F^{-1}_\kappa(u)$ is jointly real analytic on $\{\kappa > 0\} \times (0,1)$.
Similarly, when the scale fields have exponential distributions, $X(s)$ and $W(s)$ follow Laplace distributions at each location. In all these constructions, the covariance
exponents $\alpha_{k\ell}$ are unchanged because the scale fields are smooth. 
Like the preceding results in this section, Corollary \ref{cor:heavy} can be extended to relax the $\alpha_{11} < 2$ assumption using higher order differences, and to $\calD \in \mathbb R^d$ using discrete Laplacians.

\section{Multivariate Exposure}\label{sec:mult}

In many applications, $X$ is a multivariate GRF, e.g., when $X$ consists of both the exposure of interest and a set of measured confounders. The next result provides sufficient conditions for consistent estimability of $\beta$ for multivariate $X$. For simplicity, we consider a bivariate $X$, but the conditions generalize for more variables. 

\begin{theorem}\label{thm:multsharp}
	Let $(X_1,X_2,X_{3}=W)$ be a trivariate  stationary GRF on $\mathbb R^d$ with strictly positive definite covariance function matrix $K=(K_{k\ell})$ where $K_{k\ell}$ satisfies Assumption \ref{as:rd} for some $L>0$ and with exponent $\alpha_{k\ell}$ .
	Let $Y= \beta_1 X_1 + \beta_2 X_2 + W$ and $(Y,X_1,X_2)$ be observed. 
	For $v=(v_1,v_2)^T \in \mathbb R^2$, let $X_v = v_1 X_1 + v_2 X_2$, and denote the covariance of $X_v$ by $K_{vv}$ and the cross-covariance of $X_v$ and $W$ by $K_{v3}$.  
	Then $\beta=(\beta_1,\beta_2)^T$ is consistently estimable if for all $v \neq 0$, $K_{vv}$ and $K_{v3}$ satisfy Assumption \ref{eq:K.assump} with exponents $\alpha_{vv}$ and $\alpha_{v3}$ such that $\alpha_{vv} < \min\{\alpha_{v3},\alpha_{33}+1\}$.
\end{theorem}

From Theorem \ref{thm:multsharp} we see that in the case of multivariate exposure, the sufficient condition for consistent estimability of $\beta$ is more nuanced. The constraint $\alpha_{vv} < \min\{\alpha_{v3},\alpha_{33}+1\}$ is stronger than assuming that the individual variables $X_1$ and $X_2$ satisfy these constraints. Indeed, the individual conditions $\alpha_{ii} < \min\{\alpha_{i3},\alpha_{33}+1\}$ for $i=1,2$ are satisfied by setting $v=(1,0)$ and $(0,1)$ respectively. But the Theorem requires $\alpha_{vv} < \min\{\alpha_{v3},\alpha_{33}+1\}$ to hold for all non-trivial $v$. This is essentially a condition restricting the amount of collinearity between $X_1$ and $X_2$, requiring that no linear combination of $X_1$ and $X_2$ is too smooth relative to $W$. 
To see the importance of this condition, suppose the exponent of $X_1$ satisfies $\alpha_{11} < \min\{\alpha_{13},\alpha_{33}+1\}$.  Let $X_2 = cX_1 + U$ for some very smooth $U \perp W$ ($\alpha_{UU} > \alpha_{33} +1$). 
Then exponents of $X_2$ are same as those of $X_1$ and thus satisfy $\alpha_{22} < \min\{\alpha_{23},\alpha_{33}+1\}$. However, with $v=(1,-c)^T$, $v^T X = X_2 - cX_1 = U$ fails to satisfy the condition $\alpha_{vv} < \min\{\alpha_{v3},\alpha_{33}+1\}$ as $\alpha_{vv}=\alpha_{UU} > \alpha_{33}+1$. It is clear that $\beta=(\beta_1,\beta_2)^T$ cannot be identified in this scenario as $Y=\beta_1 X_1 + \beta_2 X_2 + W = (\beta_1 +c) X_1 + 0.X_2 + W^*$ where $W^*=W+U$. Hence, the collinearity of $X_1$ and $X_2$ needs to be restricted relative to the smoothness of $W$. 

Similar but more general results can be derived for $X$ being $q$-dimensional for any fixed $q$, requiring all linear combinations of $X$'s to satisfy the relative roughness condition. 

\section{Numerical experiments}\label{sec:sim}
We conduct numerical experiments using simulated data to examine finite sample results in settings covered by the theoretical results.

\subsection{Estimation in $\mathbb R$}\label{sec:sim1d}

We consider data generated on a regular grid of $n$ equispaced points in $[0,1]$. The exposure $X$ and the confounder $W$ are jointly generated from a bivariate Matérn GRF and $Y=X\beta+W$ where $\beta=2$. By Corollary \ref{cor:matern} and Figure \ref{fig:regions}, the different smoothness bands (intervals) for $\nu_X$, requiring different estimators for consistency of $\beta$, are $(0,1)$, $(1,2)$, $\ldots$. Hence, we consider two values for the smoothness of $X$, i.e., $\nu_{X}\in\{0.7,1.2\}$ such that they are on either side of the smoothness cutoff of $\nu_X=1$ and thus should require different estimators to consistently estimate $\beta$ when it is consistently estimable. We choose the smoothness of $W$ to be $\nu_{W}=\nu_{X}+\delta$
where we vary $\delta\in\{-0.6,-0.3,0,0.3\}$.
For $\delta=-0.6$, according to Corollary \ref{cor:matern}, $\beta$ is not consistently estimable even when $X$ and $W$ are independent. Hence, for this case, we choose the intra-site correlation $\rho$ between $X$ and $W$ to be $0$. For $\delta=-0.3$, we are in a scenario where the confounder $W$ is rougher than the exposure $X$, but still $\beta$ should be consistently estimable according to Corollary \ref{cor:matern}. For $\delta \neq -0.6$, we set the cross-smoothness $\nu_{XW}$ to be $\nu_X + 0.25$ and the intra-site cross-correlation to be $\rho=\min\{0.5,\sqrt{\nu_{X}\nu_{W}}/\nu_{WX}\}$. For each combination, we generate data for sample sizes $n\in\{100,500,1000,2000\}$ and for each sample size we run $100$ replicate experiments. We compare the performance of three estimators of $\beta$, the OLS estimator between $Y$ and $X$, and the OLS estimators between the first or second differences of $Y$ with the corresponding differences of $X$. 

The results are summarized in the box-whisker plots of estimates of $\beta$ in Figure \ref{fig:mat1d}. Table \ref{tab:rmse_table}, \ref{tab:bias_table}, and \ref{tab:variance_table} provide the root mean squared error (RMSE), bias, and standard deviation, respectively.
We first look at the case where $\delta=\nu_W-\nu_X=-0.6$. This is a scenario where there is no confounding, i.e., $\rho=0$. So the OLS estimators should be unbiased but not consistent, as $\beta$ is not consistently estimable according to Corollary \ref{cor:matern}. We see this corroborated in the results. For both choices of $\nu_X$, all three sets of estimates are centered around the true $\beta$, but none of the variances shrink with increasing sample size. We then look at the remaining scenarios for $\nu_X=0.7$ (top row of Figure \ref{fig:mat1d}, excluding the first column). For all these scenarios, as $\delta > -0.5$, $\beta$ is expected to be consistently estimated by taking first- and second-order differences. This is validated in the results with both the difference-based estimators converging towards the truth with shrinking variances as $n$ increases. The OLS estimator not only has a non-vanishing variance but is also biased due to confounding (as $\rho \neq 0$). Finally, for $\nu_X=1.2$ and $\delta > -0.5$, we see that in addition to the OLS estimator, the first-difference based estimator is also now biased with non-vanishing variance. The results align with Corollary \ref{cor:matern} and Figure \ref{fig:regions}, as when $\nu_X > 1$,  two- or higher-order differencing yields consistency. The second-order differencing-based estimator can be seen to have diminishing bias and variance, in line with its proven consistency. 

\begin{figure}[t!]
	\centering
	\includegraphics[width=0.9\linewidth]{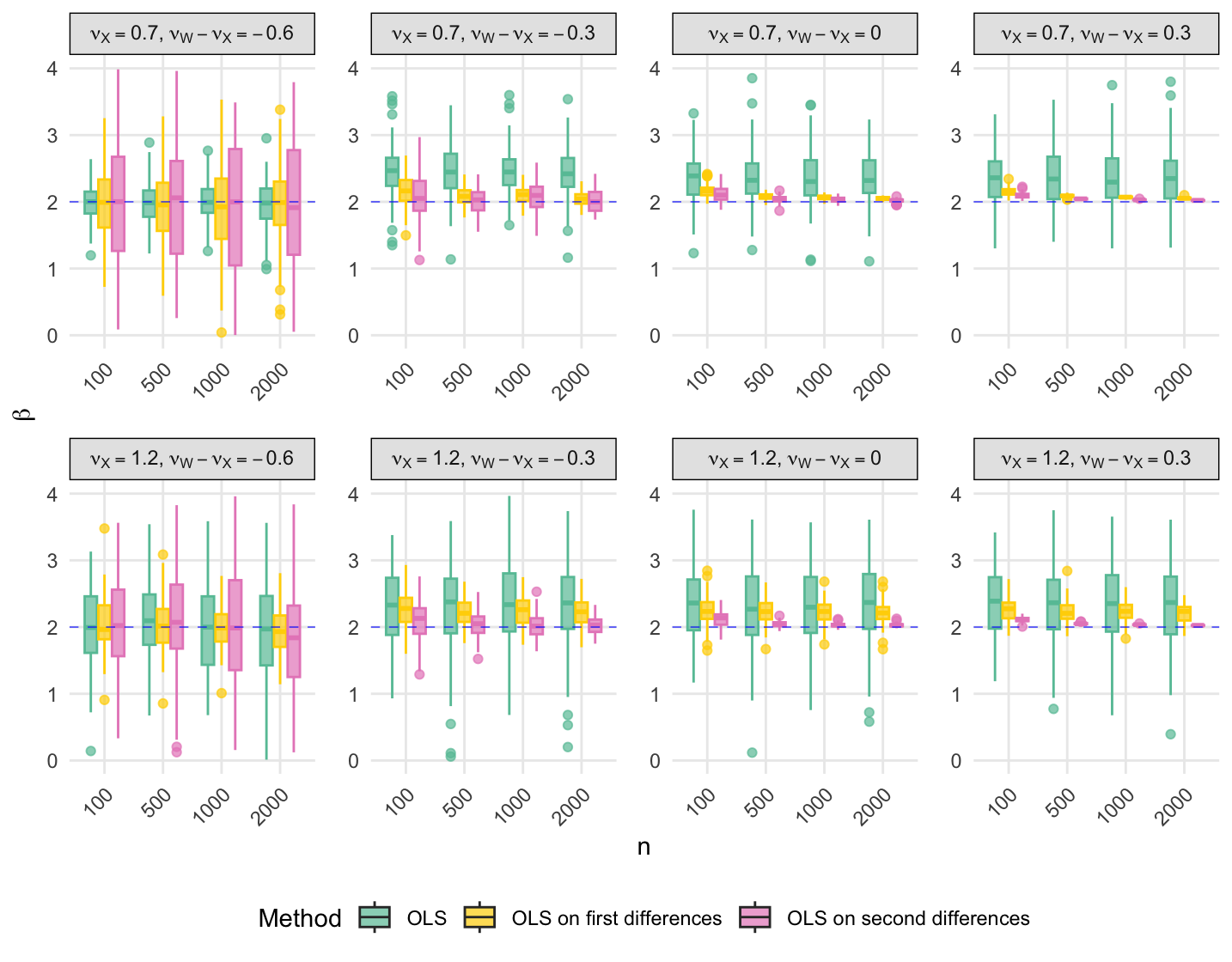}
	\caption{Estimates of $\beta$ for regression between GRF $Y=X\beta+W$ and $X$ when both the exposure $X$ and the unmeasured confounder $W$ have Matérn covariances with smoothnesses $\nu_X$ and $\nu_W$ respectively.}
	\label{fig:mat1d}
\end{figure}

\subsection{Results for $\mathbb R^2$ and variance convergence rates}\label{sec:sim2d}
We generate data on a regular $n \times n$ grid on the square $[0,1]^2$. The exposure $X$ and the unobserved confounder $W$ are modeled jointly as a bivariate Matérn GRF with parameter combinations as listed in Table \ref{tab:r2}. 

\begin{table}[h]
	\centering
	\begin{tabular}{cccc}
		\hline
		$\nu_{X}$ & $\delta_{W} = \nu_W - \nu_X$ & $\nu_{XW}$ & $\rho$ \\
		\hline
		1 & -0.6 & 1.25 & 0.204 \\
		1 & -0.4 & 1.25 & 0.306 \\
		1 & -0.2 & 1.25 & 0.408 \\
		1 & \phantom{-}0   & 1.25 & 0.500 \\
		1 &  \phantom{-}0.2 & 1.10 & 0.500 \\
		1 &  \phantom{-}0.4 & 1.20 & 0.500 \\
		\hline
	\end{tabular}
	\caption{Parameter values for simulating the bivariate Matérn GRF $(X,W)$ in $\mathbb R^2$.}
	\label{tab:r2}
\end{table}

We consider the case $\nu_X = 1$, which is in the first estimability band in Figure \ref{fig:regions} as $\nu_X < 2$, implying that
a first-order discrete Laplacian-based OLS estimator should be consistent. We vary $\delta = \nu_W - \nu_X \in \{-0.6, -0.4, -0.2, 0, 0.2, 0.4\}$.
Note that $\delta = -0.6$ constitutes a scenario where the gap between $\nu_X$ and $\nu_W$ is large enough to violate the estimability assumption if the domain was in $\mathbb R$ but is consistently estimable here as we are in $\mathbb R^2$. The cross-smoothness is fixed at $\nu_{XW} = 1.25$ if $\nu_W \leq \nu_X$ and $\nu_{XW}=(\nu_X+\nu_W)/2$ if $\nu_W > \nu_X$, ensuring that $\nu_{XW} > \nu_X$ always. The intra-site correlation $\rho$ is set according to $\min\left(0.5, \nu_X\nu_W/\nu_{XW}^2\right)$.
This expression for $\rho$ ensures that the bivariate Matérn correlation function is valid \citep{gneiting2010matern}. 

For each parameter combination, we generate 100 replicate datasets on grids of size $N \in \{225,529,1024,2025,4900,10000\}$ where $N=n^2$.
For each replicate, we compare the naïve OLS estimator of $\beta$ regressing $Y$ on $X$, and the OLS estimator regressing the first-order discrete Laplacians of $Y$ on those of $X$. For all the specified choices of $\nu_X$ and $\delta$, we have proven that this estimator is consistent. The estimates of $\beta$ are given in Figure \ref{fig:mat2d} while the actual RMSE, biases, and standard deviations are given in Supplemental Tables \ref{tab:mse_table_2d}, \ref{tab:bias_table_2d}, and \ref{tab:variance_table_2d} respectively. 
We see that the Laplacian-based estimator converges towards the true $\beta$ with diminishing bias and variance as the sample sizes increase, whereas the naïve OLS estimator is biased (as $\rho \neq 0$) and does not have vanishing variance, aligning with our theory.

We also use this experiment to empirically study the rates of convergence of the Laplacian-based estimator and compare them to the rates inferred from the theoretical results.
We
plot the empirical log-standard deviations of the Laplacian-based estimator of $\beta$ along with the theoretical variance bounds of the log-standard deviations of the numerator of the Laplacian estimator. The latter is of the form $\log sd = \text{constant} + \gamma \log N$ where $\gamma = -\frac 12+\frac 12 \max\left(\frac{\alpha_X-\alpha_W}d,0\right)$ (see e.g., Equation (\ref{eq:raterd}) in the proof of Theorem \ref{th:beta.con.Lap}).
Both curves are plotted as functions of the log-sample sizes and are centered appropriately, as we are only interested in looking at the slope $\gamma$, which controls the rate. The results are visualized in Figure \ref{fig:mat2drates}, and we see that for all the scenarios in Table \ref{tab:r2} the empirical log standard deviations align very closely with the theoretical upper bounds, being approximately log-linear in $N$ with the empirical slope $\hat \gamma \approx \gamma$.
Thus, the upper bound on the rates we establish in deriving consistency is verified in the simulations and is likely to be sharp. 

\begin{figure}[t]
	\centering
	\includegraphics[width=0.9\linewidth]{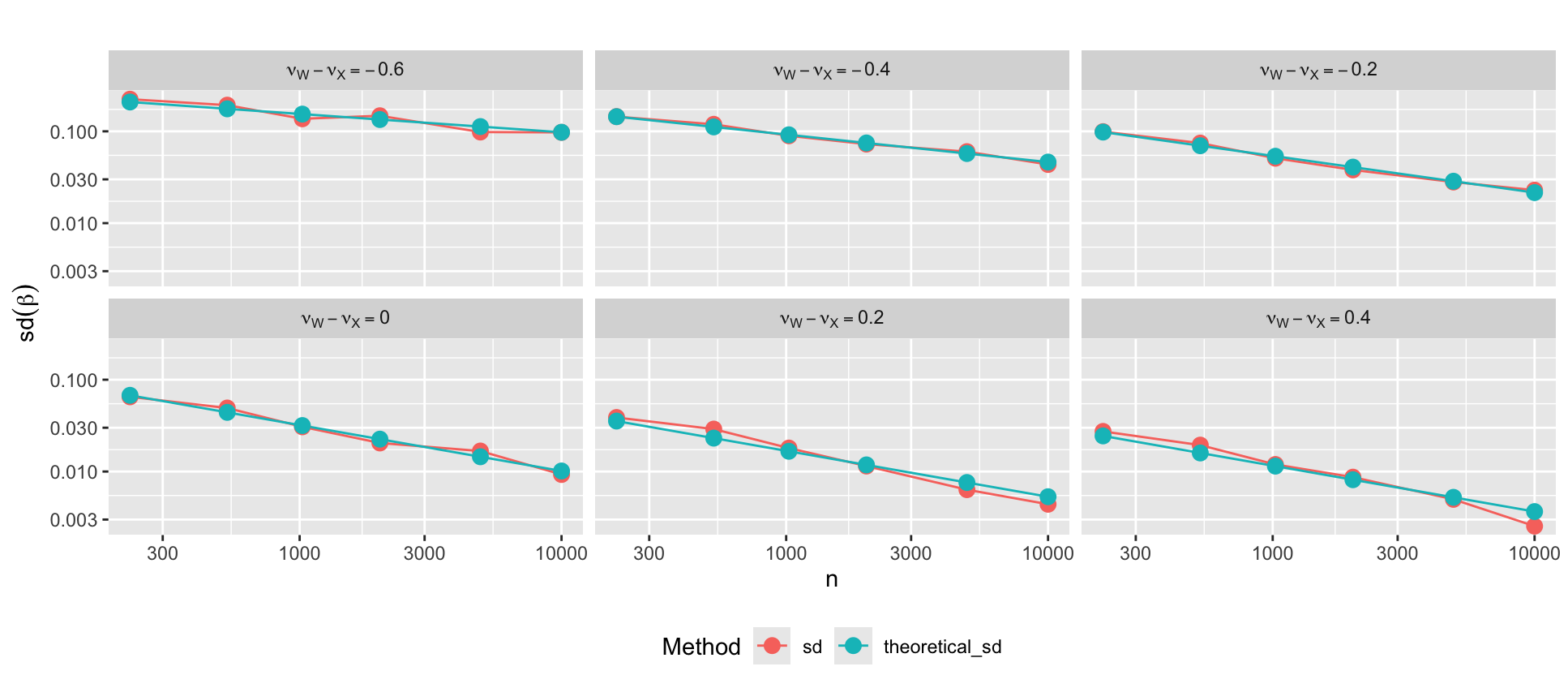}
	\caption{Comparison of empirical standard deviations of $\beta$ and theoretical bound on the standard deviation for the numerator of the Laplacian-based consistent estimators of $\beta$.}
	\label{fig:mat2drates}
\end{figure}

Additional simulations are detailed in Supplemental Section \ref{sec:addsim}, which includes comparisons with the GLS estimator (Section \ref{sec:simgls}), assessment of the estimators for irregular designs (Section \ref{sec:simirreg}), and under added noise (Section \ref{sec:simnoise}).

\section{Discussion}\label{sec:disc}
We provide both sufficient and necessary conditions for consistent estimability of the slope between two GRFs under unmeasured spatial confounding.
The sufficient condition is based on directly proving the consistency of
the local differencing or discrete Laplacian-based estimator,
as long as the exposure is not too smooth compared to the confounder, and the cross-correlation is smoother than the correlation of exposure.
Functional analysis in Paley-Wiener spaces yields a simple, easy-to-check necessary condition: the integral of the exposure to the confounder spectral density ratio must be infinite.
In the process, we develop a general result on the equivalence of multivariate Gaussian random field measures where each of the individual component fields is allowed to have a different smoothness.
We show that for common families of covariance functions like the Matérn, power exponential, generalized Cauchy, or coregionalization models,
the sufficient and necessary conditions are the same, except at a boundary point, thereby providing a complete characterization of consistent estimability. We show that the consistent estimability results remain unchanged if the outcome and the exposure are observed with measurement error, although one has to now use a local-averaging-then-differencing-based estimator. We extend the results to popular classes of non-stationarity and non-Gaussianity and to multivariate covariate processes.

While the focus of this manuscript has primarily been to resolve the consistent estimability problem for the slope between spatial random fields, in the process, we have provided explicit estimators that can be used to consistently estimate $\beta$ when the consistent estimability conditions hold. These estimators, based on local (averaging and) differencing, are non-parametric and do not rely on any parametric knowledge of the covariance functions. As the smoothness of the $X$ and $Y$ processes can typically be consistently estimated, the feasibility of confounding adjustment and, if feasible, the order of differencing needed for consistent estimation, can be informed just from the observed data (see Supplemental Section \ref{sec:guidance} for guidance on this).

Future work also needs to study the theoretical efficiency of the differencing-based estimators. Uncertainty quantification and inference methods also need development, which would require estimation of other (spatial covariance) parameters. Comparisons with other estimators also need to be done. Some preliminary results (See Supplemental Section \ref{sec:gls} and Proposition \ref{prop:gls}) show that the GLS estimator with even slight covariance misspecification is inconsistent under some settings of spatial confounding where the differencing-based estimator is consistent. Exhaustive comparisons are needed to understand the applicability of each estimator. We also show in Supplemental Section \ref{sec:irreg} and Theorem \ref{th:irreg-second} that our approach to estimating $\beta$ can be extended to irregular designs using spacing-weighted differences.

Finally, while we have restricted the study to Euclidean domains, we believe that similar sharp necessary and sufficient conditions can be derived for consistent estimability of the slope between two GRFs on Riemannian manifolds under spatial confounding. Necessary and sufficient conditions for equivalence or orthogonality of univariate GRF measures in such domains are characterized by the spectral density ratio of the individual eigenfunctions \cite[see, e.g., Lemma 3 of][]{li2023inference}. Extension of such a result to multivariate GRF, as we did in Theorem \ref{th:equiv} and subsequent application to the spatial confounding problem, should yield a condition similar to (\ref{eq:nec}) for determining consistent estimability in such domains. 

\section*{Acknowledgement} We thank Dr. Thomas Hsiao for helpful discussions on this topic. We acknowledge the use of generative AI for help with some of the proof development and writing of code for the numerical experiments. All work was independently reviewed by the authors, who take sole responsibility for the content. The work was partially supported by the National Institutes of Environmental Health Sciences grant R01 ES033739. 

\bibliographystyle{plainnat}
\bibliography{ref}

@article{zhu2002,
  title={Parameter estimation for fractional Brownian surfaces},
  author={Zhu, Zhengyuan and Stein, Michael L},
  journal={Statistica Sinica},
  pages={863--883},
  year={2002},
  publisher={JSTOR}
}

@article{bingham1972,
author = {Bingham, N. H.},
title = {A Tauberian theorem for integral transforms of Hankel type},
journal = {Journal of the London Mathematical Society},
volume = {s2-5},
number = {3},
pages = {493-503},
url = {https://londmathsoc.onlinelibrary.wiley.com/doi/abs/10.1112/jlms/s2-5.3.493},
eprint = {https://londmathsoc.onlinelibrary.wiley.com/doi/pdf/10.1112/jlms/s2-5.3.493},
year = {1972}
}

@article{nobre2021effects,
  title={On the effects of spatial confounding in hierarchical models},
  author={Nobre, Widemberg S and Schmidt, Alexandra M and Pereira, Jo{\~a}o BM},
  journal={International Statistical Review},
  volume={89},
  number={2},
  pages={302--322},
  year={2021},
  publisher={Wiley Online Library}
}

@article{wang2020prediction,
  title={On prediction properties of kriging: Uniform error bounds and robustness},
  author={Wang, Wenjia and Tuo, Rui and Jeff Wu, CF},
  journal={Journal of the American Statistical Association},
  volume={115},
  number={530},
  pages={920--930},
  year={2020},
  publisher={Taylor \& Francis}
}

@article{zammit2022deep,
  title={Deep compositional spatial models},
  author={Zammit-Mangion, Andrew and Ng, Tin Lok James and Vu, Quan and Filippone, Maurizio},
  journal={Journal of the American Statistical Association},
  volume={117},
  number={540},
  pages={1787--1808},
  year={2022},
  publisher={Taylor \& Francis}
}

@article{kleiber2012nonstationary,
  title={Nonstationary modeling for multivariate spatial processes},
  author={Kleiber, William and Nychka, Douglas},
  journal={Journal of Multivariate Analysis},
  volume={112},
  pages={76--91},
  year={2012},
  publisher={Elsevier}
}

@article{paciorek2006spatial,
  title={Spatial modelling using a new class of nonstationary covariance functions},
  author={Paciorek, Christopher J and Schervish, Mark J},
  journal={Environmetrics},
  volume={17},
  number={5},
  pages={483--506},
  year={2006},
  publisher={Wiley Online Library}
}

@article{sampson1992nonparametric,
  title={Nonparametric estimation of nonstationary spatial covariance structure},
  author={Sampson, Paul D and Guttorp, Peter},
  journal={Journal of the American Statistical Association},
  volume={87},
  number={417},
  pages={108--119},
  year={1992},
  publisher={Taylor \& Francis}
}

@article{bachoc2022asymptotically,
  title={Asymptotically equivalent prediction in multivariate geostatistics},
  author={Bachoc, Fran{\c{c}}ois and Porcu, Emilio and Bevilacqua, Moreno and Furrer, Reinhard and Faouzi, Tarik},
  journal={Bernoulli},
  volume={28},
  number={4},
  pages={2518--2545},
  year={2022},
  publisher={Bernoulli Society for Mathematical Statistics and Probability}
}

@phdthesis{yu2022parametric,
  title={Parametric estimation in spatial regression models},
  author={Yu, Nathan},
  year={2022},
  school={University of Maryland, College Park}
}

@article{gelfand2003proper,
  title={Proper multivariate conditional autoregressive models for spatial data analysis},
  author={Gelfand, Alan E and Vounatsou, Penelope},
  journal={Biostatistics},
  volume={4},
  number={1},
  pages={11--15},
  year={2003},
  publisher={Oxford University Press}
}

@article{zhang2004inconsistent,
  title={Inconsistent estimation and asymptotically equal interpolations in model-based geostatistics},
  author={Zhang, Hao},
  journal={Journal of the American Statistical Association},
  volume={99},
  number={465},
  pages={250--261},
  year={2004},
  publisher={Taylor \& Francis}
}

@article{kanagawa2018gaussian,
  title={Gaussian processes and kernel methods: A review on connections and equivalences},
  author={Kanagawa, Motonobu and Hennig, Philipp and Sejdinovic, Dino and Sriperumbudur, Bharath K},
  journal={arXiv preprint arXiv:1807.02582},
  year={2018}
}

@article{gilbert2025consistency,
  author = {Gilbert, Brian and Ogburn, Elizabeth L and Datta, Abhirup},
  title = {Consistency of common spatial estimators under spatial confounding},
  journal = {Biometrika},
  volume = {112},
  number = {2},
  year = {2025},
  month = dec,
  issn = {1464-3510},
  url = {https://doi.org/10.1093/biomet/asae070},
  eprint = {https://academic.oup.com/biomet/article-pdf/112/2/asae070/61257972/asae070.pdf}
}

@article{woodward2024instrumental,
  title={An instrumental variables framework to unite spatial confounding methods},
  author={Woodward, Sophie M and Tec, Mauricio and Dominici, Francesca},
  journal={arXiv preprint arXiv:2411.10381},
  year={2024}
}

@article{wu2025spatial,
  title={Spatial confounding in multivariate areal data analysis},
  author={Wu, Kyle Lin and Banerjee, Sudipto},
  journal={arXiv preprint arXiv:2505.07232},
  year={2025}
}

@article{anderes2010consistent,
  title={On the consistent separation of scale and variance for Gaussian random fields},
  author={Anderes, Ethan},
  journal={The Annals of Statistics},
  volume={38},
  number={2},
  pages={870--893},
  year={2010}
}

@book{da2014stochastic,
  title={Stochastic equations in infinite dimensions},
  author={Da Prato, Giuseppe and Zabczyk, Jerzy},
  year={2014},
  publisher={Cambridge University Press}
}

@book{zolotarev1986one,
  title={One-dimensional stable distributions},
  author={Zolotarev, Vladimir M},
  volume={65},
  year={1986},
  publisher={American Mathematical Society}
}

@book{wackernagel2003multivariate,
  title={Multivariate geostatistics: An introduction with applications},
  author={Wackernagel, Hans},
  year={2003},
  publisher={Springer Science \& Business Media}
}

@article{garoni2002levy,
  title={L{\'e}vy flights: exact results and asymptotics beyond all orders},
  author={Garoni, Timothy M and Frankel, Norman E},
  journal={Journal of Mathematical Physics},
  volume={43},
  number={5},
  pages={2670--2689},
  year={2002},
  publisher={American Institute of Physics}
}

@article{nolan2020univariate,
  title={Univariate stable distributions},
  author={Nolan, John P},
  journal={Springer Series in Operations Research and Financial Engineering},
  volume={10},
  pages={978--3},
  year={2020},
  publisher={Springer}
}

@article{iosevich2015exponential,
  title={Exponential bases, {P}aley--{W}iener spaces and applications},
  author={Iosevich, Alex and Mayeli, Azita},
  journal={Journal of Functional Analysis},
  volume={268},
  number={2},
  pages={363--375},
  year={2015},
  publisher={Elsevier}
}

@article{li2023inference,
  title={Inference for Gaussian processes with Mat{\'e}rn covariogram on compact Riemannian manifolds},
  author={Li, Didong and Tang, Wenpin and Banerjee, Sudipto},
  journal={Journal of Machine Learning Research},
  volume={24},
  number={101},
  pages={1--26},
  year={2023}
}

@article{tang2021identifiability,
  title={On identifiability and consistency of the nugget in {G}aussian spatial process models},
  author={Tang, Wenpin and Zhang, Lu and Banerjee, Sudipto},
  journal={Journal of the Royal Statistical Society Series B: Statistical Methodology},
  volume={83},
  number={5},
  pages={1044--1070},
  year={2021},
  publisher={Oxford University Press}
}

@article{lim2009gaussian,
  title={Gaussian fields and {G}aussian sheets with generalized {C}auchy covariance structure},
  author={Lim, SC and Teo, Lee P},
  journal={Stochastic Processes and Their Applications},
  volume={119},
  number={4},
  pages={1325--1356},
  year={2009},
  publisher={Elsevier}
}

@article{garoni2002d,
  title={d-dimensional {L}{\'e}vy flights: {E}xact and asymptotic},
  author={Garoni, TM and Frankel, NE},
  journal={Journal of Mathematical Physics},
  volume={43},
  number={10},
  pages={5090--5107},
  year={2002},
  publisher={American Institute of Physics}
}

@article{gneiting2004stochastic,
  title={Stochastic models that separate fractal dimension and the {H}urst effect},
  author={Gneiting, Tilmann and Schlather, Martin},
  journal={SIAM Review},
  volume={46},
  number={2},
  pages={269--282},
  year={2004},
  publisher={SIAM}
}

@article{apanasovich2010cross,
  title={Cross-covariance Functions for multivariate random fields based on latent dimensions},
  author={Apanasovich, Tatiyana V and Genton, Marc G},
  journal={Biometrika},
  volume={97},
  number={1},
  pages={15--30},
  year={2010},
  publisher={Oxford University Press}
}

@book{ibragimov2012gaussian,
  title={Gaussian random processes},
  author={Ibragimov, Ildar Abdulovich and Rozanov, Yurii Antol'evich},
  volume={9},
  year={2012},
  publisher={Springer Science \& Business Media}
}

@article{zastavnyi2006some,
  title={On some properties of {B}uhmann functions.},
  author={Zastavnyi, Viktor},
  journal={Ukrainian Mathematical Journal},
  volume={58},
  number={8},
  year={2006}
}

@techreport{druckenmiller2018accounting,
  title={Accounting for unobservable heterogeneity in cross section using spatial first differences},
  author={Druckenmiller, Hannah and Hsiang, Solomon},
  year={2018},
  institution={National Bureau of Economic Research}
}

@article{skorokhod1973absolute,
  title={On absolute continuity of measures corresponding to homogeneous {G}aussian fields},
  author={Skorokhod, Anatolii Volodimirovich and Yadrenko, Mikha{\i}lo Iosipovich},
  journal={Theory of Probability \& Its Applications},
  volume={18},
  number={1},
  pages={27--40},
  year={1973},
  publisher={SIAM}
}

@article{ruiz2015equivalence,
  title={Equivalence of {G}aussian measures of multivariate random fields},
  author={Ruiz-Medina, MD and Porcu, E},
  journal={Stochastic Environmental Research and Risk Assessment},
  volume={29},
  pages={325--334},
  year={2015},
  publisher={Springer}
}

@article{dupont,
  url = {https://doi.org/10.1111/biom.13656},
  year = {2022},
  month = mar,
  publisher = {Wiley},
  volume = {78},
  number = {4},
  pages = {1279--1290},
  author = {Emiko Dupont and Simon N. Wood and Nicole H. Augustin},
  title = {{Spatial+: A novel approach to spatial confounding}},
  journal = {Biometrics}
}

@article{paciorek,
author = {Christopher J. Paciorek},
title = {The importance of scale for spatial-confounding bias and precision of spatial regression estimators},
volume = {25},
journal = {Statistical Science},
number = {1},
publisher = {Institute of Mathematical Statistics},
pages = {107 -- 125},
year = {2010},
URL = {https://doi.org/10.1214/10-STS326}
}

@article{page,
  url = {https://doi.org/10.1111/sjos.12275},
  year = {2017},
  month = apr,
  publisher = {Wiley},
  volume = {44},
  number = {3},
  pages = {780--797},
  author = {Garritt L. Page and Yajun Liu and Zhuoqiong He and Donchu Sun},
  title = {Estimation and prediction in the presence of spatial confounding for spatial linear models},
  journal = {Scandinavian Journal of Statistics}
}

@article{thaden,
  url = {https://doi.org/10.1080/00031305.2017.1305290},
  year = {2018},
  month = mar,
  publisher = {Informa {UK} Limited},
  volume = {72},
  number = {3},
  pages = {239--252},
  author = {Hauke Thaden and Thomas Kneib},
  title = {Structural equation models for dealing with spatial confounding},
  journal = {The American Statistician}
}

@article{khan2,
author ={Kori Khan and Candae Berrett},
  title = {Re-thinking spatial confounding in spatial linear mixed models},
  journal={Statistical Science},
  year = {2025}
}

@article{hanks,
  url = {https://doi.org/10.1002/env.2331},
  year = {2015},
  month = feb,
  publisher = {Wiley},
  volume = {26},
  number = {4},
  pages = {243--254},
  author = {Ephraim M. Hanks and Erin M. Schliep and Mevin B. Hooten and Jennifer A. Hoeting},
  title = {{Restricted spatial regression in practice: geostatistical models,  confounding,  and robustness under model misspecification}},
  journal = {Environmetrics}
}

@article{reich,
  title={{Effects of residual smoothing on the posterior of the fixed effects in disease-mapping models}},
  author={Reich, Brian J and Hodges, James S and Zadnik, Vesna},
  journal={Biometrics},
  volume={62},
  number={4},
  pages={1197--1206},
  year={2006},
  publisher={Wiley Online Library}
}

@article{wakefield,
  url = {https://doi.org/10.1093/biostatistics/kxl008},
  year = {2006},
  month = jun,
  publisher = {Oxford University Press ({OUP})},
  volume = {8},
  number = {2},
  pages = {158--183},
  author = {J. Wakefield},
  title = {{Disease mapping and spatial regression with count data}},
  journal = {Biostatistics}
}

@article{clayton,
  title={{Spatial correlation in ecological analysis}},
  author={Clayton, David G and Bernardinelli, L and Montomoli, C},
  journal={International Journal of Epidemiology},
  volume={22},
  number={6},
  pages={1193--1202},
  year={1993},
  publisher={Oxford University Press}
}

@article{hodges,
author = {James S. Hodges and Brian J. Reich},
title = {Adding spatially-correlated errors can mess up the fixed effect you love},
journal = {The American Statistician},
volume = {64},
number = {4},
pages = {325-334},
year  = {2010},
publisher = {Taylor \& Francis},
URL = { 
        https://doi.org/10.1198/tast.2010.10052
    
},
eprint = { 
        https://doi.org/10.1198/tast.2010.10052
    
}

}

@article{khan,
  title={Restricted spatial regression methods: Implications for inference},
  author={Khan, Kori and Calder, Catherine A},
  journal={Journal of the American Statistical Association},
  volume={117},
  number={537},
  pages={482--494},
  year={2022},
  publisher={Taylor \& Francis}
}

@article{zimmerman,
  title={On deconfounding spatial confounding in linear models},
  author={Zimmerman, Dale L and Ver Hoef, Jay M},
  journal={The American Statistician},
  volume={76},
  number={2},
  pages={159--167},
  year={2022},
  publisher={Taylor \& Francis}
}

@article{guan,
    author = {Guan, Yawen and Page, Garritt L and Reich, Brian J and Ventrucci, Massimo and Yang, Shu},
    title = "{Spectral adjustment for spatial confounding}",
    journal = {Biometrika},
    volume = {110},
    number = {3},
    pages = {699-719},
    year = {2022},
    month = {12},
    issn = {1464-3510},
    url = {https://doi.org/10.1093/biomet/asac069},
    eprint = {https://academic.oup.com/biomet/article-pdf/110/3/699/51111484/asac069.pdf},
}

@article{papadogeorgou,
    author = {Papadogeorgou, Georgia and Choirat, Christine and Zigler, Corwin M},
    title = "{Adjusting for unmeasured spatial confounding with distance adjusted propensity score matching}",
    journal = {Biostatistics},
    volume = {20},
    number = {2},
    pages = {256-272},
    year = {2018},
    month = Jan,
    issn = {1465-4644},
    url = {https://doi.org/10.1093/biostatistics/kxx074},
    eprint = {https://academic.oup.com/biostatistics/article-pdf/20/2/256/28033003/kxx074.pdf},
}

@article{gilbert,
  title={A causal inference framework for spatial confounding},
  author={Gilbert, Brian and Datta, Abhirup and Casey, Joan A and Ogburn, Elizabeth L},
  journal={arXiv preprint arXiv:2112.14946},
  year={2021}
}

@article{bolin2025spatial,
  title={Spatial self-confounding: Smoothness-related estimation bias in spatial regression models},
  author={Bolin, David and Wallin, Jonas},
  journal={Biometrika},
  pages={(In press)},
  year={2025},
  publisher={Oxford University Press}
}

@book{stein1999interpolation,
  title={Interpolation of Spatial Data: Some Theory for Kriging},
  author={Stein, Michael L},
  year={1999},
  publisher={Springer Science \& Business Media}
}

@article{gneiting2010matern,
  title={Mat{\'e}rn cross-covariance functions for multivariate random fields},
  author={Gneiting, Tilmann and Kleiber, William and Schlather, Martin},
  journal={Journal of the American Statistical Association},
  volume={105},
  number={491},
  pages={1167--1177},
  year={2010},
  publisher={Taylor \& Francis}
}

\newpage
\begin{center}
	\LARGE{Supplementary materials}
\end{center}

\newtheorem{proposition}{Proposition}
\newtheorem{Hlemma}[lemma]{Lemma}
\renewcommand\thesection{S\arabic{section}}
\renewcommand\theequation{S\arabic{equation}}
\renewcommand\thelemma{S\arabic{lemma}}
\renewcommand\thefigure{S\arabic{figure}}
\renewcommand\thetable{S\arabic{table}}
\setcounter{theorem}{0}
\counterwithout{theorem}{section}
\renewcommand{\thetheorem}{S\arabic{theorem}}
\renewcommand\theproposition{S\arabic{proposition}}
\renewcommand\thelemma{S\arabic{lemma}}
\setcounter{figure}{0}
\setcounter{section}{0}
\setcounter{equation}{0}
\setcounter{table}{0}
\setcounter{lemma}{0}
\setcounter{theorem}{0}

\renewcommand{\theHsection}{supp.section.\arabic{section}}
\renewcommand{\theHequation}{supp.equation.\arabic{equation}}
\renewcommand{\theHfigure}{supp.figure.\arabic{figure}}
\renewcommand{\theHtable}{supp.table.\arabic{table}}
\renewcommand{\theHtheorem}{supp.theorem.\arabic{theorem}}
\renewcommand{\theHlemma}{supp.lemma.\arabic{lemma}}
\renewcommand{\theHproposition}{supp.proposition.\arabic{proposition}}

\section{Some considerations for estimation of $\beta$ in practice}\label{sec:est}

The focus of this manuscript is primarily on characterizing conditions for the consistent estimability of the regression coefficient $\beta$ in the presence of unmeasured spatial confounding. We further show that under the sharp conditions for consistent estimability, the difference- or Laplacian-based estimator is consistent. We now discuss the practical importance of our theory and the strengths and limitations of our estimators compared to the generalized least squares (GLS) estimator.

\subsection{Feasibility of estimation and order of differencing}\label{sec:guidance} 
We first discuss how we can use the observed data $(X,Y)$ to assess the feasibility of consistently estimating $\beta$ and, if feasible, how to determine the order of differencing/Laplacian from the data. 

We note that through all our theory results, we have shown that when the sharp conditions for consistent estimability are met (except possibly the boundary case), the differencing or Laplacian-based estimator of suitable order is consistent. This estimator is non-parametric as it does not require any knowledge of the parameters of the true covariance family of $(X,W)$. Even knowledge of the covariance class is not needed, as long as they belong to Assumption \ref{eq:K.assump}, which, as we have seen in Section \ref{sec:examples}, holds for Matérn, powered exponential, Cauchy, or their linear combinations. We have also shown in Section \ref{sec:nsandng} that the estimator is robust to popular forms of non-stationarity or non-Gaussianity. 

If $(X,W)$ satisfies Assumption \ref{eq:K.assump}, then so does $Y$. So, using observed $(X,Y)$, one can consistently estimate the exponents of the principal irregular terms for both $Y$ and $X$ (see \cite{zhu2002} for $\alpha<2$; similar results hold for larger $\alpha$ by considering higher order differences).
If $\alpha_{YY}$ is estimated to be less than $\alpha_{XX} - d$, then we know from our theory that one cannot consistently estimate $\beta$. If that is not the case, then we can use the OLS estimator based on $p^{th}$ order differences for the smallest $p$ such that $2p$ is greater than the estimated $\alpha_{XX}$ (or $m^{th}$ order discrete Laplacians for the smallest $m$ such that $4m$ is greater than the estimated $\alpha_{XX}$). This should give us a way to consistently estimate $\beta$ in a fully data-driven manner, contingent on the assumption that $\alpha_{XW} > \alpha_{XX}$, which is required for the existence of a consistent estimate of $\beta$.
Thus, our theory provides practical guidance on data-driven assessment of the feasibility of estimation of $\beta$ and the order of differencing required, if feasible. The effectiveness of this strategy in real-world settings needs further study. 

\subsection{GLS estimator}\label{sec:gls}
An exhaustive comparison among different estimators of $\beta$ under spatial confounding is a separate study in its own right. Nonetheless, some discussion is merited regarding the generalized least squares (GLS) estimator, which is commonly used to estimate $\beta$ in the presence of dependence and/or confounding. 

As discussed in the main text, there are important existing results on the consistency of GLS. \cite{wang2020prediction} proved that the GLS estimator is inconsistent when $X$ is a fixed function of space that lies in the reproducing kernel Hilbert space of the kernel of $W$.
In \cite{bolin2025spatial}, for the setting where the smoothing operator is identity, i.e., the true fixed function $X$ is observed, this GLS estimator is proved to be consistent if $X$ is not in the RKHS of $\textrm{Cov}(W)$. \cite{yu2022parametric} obtain similar conditions based on sample paths of a random $X$ (GRF) that is independent of $W$.

These theoretical results assume either $X$ to be a fixed function of space or a GRF independent of $W$, implying that there is no statistical correlation between $X$ and the error process $W$. The GLS estimator they study is essentially an {\em `oracle GLS'}, using the true covariance of $W$, which is the unbiased minimum variance estimator and the maximum likelihood estimator. This is no longer true when $X$ is correlated with $W$ (as the observed $X$ is informative about the distribution of $W$) and/or when the covariance function class used for the GLS is misspecified, which is likely to be the case in any real-world application. 

We now present a theoretical counter-example showing that when there is spatial confounding, i.e., $X$ is correlated with $W$, then the GLS estimator from a misspecified covariance family is inconsistent in a setting where the Laplacian-based estimator is consistent. 

\begin{proposition}[Inconsistency of GLS with misspecified covariance family]\label{prop:gls} Let $(X(s),W(s))^T$ be a zero mean stationary bivariate GRF on $\mathbb R$ with covariance satisfying Assumption \ref{eq:K.assump} for some $L>0$ and with $c_{12} \neq 0$, $\alpha_{12} > \alpha_{11}$ and $\alpha_{11} < \alpha_{22} +1$.
	Let $X$ and $Y=X\beta + W$ be observed at the grid points $s_i = i h$ with $h = \blue L/n$ and $0 \le i \le n$. Let $\Sigma$ be the $(n+1)\times (n+1)$ covariance matrix over the grid   based on the exponential covariance function, i.e., 
	\[
	\Sigma_{ij} = \exp\!\Big(-\lambda\,|s_i-s_j|\Big),\qquad \lambda>0.
	\] 
	If $\alpha_{11} > 2$,
	then for all $\lambda \in (0,\infty) \setminus S$ where $S \subseteq \mathbb (0,\infty)$ with $|S|\leq 12$, the GLS estimator based on the working covariance matrix $\Sigma$, i.e., $\hat\beta_{GLS}= \frac{X^T \Sigma^{-1} Y}{X^T \Sigma^{-1} X}$ is inconsistent for $\beta$. 
\end{proposition}

The result shows that the GLS with an exponential working covariance matrix will lead to an inconsistent estimate of $\beta$ as long as there is spatial confounding ($c_{12} \neq 0$) and $2 < \alpha_{11} < \min\{\alpha_{12},\alpha_{22}+1\}$.
From the second part of the inequality, we know that the difference-based estimator of suitable order will be consistent. Yet, the GLS is inconsistent here when using the exponential covariance family. 
This includes scenarios where the covariance misspecification is infinitesimally small.
Consider the case where for a small $\delta > 0$, $(X,W)$ is a bivariate Matérn where $\nu_{XX}=1 + \delta/2$, $\nu_{WW}=\frac 12 + \delta$, $\nu_{XW} > \nu_{XX}$. Here, the exponential covariance is only slightly misspecified, using $\nu=1/2$ in the Matérn, where the true $\textrm{Cov}(W)$ is based on $\nu_{WW}=\frac 12 + \delta$. However, in this scenario we will have the true $\alpha_{22}=1+\blue 2\delta$, $\alpha_{11}=2+\delta$, $\alpha_{12}> \alpha_{11}$ and we know from Proposition \ref{prop:gls} that the GLS estimator using the exponential working covariance will be inconsistent for all possible positive values of $\lambda$ except possibly a small set $S$ with maximum cardinality $12$ ($S$ consists of possible zeroes in the mean and variances of the numerator and denominator in the bias expression for the GLS, which are polynomials in $\lambda$). 

This inconsistency of the GLS using the exponential  covariance is problematic as the exponential family is popularly used as the working covariance model for GLS in place of the more general Matérn family because this choice circumvents the time-intensive evaluation of the Bessel function featuring in the general Matérn family, and our result shows that this practice is not advisable if estimation of $\beta$ is the goal under spatial confounding. 

These findings are supported in our numerical experiments in Section \ref{sec:simgls}. The Proposition is based on the exponential covariance since the GLS precision matrix on a grid is then available in closed form, facilitating the tractability of the GLS expression. However, we conjecture that similar inconsistency results will hold for spatial domains in $\mathbb R^d$ and for other undersmoothed working covariance classes of fixed smoothness.

Unlike the GLS estimator, the differencing-based estimator of suitable order is a robust, non-parametric estimator that does not rely on knowing or estimating the functional form of the covariance family or its parameters. It is consistent for every scenario where $\beta$ can be consistently estimable. While similar results possibly will hold also for GLS with a suitably chosen working covariance family and estimated parameters, this is yet to be established for the GLS when $(X,W)$ is a bivariate correlated random field, and is generally challenging due to having to account for the covariance parameter estimation involved in the GLS covariance matrix. The differencing-based estimator is also scalable to very large datasets, where the exact GLS is challenging to implement due to requiring large matrix inversions.

\subsection{Irregular designs}\label{sec:irreg}

For irregular designs, instead of using unweighted first differences or Laplacians, we propose using spacing-weighted differences. The following result shows that this leads to a consistent estimator with the same rates of convergence for the numerator and denominator terms (up to constants). For simplicity,  we have proved the result for one dimension. The results should be extendable to higher dimensions, as we discuss after the theorem. 

\begin{theorem}\label{th:irreg-second}
	Let $(X,W)$ denote a bivariate stationary GRF on $\mathbb R$ with covariance
	$(K_{k\ell})_{k,\ell=1,2}$ such that each $K_{k\ell}$ satisfies Assumption \ref{eq:K.assump} for some $L>0$ 
	with parameter $\alpha_{k\ell}>0$, and assume $\alpha_{11}<\alpha_{12}$.
	Let $Y(s)=X(s)\beta+W(s)$.
	For each $n$, let
	$\calS_n=\{s^{(n)}_0=0<s^{(n)}_1<\cdots<s^{(n)}_n=L\}$ be an irregular grid with spacings
	$h^{(n)}_i=s^{(n)}_{i+1}-s^{(n)}_i$ satisfying
	\[
	\frac{m}{n}\le h^{(n)}_i\le\frac{M}{n}\qquad(0\le i\le n-1)
	\]
	for constants $0<m\le M<\infty$ independent of $n$.
	
	Let the first–order spacings be $h_i=h^{(n)}_i$, and define the
	second-order spacings as 
	\[
	\tilde h_i=\frac{h_{i-1}+h_i}{2},\qquad
	\tilde W=\mathrm{diag}(\tilde h_1,\ldots,\tilde h_{n-1}). 
	\]
	Define the spacing-weighted second-order difference (or first Laplacian) operator $\Delta X$ as 
	\[
	(\Delta X)_i
	=
	\tilde h_i^{-1}\Bigg(
	\frac{X(s_{i+1})-X(s_i)}{h_i}
	-
	\frac{X(s_i)-X(s_{i-1})}{h_{i-1}}
	\Bigg). 
	\]
	If $\alpha_{11}<4$, $\alpha_{11} \neq 2$ and $\alpha_{22}>\alpha_{11}-1$, then 
	\[
	\mathrm{OLS}_n(\Delta X,\Delta Y)\;\rightarrow\beta \mbox{ in probability } 
	\text{as }n\to\infty.
	\]
\end{theorem}

In the theorem, we considered a larger range $(0,4)$ for the exponent $\alpha_{11}$ of the principal irregular term for $X$. This is the more interesting setting requiring second-order spacing-weighted differences, as 
for the setting $\alpha_{11} \in (0,2)$, even unweighted differences can work. We present numerical experiments in Section \ref{sec:simirreg} which show that for $\alpha_{11} > 2$, none of the unweighted first- or second-difference estimators, or the spacing-weighted first-difference estimator works, but the spacing-weighted second-difference estimator performs well. 

We believe the idea can be extended to irregular designs in $\mathbb R^d$ as follows. For irregular designs in $[0,1]^d$, the idea is to estimate $\beta$ by OLS after applying a discrete
second-order operator. Specifically, form a sparse matrix $A\in\mathbb R^{n\times n}$ from a
local neighbor graph on the design points $\{s_i\}_{i=1}^n$ and use
\[
\widehat\beta_{A,n}
=
\mathrm{OLS}_n(AX,AY)
=
\frac{(AX)^T(AY)}{(AX)^T(AX)}.
\]
The key requirement on $A$ is that, at interior points, each row acts like a Laplacian and
therefore cancels the low-order analytic part of the covariance expansion. This can be done as follows. For an interior node $s_i$, let $N(i)$ denote a set of local graph neighbors and write the
offsets $r_{ij}=s_j-s_i$. Choose weights $\{a_{ij}: j\in N(i)\}$ and for any function $f$ evaluated at $s_\blue 1,\ldots,s_n$, define the matrix $A$ as
\[
(Af)_i
=
\sum_{j\in N(i)} a_{ij}\,\bigl(f(s_j)-f(s_i)\bigr).
\]
Impose the moment cancellation conditions 
\[
\sum_{j\in N(i)} a_{ij}=0,
\qquad
\sum_{j\in N(i)} a_{ij}\, r_{ij}=0,
\qquad
\sum_{j\in N(i)} a_{ij}\, r_{ij} r_{ij}^T
=
c_i I_d,
\]
for some scalar $c_i>0$. These conditions make $(Af)_i$ behave like a scaled Laplacian
$\Delta f(s_i)$ and annihilate the constant and linear Taylor expansion terms for $f$, so that when $A$ is applied to both $X$ and $Y$ the remaining leading contribution comes from the principal
irregular term governed by $\alpha_{11}$, yielding matching orders for the numerator and
denominator in $\mathrm{OLS}_n(AX,AY)$.

\section{Additional simulation experiments}\label{sec:addsim}

\subsection{Comparisons with the GLS estimator}\label{sec:simgls}
We present an empirical study comparing the difference-based estimators with three GLS estimators --- the GLS estimator that uses the true marginal covariance of $W$, a GLS estimator where the covariance family is correctly specified as Matérn and all parameters are estimated, and a GLS estimator where the working covariance uses the exponential family where parameters are estimated.

We generate $X$ and $W$ from a bivariate Matérn stationary GRF where the marginal smoothness for $X$ is 1.2 and for $W$ is 0.9, and the cross-smoothness is $1.5$. This constitutes a situation where the theory says that the second differences-based estimator is consistent. Hence, in addition to the GLS estimators, we also present the first- and second-differences based OLS estimators.

\begin{figure}[h]
	\centering
	\includegraphics[width=0.8\textwidth]{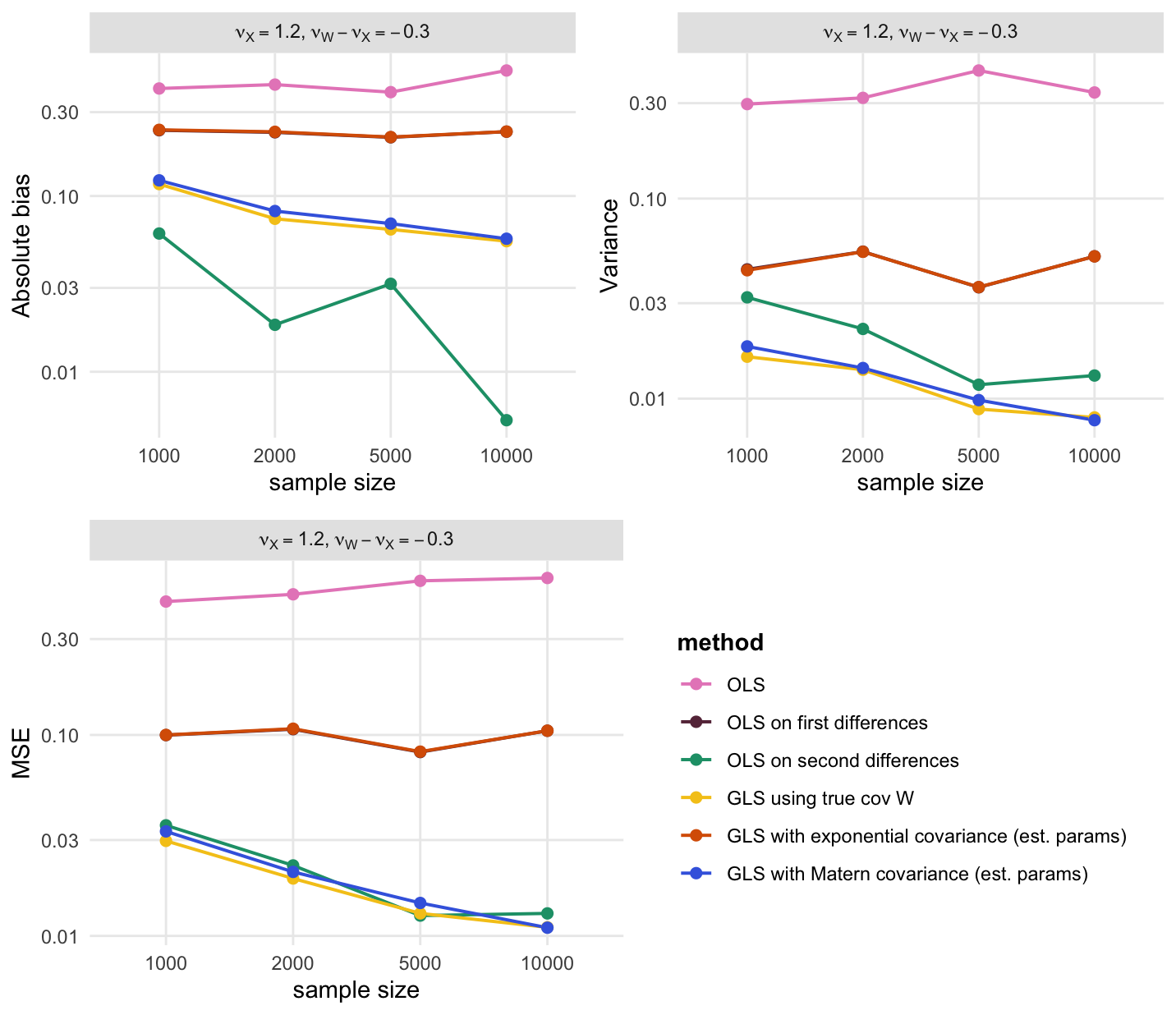}
	\caption{Summary metrics for difference-based and GLS estimators.}
	\label{fig:gls_metrics_settings2}
\end{figure}

The results are as follows. The GLS estimator using a working covariance matrix from the exponential family (the setting of Proposition \ref{prop:gls}) performs poorly, with both bias and variance converging to non-zero limits, aligning with what is expected from the theory. The first-differences based estimator is almost identical in performance to this GLS estimator (hence, its lines are not visible in the figure) with non-decreasing bias or variance. This is not surprising as the first-difference estimator is essentially GLS (on contrasts) based on a Brownian motion covariance, and from our theoretical results, we have seen that when $\alpha_{11} > 2$, we need second or higher order differencing for consistent estimation of $\beta$. 
The GLS using the true marginal covariance of $W$ and the GLS with Matérn covariance with all parameters estimated perform on par with the second-difference based estimator, which is the provably consistent estimator for this setting. The latter has slightly less bias and slightly higher variance than these GLS estimators, and all three estimators lead to very similar MSE. 

\subsection{Simulations for irregular design}\label{sec:simirreg}
We present results from a simulation experiment based on irregular design  on $[0,1]$. For each
$n \in \{100,500,1000,2000\}$, the locations are obtained by subsampling $n$
points without replacement from a fine grid of size approximately $5n$, and
then sorting them. So the spacings are random and nonuniform.

At these locations, a zero mean bivariate GRF $(X,W)$ is
simulated with Mat\'ern covariances sharing range parameter $0.2$ and
marginal variances $\sigma_1^2=\sigma_2^2=1$. The smoothness parameters are
$\nu_X=\nu_{11}=1.5$, $\nu_W=\nu_{22}=\nu_{11}+\delta_{22}=2.0$, and the
cross smoothness is $\nu_{12}=\nu_{11}+0.25=1.75$, with cross correlation
parameter $\rho=\min\{0.5,\sqrt{\nu_{11}\nu_{22}}/\nu_{12}\}$. Observations follow the data generation process $Y=\beta X + W$ with $\beta=2$.

The results are displayed in Figure \ref{fig:irreg}. The unweighted (equal-spacing) first and second differences based OLS estimators do not work well, nor does the first-order spacing-weighted differencing-based one (as $\alpha_{11} > 2$). The second-order spacing weighted differencing-based estimator works well. As $\alpha_{11} < 4$, this aligns with the result of Theorem \ref{th:irreg-second} and shows that the idea of differencing or Laplacian-based estimation generalizes well to practical settings. 

\begin{figure}
	\centering
	\includegraphics[width=0.9\linewidth]{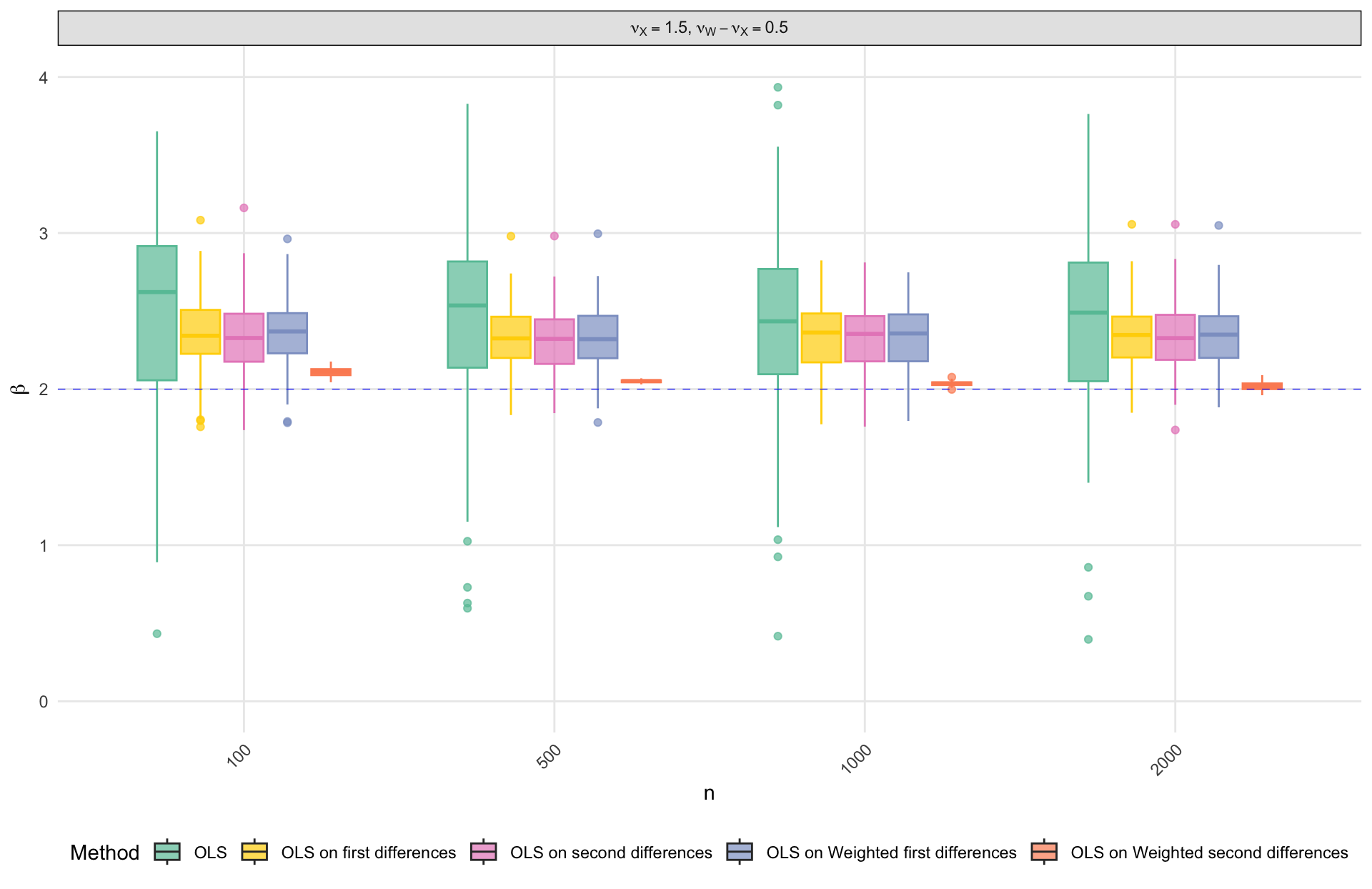}
	\caption{True $\beta=2$ (dotted line) and estimated $\beta$ for irregular designs of different sample sizes.}
	\label{fig:irreg}
\end{figure}

\subsection{Simulations for noisy data}\label{sec:simnoise}

We also considered a setting with measurement error or noise in the observed outcome to assess the local averaging-and-differencing-based estimator, which was proved to be consistent for the noisy case in Theorem \ref{thm:error}. We use a nested gridded design as in Figure \ref{fig:lad} but in $\mathbb R$. For each coarse resolution $n \in \{50,100,200\}$ and averaging index
$\rho \in \{0.2,0.3,0.4,0.5\}$, locations are generated by first taking the coarse
grid to be $\{i/n: i=0,\dots,n\}$. Then, around each point on this coarse grid, we add a symmetric
subgrid of size $2\lceil n^{\rho}\rceil+1$ at a very fine spacing, yielding a
clustered, irregular design with total size
$N=(n+1)\bigl(2\lfloor n^{\rho}\rceil+1\bigr)$. At these locations, a zero mean
bivariate GRF $(X,W)$ is simulated with Mat\'ern marginals
having range $0.2$, variances $\sigma_1^2=\sigma_2^2=1$, smoothness
$\nu_X=0.5$ and $\nu_W=0.8$, and cross smoothness $\nu_{XW}=0.75$, with
cross correlation parameter $\rho_{XW}=0.5$. The response is generated as
$Y=\beta X + W + \varepsilon$ with $\beta=2$ and independent noise
$\varepsilon \sim N(0,1)$, and the experiment uses $100$ Monte Carlo
replicates for each value of $\rho$.

The results are shown in Figure \ref{fig:noise}. Across all panels, plain OLS is biased upward relative to the truth
$\beta=2$. OLS on first differences reduces the bias but remains more
variable, while OLS on second differences shows the largest variability and
outliers, consistent with differencing amplifying observation noise. The
`averaging-and-differencing' estimator, which averages within each cluster
(block) and then differences the block means, is the most stable and typically
closest to $\beta=2$, with dispersion generally shrinking as $n$ increases
and as $\rho$ increases, consistent with the theory.

\begin{figure}
	\centering
	\includegraphics[width=0.9\linewidth]{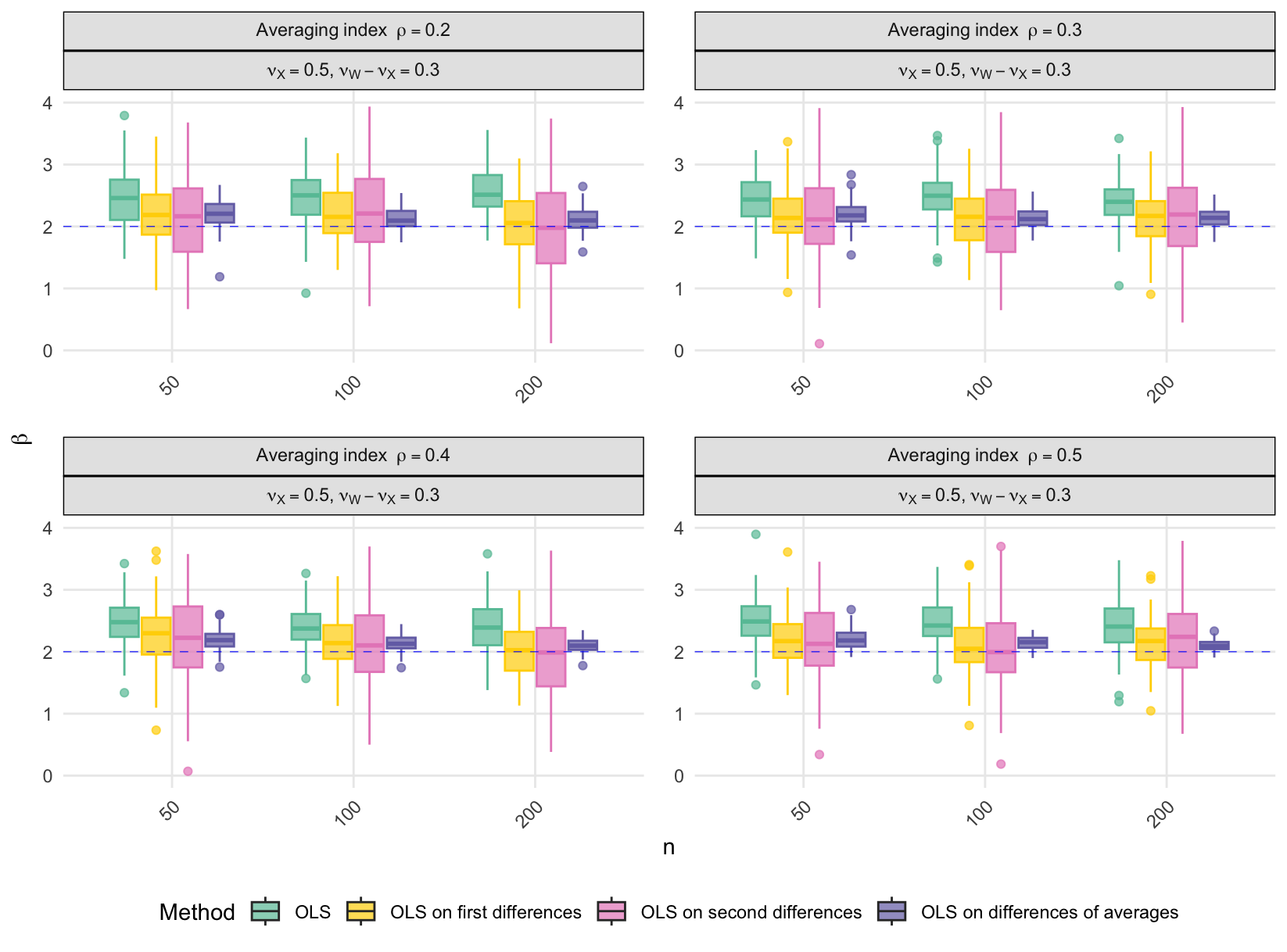}
	\caption{True $\beta=2$ (dotted line) and estimated $\beta$ for data observed with measurement error.}
	\label{fig:noise}
\end{figure}

\newpage
\section{Main proofs}\label{sec:mainproofs} We provide the proofs of some of the main results here that illustrate the central techniques used for both consistent estimability (orthogonality) and equivalence. 

\subsection{Proofs for estimability in one dimension}\label{sec:proof1d}
The proofs of the general results on estimability of $\beta$ in one dimension (Theorems \ref{th:beta.con},  \ref{th:suff}, and Corollary \ref{cor:boundary}) are provided here.

\begin{proof}[\textbf{Proof of Theorem \ref{th:beta.con}}]
	We first prove the case for $p=1$, $0<\alpha_{11}<2$. From the statement of the Theorem we have $\alpha_{12}=\alpha_{11}$, $c_{12}=\beta c_{11}$ and $\alpha_{11}-1 < \alpha_{22} \le \alpha_{11}$. 
	As $K$ is even from Assumption \ref{eq:K.assump}, we have
	\begin{equation*}
		E \left[\{\nabla_h^{(1)}Z_k(0)\} \{\nabla_h^{(1)}Z_\ell(0)\}\right] = \frac 1{h^2} \left[2K_{k\ell}(0)-2K_{k\ell}(h)\right]
	\end{equation*}
	for $k,\ell=1,2$.
	Since $\alpha_{k\ell} < 2$ and, from Assumption \ref{eq:K.assump}, $A_{k\ell}$ is analytic, we have 
	$2A_{k\ell}(0)-2A_{k\ell}(h) = O(h^2) = o(h^{\alpha_{k\ell}})$ and 
	$2B_{k\ell}(0)-2B_{k\ell}(h) = -2 c_{k\ell}h^{\alpha_{k\ell}} + o(h^{\alpha_{k\ell}})$. So, we have
	\begin{equation}\label{eq:num.mean}
		\frac 1 {nh^{\alpha_{11}-2}} E\sum_{j=0}^{n-1}\{\nabla_h^{(1)}Z_1(h j)\}\{\nabla_h^{(1)}Z_2(h j)\} = -2\beta c_{11} + o(1)
	\end{equation} 
	and
	\begin{equation}\label{eq:denom.mean}
		\frac 1 {nh^{\alpha_{11}-2}}  E\sum_{j=1}^n\{\nabla_h^{(1)}Z_1(h j)\}^2 = - 2 c_{11} + o(1).
	\end{equation}
	
	Note that the bivariate first difference process
	$(\nabla_h^{(1)} Z_k(h j),\nabla_h^{(1)} Z_\ell(h j))$ is stationary on the grid. For a function \(f:\mathbb{R}\to\mathbb{R}\), define the second-order discrete difference operator
	\begin{equation}\label{eq:funclap}
		\nabla_h^{(2)} f(x)
		=
		h^{-2}\{f(x+h)-2f(x)+f(x-h)\}.
	\end{equation} 
	Then $Cov(\nabla_h^{(1)} Z_k(h j),\nabla_h^{(1)} Z_\ell(h j'))=-\nabla_{h}^{(2)} K_{k\ell}(h(j-j'))$.
	Using standard properties of the multivariate normal distribution, the joint stationarity of $(Z_1,Z_2)$, and Isserlis's theorem, we have
	
	\begin{equation}\label{eq:var.sum}
		\begin{aligned}
			& \mathrm{Var}\left[ \sum_{j=0}^{n-1}\{\nabla_h^{(1)} Z_k(h j)\}\{\nabla_h^{(1)} Z_\ell(h j)\}\right] \\
			& \quad = n
			\left[
			\{\nabla_h^{(2)}K_{kk}(0)\}
			\{\nabla_h^{(2)}K_{\ell\ell}(0)\}
			+
			\{\nabla_h^{(2)}K_{k\ell}(0)\}^2
			\right] \\
			&\quad+
			2\sum_{j=1}^{n-1}(n-j)
			\left[
			\{\nabla_h^{(2)}K_{kk}(hj)\}
			\{\nabla_h^{(2)}K_{\ell\ell}(hj)\}
			+
			\{\nabla_h^{(2)}K_{k\ell}(hj)\}^2
			\right].
		\end{aligned}
	\end{equation}
	Now
	\begin{equation}\label{eq:var.bound1}
		K_{k\ell}(0)-K_{k\ell}(h) =  -c_{k\ell}h^{\alpha_{k\ell}} + o(h^{\alpha_{k\ell}})= O(h^{\alpha_{k\ell}}),
	\end{equation}
	\begin{equation}\label{eq:var.bound2}
		\nabla_h^{(2)}K_{k\ell}(0) = 2c_{k\ell}
		h^{\alpha_{k\ell}-2} + o\big(h^{\alpha_{k\ell}-2}\big)
	\end{equation}
	and, using the second-order Taylor series, for $j>0$,
	\begin{align*}
		\nabla_h^{(2)}K_{k\ell}(hj) & =  \nabla_h^{(2)}A_{k\ell}(hj) + \nabla_h^{(2)}B_{k\ell}(hj) \\
		& =  O(1) + c_{k\ell} \{2B_{k\ell}''(\xi_2) - B_{k\ell}''(\xi_1)\}
	\end{align*}
	for some $\xi_1$ and $\xi_2$ in $[hj,h(j+2)]$.
	From the form of $B''_{k\ell}$ in Assumption \ref{eq:K.assump}, it follows that for $j>0$,
	\begin{equation}\label{eq:var.bound3}
		|\nabla_h^{(2)}K_{k\ell}(hj)| = O(1) +   O(h^{\alpha_{k\ell}-2}j^{\alpha_{k\ell}-2}). 
	\end{equation}
	
	From (\ref{eq:var.bound1})--(\ref{eq:var.bound3}), we have
	\begin{align}
		\mathrm{Var}\left[  \sum_{j=0}^{n-1}\{\nabla_h^{(1)} Z_1(h j)\}^2\right] \nonumber & = nO(h^{2\alpha_{11}-4}) +  O(n^2) + O\left( \sum_{j=1}^{n-1} nh^{2\alpha_{11}-4}j^{2\alpha_{11}-4}\right) \nonumber\\
		& =  O(n^{5-2\alpha_{11}}) + O(n^2) + n^{5-2\alpha_{11}} O\left( \sum_{j=1}^{n-1} j^{2\alpha_{11}-4}\right). 
		\label{eq:var.bd1}
	\end{align}
	Now
	\[
	\sum_{j=1}^{n-1} j^{2\alpha_{11}-4} = 
	\begin{cases}
		O(1) & \mbox{if } 2\alpha_{11} < 3 \\
		O(\log n) & \mbox{if } 2\alpha_{11} = 3 \\
		O(n^{2\alpha_{11}-3}) & \mbox{if } 3 < 2\alpha_{11} < 4,
	\end{cases}
	\]
	so
	\[
	\mathrm{Var}\left[  \sum_{j=0}^{n-1}\{\nabla_h^{(1)} Z_1(h j)\}^2\right] 
	= 
	\begin{cases}
		O(n^{5-2\alpha_{11}}) & \mbox{if } 2\alpha_{11} < 3 \\
		O(n^2\log n) & \mbox{if } 2\alpha_{11} = 3 \\
		O(n^2) & \mbox{if } 3 < 2\alpha_{11} < 4.
	\end{cases}
	\] 
	Applying similar logic and using that Equation (\ref{eq:K12limit}) implies $2\alpha_{12} = 2\alpha_{11} \geq \alpha_{11}+\alpha_{22}$, we have
	\begin{align*} \mathrm{Var}\left[ \sum_{j=0}^{n-1}\{\nabla_h^{(1)} Z_1(h j)\}\{\nabla_h^{(1)} Z_2(h j)\}\right]
		= \begin{cases}    
			O\big(n^{5-\alpha_{11}-\alpha_{22}}\big) & \mbox{if } \alpha_{11}+\alpha_{22} < 3 \\
			O\big(n^{2}
			\log n\big) & \mbox{if } \alpha_{11}+\alpha_{22} = 3 \\ 
			O\big(n^{2}
			\big) & \mbox{if } 3< \alpha_{11}+\alpha_{22} < 4
		\end{cases}
	\end{align*}
	
	Normalizing both the numerator and denominator, we have 
	\begin{equation}\label{eq:varden}
		\mathrm{Var}\left[ \frac 1{nh^{\alpha_{11}-2}}  \sum_{j=0}^{n-1}\{\nabla_h^{(1)} Z_1(h j)\}^2\right] = 
		\begin{cases}
			O(n^{-1}) & \mbox{if } 2\alpha_{11} < 3 \\
			O(n^{-1}\log n) & \mbox{if } 2\alpha_{11} = 3 \\
			O(n^{2\alpha_{11}-4}) & \mbox{if } 3 < 2\alpha_{11} < 4,
		\end{cases}
	\end{equation} 
	and 
	\begin{align}\label{eq:varnum}
		&   \mathrm{Var}\left[ \frac 1{nh^{\alpha_{11}-2}} \sum_{j=0}^{n-1}\{\nabla_h^{(1)} Z_1(h j)\}\{\nabla_h^{(1)} Z_2(h j)\}\right] \nonumber \\
		& \quad = 
		\begin{cases}
			O(n^{-1+\alpha_{11}-\alpha_{22}}) & \mbox{if } \alpha_{11}+\alpha_{22} < 3 \\
			O(n^{2\alpha_{11}-4}\log n) & \mbox{if } \alpha_{11}+\alpha_{22} = 3 \\
			O(n^{2\alpha_{11}-4}) & \mbox{if } 3 < \alpha_{11}+\alpha_{22} < 4.
		\end{cases}  
	\end{align}
	Both variances converge to zero as $\alpha_{11} < 2$ and $\alpha_{22} > \alpha_{11}-1$. Hence, from (\ref{eq:num.mean}) and (\ref{eq:varnum}), as $\alpha_{11} < \alpha_{22}+1$, we have
	\begin{equation*}
		\frac{\sum_{j=0}^{n-1}\nabla_h^{(1)}Z_1(h j)\nabla_h^{(1)}Z_2(h j)}
		{nh^{\alpha_{11}}} \to -2\beta c_{11} \mbox{ in }L^2
	\end{equation*}
	and from 
	(\ref{eq:denom.mean}) and (\ref{eq:varden}), we have
	\begin{equation*}
		\frac{\sum_{j=0}^{n-1}\{\nabla_h^{(1)}Z_1(h j)\}^2}
		{nh^{\alpha_{11}}} \to -2 c_{11} \mbox{ in }L^2.
	\end{equation*}
	Hence, their ratio converges to $\beta$ in probability.  
	
	If, instead of using first differences of $Z_1$ and $Z_2$, one uses differences of order $p$, then a similar proof shows that the consistency holds for $\alpha_{11} < 2p$.
\end{proof}

\begin{proof}[Proof of Corollary \ref{cor:boundary}]
	Suppose $\alpha_{11}=2p $ for integer $p$ and $\alpha_{22} = \alpha_{11}-1$.
	As in the proof of Theorem \ref{th:beta.con}, we only give the details when $p = 1$. 
	We need to include a $\log h$ term in the analogs to (\ref{eq:num.mean}) and (\ref{eq:denom.mean}):
	\begin{equation}\label{eq:num.mean1}
		-\frac {1}{n\log h} E\sum_{j=0}^{n-1}\{\nabla_h^{(1)}Z_1(h j)\}\{\nabla_h^{(1)}Z_2(h j)\} = 2\beta c_{11} + o(1)
	\end{equation} 
	and
	\begin{equation}\label{eq:denom.mean1}
		-\frac {1} {n\log h}  E\sum_{j=0}^{n-1}\{\nabla_h^{(1)}Z_1(h j)\}^2 = 2 c_{11} + o(1).
	\end{equation}
	It is straightforward to show that
	\[
	\mathrm{Var}\left[ \frac {1}{n\log h} \sum_{j=0}^{n-1}\{\nabla_h^{(1)}Z_1(h j)\}^2 \right] \to 0
	\]
	as $n\to\infty$.
	Furthermore, since $\alpha_{22}=1$, we have from (\ref{eq:var.bound2}),
	$\nabla_h^{(2)}K_{22}(0) = O\big(h^{-1}\big)$ and, for $j>0$, $\nabla_h^{(2)}K_{22}(hj) = o\big((hj)^{-1}\big)$. Note that the last argument requires that $\alpha_{22}=1$, since otherwise $\nabla_h^{(2)}K_{22}(hj)$ is only $O\big((hj)^{-\alpha_{22}}\big)$ and not $o\big((hj)^{-\alpha_{22}}\big)$ for $j>0$.
	
	Similar to (\ref{eq:var.bd1}), we get
	\begin{align*}
		& \mathrm{Var}\left[ \sum_{j=0}^{n-1}\{\nabla_h^{(1)}Z_1(h j)\}\}\{\nabla_h^{(1)}Z_2(h j)\}\} \right] \\
		& = O\big( n\log^2 n\big) + O\big( n^2\log n)\\
		& \qquad + O\left( n\sum_{j=1}^n (1+\log^2(hj))\right) +   n\sum_{j=1}^n  (|\log hj|+1)\left\{O(1)+o\left(\frac{n}{j}\right)\right\}.
	\end{align*}
	Since $\sum_{j=1}^n \log^2(hj) = O(n)$, $\sum_{j=1}^n n/j = O(n\log n)$ and $\sum_{j=1}^n |\log hj| n/j = O(n\log^2 n)$, we have 
	\[
	\mathrm{Var}\left[ \frac {1}{n\log h} \sum_{j=0}^{n-1}\{\nabla_h^{(1)}Z_1(h j)\}\{\nabla_h^{(1)}Z_2(h j)\} \right] \to 0
	\]
	as $n\to\infty$ with the variance converging to zero at a logarithmic rate. The theorem follows for $\alpha_{11}=2$ and $\alpha_{22}=1$.
\end{proof}

\begin{proof}[\textbf{Proof of Theorem \ref{th:suff}}]
	We start with the case $\beta \neq 0$. We consider the bivariate process $(X,Y)$ on $\mathbb R$. Using the conditions of the theorem, as $\alpha_{12} > \alpha_{11}$, for $t \downarrow 0$, we have 
	$$
	\begin{aligned}
		\mbox{Cov}(Y(s+t),X(s)) &= \beta \mbox{Cov}(X(s+t),X(s)) + \mbox{Cov}(W(s+t),X(s)).
	\end{aligned}
	$$
	The irregular term in the cross-covariance is 
	$$
	\begin{aligned}
		\beta c_{11}t^{\alpha_{11}} + c_{12}t^{\alpha_{12}} + o(t^{\alpha_{11}}) + o(t^{\alpha_{12}}) &= \beta c_{11}t^{\alpha_{11}} + o(t^{\alpha_{11}}) \quad (\mbox{as } \alpha_{11}<\alpha_{12} \mbox{ and } \beta \neq 0) \\
		&= c^*_{12}t^{\alpha^*_{12}} + o(t^{\alpha^*_{12}})
	\end{aligned}
	$$ where $c^*_{12}=\beta c_{11}$ and $\alpha^*_{12}=\alpha_{11}$.  
	Similarly, we have 
	$$
	\begin{aligned}
		\mbox{Cov}(Y(s+t),Y(s)) &= \beta^2 c_{11}t^{\alpha_{11}} + 2\beta c_{12}t^{\alpha_{12}} + c_{22}t^{\alpha_{22}} + o(t^{\alpha_{11}}) + o(t^{\alpha_{22}})\\
		& = c^*_{22}t^{\alpha^*_{22}} + o(t^{\alpha^*_{22}})
	\end{aligned}
	$$
	where $\alpha^*_{22}=\min(\alpha_{11},\alpha_{22})$ as $\beta \neq 0$ and $c^*_{22}$ is the corresponding coefficient. 
	
	Letting $K^*=(K^*_{k\ell})_{\{1\leq k,\ell \leq 2\}}$ denote the covariance of $(X,Y)$, we have, as $t \downarrow 0$, 
	
	$$
	\begin{aligned}
		K^*_{k\ell} & = c^*_{k\ell}t^{\alpha^*_{k\ell}} + o(t^{\alpha^*_{k\ell}}) \mbox{ for $k,\ell = 1,2$ where }\\
		& c^*_{11}=c_{11}, \alpha^*_{11}=\alpha_{11}, \\
		& c^*_{12}=\beta c_{11}, \alpha^*_{12}=\alpha_{11}, \mbox{ and } \\
		&\alpha^*_{22}=\min(\alpha_{11},\alpha_{22}).
	\end{aligned}
	$$
	
	This $K^*$ thus satisfies Assumption \ref{eq:K.assump} and (\ref{eq:K12limit}). 
	Also, since $\alpha_{22} > \alpha_{11} - 1$ and $\alpha^*_{22}=\min(\alpha_{11},\alpha_{22})$, we have $\alpha^*_{11} - 1 < \alpha^*_{22} \leq  \alpha^*_{11}$. 
	Thus $(Z_1,Z_2) = (X,Y)$ satisfies all conditions of Theorem \ref{th:beta.con}, and $\beta$ is consistently estimable with the consistent estimator given by $OLS_n^{(p)} (Z_{1},Z_{2})= OLS_n^{(p)} (X,Y)$. 
	
	For the case $\beta=0$, choose a $\beta_0 \neq 0$, and define $Y^*= X\beta_0 + W$, and $\widetilde Y = Y+Y^* = X\beta_0 + 2W$. As $\beta_{0} \neq 0$, from the first scenario, we have $OLS_n^{(p)} (X,Y^*) \to \beta_0$ and $OLS_n^{(p)} (X,\widetilde Y) \to \beta_0$. Due to linearity of differencing, we have $\nabla_p \widetilde Y - \nabla_p Y^* = \nabla_p Y$ implying 
	$OLS_n^{(p)} (X,Y) = OLS_n^{(p)} (X,\widetilde Y) - OLS_n^{(p)} (X,Y^*) \to 0$. 
	
\end{proof}

\subsection{Proofs of equivalence}\label{sec:proofequiv}

We first prove the following Lemma.

\begin{lemma}\label{lem:integral}
	Let $b_{ij}$ be defined as in (\ref{eq:bij}). Then \begin{equation}\label{eq:polya2}
		\intrd \intrd |b_{ij}(\mu,\omega)|^2d\mu\, d\omega < \infty.
	\end{equation}.
\end{lemma}

\begin{proof}
	As $\calD=[-T,T]^d$, any $b_{ij}$ as in (\ref{eq:bij}) can be written, using change of variable $h \to -h$, as $
	b_{ij}(\mu,\omega) = (2\pi)^{-2d} \int_\calD \int_\calD \exp(-\iota a^T\mu - \iota h^T \omega)\rho_{ij}(a,-h) da\, dh$. Let $\zeta$ denote a $2d \times 1$ vector in $\mathbb R^{2d}$ stacking up $\mu$ and $\omega$, and we can write $b_{ij}(\mu,\omega) = \tilde b_{ij}(\zeta)$. Similarly, we create a $2d \times 1$ vector  $t$ in $\calD_2 = [-T,T]^{2d}$ by stacking $a$ and $h$ and write $\tilde \rho_{ij}(t)=\rho_{ij}(a,-h)$. Then  $\tilde \rho_{ij}(t) \in \calL_2(\mathbb R^{2d})$ if $\rho_{ij}(a,h) \in \calL_2(\mathbb R^d \times \mathbb R^d)$ and is zero outside  $\calD_2$. Then we have $\tilde b_{ij}(\zeta)=(2\pi)^{-2d} \int_{\calD_2} \exp(-\iota t^T \zeta) \tilde \rho_{ij}(t) dt$. Hence, $\tilde b_{ij}$ is the Fourier transform of $\tilde \rho_{ij}$ and as $\tilde \rho_{ij}(t) \in \calL_2(\mathbb R^{2d})$, applying the Plancherel theorem once again, we have  $\int_{\mathbb R^{2d}} |\tilde b_{ij}(\zeta)|^2d\zeta < \infty$, which implies the result.
\end{proof}

The proof of Theorem \ref{th:equiv} relies on the following more general but technical result on the equivalence of multivariate random fields. 
\begin{theorem}\label{th:techequiv} Let $\calP_0$ and $\calP_1$ denote two $p$-dimensional stationary zero-mean Gaussian random field measures on $\calD$. Let $C^{(i)}$ and $F^{(i)}$ denote their respective covariance functions and spectral densities. Suppose Condition \eqref{eq:cond1} is satisfied, and there exists a function $B(\omega,\mu) \in  \calW_{2,\Fz}$ such that $B(\omega,\mu)=B(\mu,\omega)^*$,
	and for all $s,s' \in \calD$, 
	\begin{equation}\label{eq:equivtech}
		C^{(1)}(s-s') -  C^{(0)}(s-s') = \intrd \intrd \exp(-\iota {s}^T \omega+\iota {s'}^T \mu)  \Fz(\omega) B(\omega,\mu) \Fz(\mu)  d\omega \, d\mu.
	\end{equation}
	Then $\calP_0 \equiv \calP_1$.
\end{theorem}

Proof of Theorem \ref{th:techequiv} is provided in Section \ref{sec:otherproofs}.

\begin{proof}[\textbf{Proof of Theorem \ref{th:equiv}}] 
	For a spectral density $F$, let $\calP(F)$ denote the measure on the paths of the zero-mean stationary Gaussian random field on $\mathbb R^d$ which has spectral density $F$.
	
	We first consider the case where $\Fo(\omega) \geq \Fz(\omega)$ for all $\omega$. Define
	\begin{align*}
		\widetilde \Fz = c_1 \Phi; \widetilde \Fo = c_1 \Phi + \Fo - \Fz; \widetilde \Ft = \Fz - c_1\Phi.
	\end{align*}
	
	Then for all $i$ and $\omega$, the matrix $\widetilde \Fi(\omega)$ is positive definite and $\tfz$ and $\tfo$ also satisfy (\ref{eq:cond1}). Let $\widetilde Z^{(i)}$ denote independent zero-mean GRFs with spectral densities $\widetilde \Fi$, for $i=0,1,2$. Then for $i=0,1$, defining $Z^{(i)} = \widetilde Z^{(i)} + \widetilde Z^{(2)}$, we have $Z^{(i)}$'s to be zero-mean GRFs with respective spectral densities $\Fi$, and it suffices to show that the measures corresponding to the paths of $\widetilde Z^{(i)}$ for $i=0,1$ are equivalent. 
	
	Let $\{g_k\}_{k \in \mathbb N}$ denote an orthonormal basis function of $\calW_\calD(\widetilde\Fz)$ and let $H=\Phi^{-1/2}(\Fo-\Fz)\Phi^{-1/2} = \Phi^{-1/2}(\widetilde\Fo-\widetilde\Fz)\Phi^{-1/2}$.
	Then 
	\begin{equation}\label{eq:basis-bound-step}
		\begin{aligned}
			\sum_k \left[\|g_k\|^2_{\widetilde \Fo} - \|g_k\|^2_{\widetilde \Fz} \right]^2 = & \sum_k \left(\intrd g_k(\omega)^* \Phi^{1/2}(\omega) H(\omega) \Phi^{1/2}(\omega) g_k(\omega) d\omega \right)^2\\
			\leq & \sum_k  \left(\intrd g_k(\omega)^* \Phi^{1/2}(\omega) H(\omega) H^*(\omega) \Phi^{1/2}(\omega) g_k(\omega) d\omega \right) \times \\
			& \qquad \left(\intrd g_k(\omega)^* \Phi(\omega) g_k(\omega) d\omega \right) \\
			= & \frac 1{c_1} \sum_k  \left(\intrd g_k(\omega)^* \Phi^{1/2}(\omega) H(\omega) H^*(\omega) \Phi^{1/2}(\omega) g_k(\omega) d\omega \right) \\
			\leq & \frac 1{c_1} \sum_k  \left(\intrd g_k(\omega)^* \Phi(\omega)g_k(\omega) \| H(\omega) \|^2 d\omega \right) \\
			= & \frac 1{c_1} \intrd \left(\sum_k g_k(\omega)^* \Phi(\omega)g_k(\omega) \right) \| H(\omega) \|^2 d\omega . 
		\end{aligned}
	\end{equation}
	Here, the first inequality is due to the Cauchy-Schwartz inequality and the second equality uses the fact that the $g_k$'s are an orthonormal basis set in $\calW_\calD(\widetilde\Fz)$ and that $\Phi = \tfz /c_1$. Also note that we assumed the basis $g_k$ of $\calW_\calD(\widetilde\Fz)$ lies in $\calW_\calD$. This is because, following the lemma on page 34 of \cite{skorokhod1973absolute}, it is enough to prove the case when $g_k \in \calW_\calD$ for all $k$, as $\calW_\calD$ is dense in $\calW_\calD(\widetilde\Fz)$. See Lemma 2 of
	\cite{bachoc2022asymptotically} for a formal proof of this.
	
	Let $g_k=(g_{k,1},\ldots,g_{k,p})^T$ and $h_k=(h_{k,1},\ldots,h_{k,p})^T$, where $h_{k,i}(\omega)=\phi_i(\omega)g_{k,i}(\omega)$. Then as both $\phi_i$ and $g_{k,i}$ lie in $\calW_\calD$, by convolution, there exists a square-integrable function $\psi_{k,i}: \mathbb R^d \to \mathbb C$ that is zero outside of $[-2T,2T]^d$ such that $h_{k,i}$ is the Fourier transform of $\psi_{k,i}$. 
	
	Since $\langle g_k, g_{k'}\rangle_{\widetilde \Fz} = \delta_{kk'}$, using Parseval's identity we have 
	$$
	\begin{aligned}
		\delta_{kk'} = c_1 \intrd g_k(\omega)^* \Phi(\omega)g_{k'}(\omega) =& c_1 \intrd \sum_{i=1}^p \overline{g_{k,i}}(\omega)  {g_{k',i}}(\omega) \phi_i(\omega)^2 d\omega \\
		=& c_1 (2\pi)^{-d} \int_[-2T,2T]^d \sum_{i=1}^p 
		\overline
		{\psi_{k,i}(h)}\psi_{k',i}(h) dh \\
		=& c_1 (2\pi)^{-d} \int_[-2T,2T]^d \psi_{k}(h)^*\psi_{k'}(h) dh .\end{aligned}
	$$
	
	So $\{\sqrt{c_1 (2\pi)^{-d}} \psi_k \}_k$ is an orthonormal system of $\calL_2([-2T,2T]^d)$. Writing $e_i$ for the $i^{th}$ row of $I_{p \times p}$, we then have
	
	$$
	\begin{aligned}
		\sum_k g_k(\omega)^* \Phi(\omega)g_k(\omega) &= \sum_k \sum_{i=1}^p |h_{k,i}(\omega)|^2 \\
		&= \frac 1{(2\pi)^{2d}} \sum_k \sum_{i=1}^p \Big| \int_[-2T,2T]^d \exp(-\iota h^T \omega) \psi_{k,i}(h) dh \Big|^2 \\
		&= \frac 1{c_1(2\pi)^{d}} \sum_{i=1}^p \sum_k  \Big| \int_[-2T,2T]^d \exp(-\iota h^T \omega) \sqrt{c_1 (2\pi)^{-d}} \psi_{k,i}(h) dh \Big|^2 \\
		&\leq \frac 1{c_1(2\pi)^{d}} \sum_{i=1}^p \int_[-2T,2T]^d \|\exp(-\iota h^T \omega) e_i\|^2 dh  \\
		& \leq \frac {p(\blue 4 T)^d}{c_1(2\pi)^{d}}.
	\end{aligned}
	$$
	Here, the penultimate step follows from Bessel's inequality. Using (\ref{eq:equivgeneral}), we have   
	\begin{equation}\label{eq:basis}
		\begin{aligned}
			\sum_k \left[\|g_k\|^2_{\widetilde \Fo} - \|g_k\|^2_{\widetilde \Fz} \right]^2 &\leq \frac {p(\blue 4T)^d}{c_1^\blue 2(2\pi)^{d}} \intrd \|H(\omega)\|^2 d\omega < \infty.
		\end{aligned}
	\end{equation}
	
	Let $V$ denote a symmetric operator on $\calW_\calD(\widetilde \Fz)$ such that for $u,v \in \calW_\calD(\widetilde \Fz)$ we have $$\langle Vu,v\rangle_{\widetilde \Fz} = \intrd u(\omega)^*\widetilde \Fo(\omega)v(\omega)d\omega.$$
	Existence of such a $V$ follows from the Riesz representation theorem.
	Using (\ref{eq:basis}), for every orthonormal basis $\{g_k\}$ of $\calW_\calD(\tfz)$, we have $$\sum_k \langle(V-I)g_k,g_k\rangle_{\tfz}\, < \infty.$$ Also, as $\tfo \geq \tfz$, $V-I$ is positive definite. So, $V-I$ is a Hilbert-Schmidt operator with eigenfunctions $\{v_k(\omega)\}_k$ and non-negative eigenvalues $\{\lambda_k\}$ such that $\sum_k \lambda_k^2 < \infty$. 
	
	Let $B_K(\mu,\omega)=\sum_{k=1}^K \lambda_k v_k(\mu)v_k(\omega)^*$. As $v_k \in \calW_\calD(\tfz)$, $B_K(\mu,\omega) \in \calW^2_\calD(\tfz)$ (see discussion after (\ref{eq:w2})). Following (\ref{eq:w2}), we have 
	$$
	\begin{aligned}
		& \|B_K\|^2_{2,\tfz} \\ & = \intrd \intrd \mbox{trace} \left[ B_K(\mu,\omega) \tfz(\omega) B_K(\mu,\omega)^* \tfz(\mu) \right] d\mu d\omega\\
		& = \intrd \intrd \mbox{trace} \left[  \left( \sum_{k=1}^K \lambda_k v_k(\mu) v_k(\omega)^* \right) \tfz(\omega) 
		{\left( \sum_{k'=1}^K \lambda_{k'} v_{k'}(\omega)v_{k'}(\mu)^*  \right)}\tfz(\mu) \right] d\mu d\omega \\
		& = \sum_{k,k'=1}^K \lambda_k \lambda_{k'}  \left( \intrd v_{k'}(\mu)^*  \tfz(\mu)v_{k}(\mu)  d\mu \right) \left( \intrd    
		v_{k}(\omega)^* \tfz(\omega) 
		v_{k'}(\omega) d\omega \right) \\  
		& = \sum_{k=1}^K \lambda_k^2. 
	\end{aligned}
	$$
	
	As $\sum_{k=1}^\infty \lambda_k^2 < \infty$, the limit of $B_K$ as $K \to \infty$ is well-defined as 
	\begin{equation}\label{eq:b}
		B(\mu,\omega) = \sum_k \lambda_k v_k(\mu)v_k(\omega)^* \mbox{ with } B \in \calW_{2,\tfz} \mbox{ and }  \|B\|_{2,\tfz}^2 = \sum_{k=1}^\infty \lambda^2_k. 
	\end{equation}
	
	Then, letting $a_{is}(\mu)=\exp(-\iota s^T \mu) e_i$, we have
	\begin{equation*}
		\begin{aligned}
			\widetilde  C_{ij}^{(1)} & (s-s') - \widetilde C_{ij}^{(0)}(s-s') \\
			&= \intrd \exp(\iota (s-s')^T\mu) \tfo_{ij}(\mu)d\mu - \intrd \exp(\iota (s-s')^T\mu) \tfz_{ij}(\mu)d\mu \\
			&= \intrd a_{is}(\mu)^* \left(\tfo(\mu) - \tfz(\mu)\right) a_{js'}(\mu) d\mu \\
			&= \langle(V-I)a_{is},a_{js'}\rangle_{\tfz} \\
			&= \sum_k \lambda_k \langle a_{is},v_{k}\rangle_{\tfz}\langle v_k,a_{js'}\rangle_{\tfz}\\
			&= \sum_k \lambda_k \intrd  \intrd  a^*_{is}(\omega) \tfz(\omega)v_{k}(\omega) v_k^*(\mu)\tfz(\mu)a_{js'}(\mu) d\omega\, d\mu \\
			&= \intrd \intrd \exp(\iota {s}^T \omega - \iota {s'}^T \mu)   \left[ \tfz(\omega) \left( \sum_k \lambda_k v_{k}(\omega) v_k^*(\mu) \right) \tfz(\mu) \right]_{ij} d\omega\, d\mu.
		\end{aligned}
	\end{equation*}
	
	Since this holds for all $s,s'$, noting that $s-s'=(-s') - (-s)$ and using $B(\mu,\omega)^*=B(\omega,\mu)$, we have 
	$$
	\begin{aligned}
		& \widetilde C^{(1)}(s-s') - \widetilde C^{(0)}(s-s') \\
		&= \intrd \intrd \exp(-\iota {s'}^T \omega + \iota {s}^T \mu)  \left[ \tfz(\omega) B(\mu,\omega)^* \tfz(\mu) \right] d\omega\, d\mu\\
		&= \intrd \intrd \exp(-\iota {s}^T \omega + \iota {s'}^T \mu)  \left[ \tfz(\omega) B(\omega,\mu) \tfz(\mu) \right] d\omega\, d\mu.
	\end{aligned}
	$$
	As we have already shown in (\ref{eq:b}) that $B \in \calW_{2,\Fz}$, the matrix $B$ then satisfies all the conditions of Theorem \ref{th:techequiv}
	and we have $\calP(\tfz)\equiv\calP(\tfo)$ 
	and consequently, $\calP_0\equiv\calP_1$. 
	
	Now we consider the case where neither $F^{(1)} \geq F^{(0)}$ or $F^{(0)} \geq F^{(1)}$. Then we have $\sup_{\|x\|=1} x^* \Phi^{-1/2}(\Fo - \Fz)\Phi^{-1/2}x \geq 0$. We define the following spectral densities:
	
	\begin{align*}
		\widehat \Fo =& \tfz + \left(\sup_{\|x\|=1} x^* \Phi^{-1/2}(\Fo - \Fz)\Phi^{-1/2}x\right)\Phi;\, 
		F^{(3)} = \widehat \Fo + \widetilde{F^{(2)}}.
	\end{align*}
	
	We will show that $\calP(\Fz) \equiv \calP(\Fo)$
	by showing that both of them are equivalent to and dominated by $F^{(3)}$. Note that, as $\Fz = \tfz + \widetilde{F^{(2)}}$, we have $\tfz=c_1 \Phi$ and $\tfz \leq \widehat \Fo$.
	Also, as $\Fo - \Fz \leq (c_2 - c_1) \Phi$, we have $\sup_{\|x\|=1} x^* \Phi^{-1/2}(\Fo - \Fz)\Phi^{-1/2}x \leq (c_2 - c_1)$ and thus $\widehat \Fo < c_2 \Phi$. So both $\tfz$ and $\widehat \Fo$ are bounded from below and above respectively by $c_1 \Phi$ and $c_2 \Phi$. Hence, it is enough to show $\calP(\tfz) \equiv \calP(\widehat \Fo)$. 
	
	We have
	
	$$
	\begin{aligned}
		\|\Phi(\omega)^{-1/2}(\widehat{F^{(1)}}-\tfz)\Phi(\omega)^{-1/2}\|^2 & = \left(\sup_{\|x\|=1} x^* \Phi^{-1/2}(\Fo - \Fz)\Phi^{-1/2}x\right)^2 \\
		& \leq \|\Phi(\omega)^{-1/2}(F^{(1)}-F^{(0)})\Phi(\omega)^{-1/2}\|^2. 
	\end{aligned}
	$$
	Using (\ref{eq:equivgeneral}), we then have $\calP(\tfz) \equiv \calP(\widehat \Fo)$ implying $\calP(\Fz) \equiv \calP(F^{(3)})$. 
	
	Next to show $\calP(\Fo) \equiv \calP(F^{(3)})$, we have for any $x$ with $\|x\|=1$, 
	\begin{align}
		&  x^*\Phi^{-1/2}(F^{(3)}-\Fo)\Phi^{-1/2}x \\
		& \quad = \left(\sup_{\|x\|=1} x^* \Phi^{-1/2}(\Fo - \Fz)\Phi^{-1/2}x\right) - x^* \Phi^{-1/2}(\Fo - \Fz)\Phi^{-1/2}x \geq 0.
	\end{align}

	So, $\Phi^{-1/2}F^{(3)}\Phi^{-1/2} \geq \Phi^{-1/2}\Fo\Phi^{-1/2}$ which implies $F^{(3)} \geq \Fo \geq c_1 \Phi$ as $\Phi$ is a diagonal matrix with positive entries. 
	Also, 
	$$
	\begin{aligned}
		F^{(3)}(\omega)&=F^{(0)}(\omega)+\lambda_{\max}\!\left(\Phi(\omega)^{-1/2}\{F^{(1)}(\omega)-F^{(0)}(\omega)\}\Phi(\omega)^{-1/2}\right)\Phi(\omega) \\
		&\le F^{(0)}(\omega)+(c_2-c_1)\Phi(\omega)\\
		& \le (2c_2-c_1)\Phi(\omega). 
	\end{aligned}
	$$
	So, $F^{(3)}$ satisfies the condition of Theorem \ref{th:equiv} with constants $c_1$ and $c_2'=2c_2-c_1$.
	As $$\|\Phi^{-1/2}(F^{(3)}-\Fo)\Phi^{-1/2} \|^2 \leq 4 \|\Phi^{-1/2}(F^{(1)}-\Fz)\Phi^{-1/2}\|^2,$$ by (\ref{eq:equivgeneral}), we have $$\intrd \|\Phi(\omega)^{-1/2}(F^{(3)}(\omega)-\Fo(\omega))\Phi(\omega)^{-1/2}\|^2 d\omega < \infty,$$ yielding $\calP(\Fo)\equiv \calP(F^{(3)})$, and proving the theorem. 
	
\end{proof}

\begin{proof}[\textbf{Proof of Theorem \ref{th:nonidgen}}]
	Without loss of generality, we can take $\calD=\calD^*=[-T,T]^d$ for some $T$ as we can always embed the original $\calD$ in such a larger rectangle. Let $Z=(X,\beta X + W)$ and $\calP_0$ and $\calP_1$ denote two measures for $\beta = 0$ and $\beta=1$. We will use Theorem \ref{th:equiv} to prove the equivalence of $\calP_0$ and $\calP_1$ when (\ref{eq:nec}) holds.
	
	Let $F^{(i)}$ denote the $2\times 2$ spectral density matrix of $Z$ under $\calP_i$.
	Then we have 
	\begin{equation}\label{eq:spectral}
		\Fz(\omega)=\left(\begin{array}{cc}
			f_X(\omega) & 0 \\
			0 & f_W(\omega)
		\end{array}\right) \mbox{ and } 
		\Fo(\omega)=
		\left(\begin{array}{cc}
			f_X(\omega)  &\quad f_X(\omega) \\
			f_X(\omega) &\quad f_X(\omega) + f_W(\omega)
		\end{array}\right).
	\end{equation}
	Let $\Phi(\omega)=\diag(\phi^2_X(\omega),\phi^2_W(\omega))$. Then $c_1 \Phi(\omega) \leq \Fz \leq c_2 \Phi(\omega)$ for all $\omega$. 
	
	Let $K = 1+\sup_{\omega \in \mathbb R^d} f_X(\omega)/f_W(\omega)$ which is finite by the statement of the theorem.
	We then have 
	$$ (2K+1) c_2 \Phi(\omega) -  \Fo(\omega) \geq (2K+1) \Fz -  \Fo(\omega) = \left(\begin{array}{cc}
		2K f_X(\omega)  & -f_X(\omega) \\
		-f_X(\omega) & 2Kf_W(\omega) - f_X(\omega)
	\end{array}\right).$$
	This is a diagonally dominant symmetric matrix, and is thus positive definite. So $\Fo \leq (2K+1)c_2 \Phi$. Also, 
	$$ \Fo(\omega) - \frac {c_1}{K+2} \Phi(\omega) \geq \Fo(\omega) - \frac {1}{K+2} \Fz(\omega) = \left(\begin{array}{cc}
		\frac{K+1}{K+2} f_X(\omega)  & f_X(\omega) \\
		f_X(\omega) & f_X(\omega) + \frac{K+1}{K+2} f_W(\omega).
	\end{array}\right).$$
	The $(1,1)^{th}$ entry of this matrix is positive, and the determinant is 
	$$ f_X(\omega) \left\{\left(\frac{K+1}{K+2}\right)^2 f_W(\omega) - \frac{1}{K+2}f_X(\omega) \right\} > 0\; \forall \omega.$$ 
	
	So, $\Fo(\omega) - \frac {c_1}{K+2} \Phi(\omega)$ is positive definite for all $\omega$. Redefining $c_1=c_1/(K+2)$ and $c_2=(2K+1)c_2$, (\ref{eq:cond1}) is satisfied. 
	
	Applying Theorem \ref{th:equiv}, $\calP_0 \equiv \calP_1$ if we can show 
	\begin{equation*}
		\intrd \|\Phi(\omega)^{-1/2}(F^{(1)}(\omega)-F^{(0)}(\omega))\Phi(\omega)^{-1/2}\|^2 d\omega < \infty . 
	\end{equation*}
	
	From (\ref{eq:spectral}) we have for some constant $c$, 
	\begin{align*}
		& \intrd \|\Phi(\omega)^{-1/2}(F^{(1)}(\omega)-F^{(0)}(\omega))\Phi(\omega)^{-1/2}\|^2 d\omega \\
		&\leq c \intrd \mbox{trace}\left[ \left(F^{(1)}(\omega)F^{(0)}(\omega)^{-1} - I_{2 \times 2} \right)^2 \right] d\omega \\
		&= c \intrd \mbox{trace}\left[ \left(\begin{array}{cc}
			0  & \frac{f_X(\omega)}{f_W(\omega)} \\
			1 & \frac{f_X(\omega)}{f_W(\omega)} 
		\end{array} \right)^2 \right] d\omega \\
		&= c \intrd \left(2 \frac{f_X(\omega)}{f_W(\omega)} + \frac{f_X(\omega)^2}{f_W(\omega)^2} \right)  d\omega. 
	\end{align*}
	Here, the first inequality follows from $\| A^2 \| \leq \mbox{trace}(A^2)$ for the  Hermitian matrix $A=\Phi(\omega)^{-1/2}(F^{(1)}(\omega)-F^{(0)}(\omega))\Phi(\omega)^{-1/2}$ and  $\|\Phi^{-1/2}\Fz\Phi^{-1/2}\| \leq c_2$. 
	
	As $f_X(\omega)/f_W(\omega) < K$, $f_X(\omega)^2/f_W(\omega)^2$ is dominated by $K f_X(\omega)/f_W(\omega)$, so equivalence holds when 
	$$
	\begin{aligned}
		\intrd \frac{f_X(\omega)}
		{f_W(\omega)} d\omega < \infty. 
	\end{aligned}
	$$
\end{proof}

\subsection{Proofs for main examples in Section \ref{sec:examples}}\label{eq:proofsmain}

\begin{proof}[\textbf{Proof of Corollary \ref{cor:matern}}]
	{\em Part (a):} Let $(K_{k\ell})_{1 \leq k,l \leq 2}$ denote the matrix-valued covariance function of $(X,W)$. Each $K_{k\ell}$ satisfies Assumption \ref{eq:K.assump} with $\alpha_{11}=2\nu_X$, $\alpha_{22}=2\nu_W$ and $\alpha_{12}=2\nu_{XW}$. Then $\alpha_{11} < \alpha_{12}$ and $\alpha_{11} < \alpha_{22} + d$ and the conditions of Theorem \ref{th:suff}
	(for $d=1$) or Theorem \ref{th:suff.rd} (for $d > 1$) are satisfied. 
	
	{\em Part (b):} Let $\theta$ be the total set of unknown parameters which includes $\beta$ and all the parameters of the bivariate Matérn covariance for $(X,W)$. The parameter $\beta$ will not be consistently estimable if for two sets of values of $\theta$ with different choices of $\beta$, the corresponding measures on the paths of the Gaussian random fields $(X,Y)$ are equivalent. We show that this happens on the two following choices: $(\beta=0,\rho_{XW}=0)$ and $(\beta=1,\rho_{XW}=0)$. Here $\rho_{XW}$ is the intra-site correlation parameter between $X$ and $W$ for the bivariate Matérn process, and $\rho_{XW}=0$ implies independence of $X$ and $W$.
	
	For the univariate Matérn covariance functions of $X$ and $W$, using Theorem 3.6 (iii) and Theorem 6.1 (i) in \cite{zastavnyi2006some}, we have that
	there exist positive functions $\phi_X$ and $\phi_W$ that are Fourier transforms of 
	compactly supported functions in $\mathbb R^d$ satisfying 
	\begin{equation}\label{eq:maternbound}
		\begin{aligned}
			c \phi^2_X(\omega) &\leq (1+\|\omega\|)^{-2\nu_X-d} \leq c' \phi^2_X(\omega), \mbox{ and }\\
			c \phi^2_W(\omega) &\leq (1+\|\omega\|)^{-2\nu_W-d} \leq c' \phi^2_W(\omega)
		\end{aligned}
	\end{equation}
	for all $\omega \in \mathbb R^d$ and for some positive constants $c,c'$. We refer to the proof of Lemma A.4 in \cite{bachoc2022asymptotically} for a detailed discussion on how (\ref{eq:maternbound}) is established.
	
	As $\sup_{\omega \in \mathbb R^d} (1+\|\omega\|)^{-2\nu_X-d}/f_X(\omega)$ and $\sup_{\omega \in \mathbb R^d} (1+\|\omega\|)^{-2\nu_W-d}/f_W(\omega)$ are uniformly bounded away from $0$ and $\infty$, the condition (\ref{eq:bound}) of Theorem \ref{th:nonidgen} is satisfied. Also, as $\nu_X > \nu_W$, we have $\sup_{\omega \in \mathbb R^d} f_X(\omega)/f_W(\omega) < \infty$.
	
	So (\ref{eq:nec}) will be established if we can show
	\begin{align*}
		\int \frac{(\phi_W^2+\|\omega\|^2)^{\nu_W+d/2}}{(\phi_X^2 + \|\omega\|^2)^{\nu_X + d/2}} d\omega < \infty. 
	\end{align*}
	
	Using the transformation $u=\|\omega\|$ we have
	$$
	\int \frac{(\phi_W^2+\|\omega\|^2)^{\nu_W+d/2}}{(\phi_X^2 + \|\omega\|^2)^{\nu_X + d/2}} d\omega = M \int_{0}^\infty u^{d-1}\frac{(\phi_W^2+u^2)^{\nu_W+d/2}}{(\phi_X^2 + u^2)^{\nu_X + d/2}} du$$
	for some constant $M$. The function within the integral on the right is bounded away from $\infty$ near $u=0$ and is $O(u^{-(2\nu_X - 2\nu_W - d + 1)})$ as $u \to \infty$. So, the integral is finite when 
	$2\nu_X - 2\nu_W - d + 1 > 1$, i.e., when $\nu_X > \nu_W + d/2$.
\end{proof}

\begin{proof}[Proof of Theorem \ref{thm:error}] We prove the result when $(X,W)$ is a bivariate Matérn GRF on an interval $[0,L]$ of positive length in $\mathbb R$ with smoothness parameters $\nu_X$ and $\nu_W$ respectively and cross-smoothness $\nu_{XW}$. Let $\alpha_{X}=2\nu_X$, $\alpha_{W}=2\nu_X$, and $\alpha_{XW}=2\nu_{XW}$. The same proof technique will hold for higher-dimensional spatial domains and for the other covariance families. 
	
	As $d=1$, when $\alpha_{11} < \alpha_{22} + 1$, $\beta$ is consistently estimable when there is no noise and $\alpha_{12} > \alpha_{11}$. We want to show that $\beta$ is consistently estimable even when there is noise. For simplicity, we consider the case when $\alpha_X < 1$, where taking first differences suffices (see the proof of Theorem \ref{th:beta.con}). The results for larger $\alpha_X$ can be proved by taking differences of higher order. 
	
	We first consider the case where only the outcome is observed with noise, and the exposure is noise-free, i.e., we observe $X(s)$ and $Z(s) = Y(s) + \epsilon(s)$ where $Y(s)=X(s)\beta + W(s)$. In Theorem \ref{th:suff}, based on noise-free data observed on a regular 1-dimensional lattice, the OLS estimator regressing first differences of $Y$ on those of $X$ was shown to be consistent.
	This estimator may no longer be consistent when replacing $Y$ with $Z$, as differencing the noisy $Z$ inflates the noise.
	Instead, we will first do local averaging to make the noise variance as small as desired and then do differencing. 
	
	We consider observations on a $\calG_{n^*}$ where $n^*=n^{\rho+2}$ for some integer $\rho$ whose value will be specified later. We create two grids from this. The coarse grid is $\calG_n=\{0,hL,2hL,\ldots,nhL\}$ where $h=1/n$. For some $\rho > 1$, at each $hj$ for $0 < j < n$ we consider the process $X^*(hj)=\frac 1{2n^\rho + 1} \sum_{k=-n^\rho}^{n^\rho} X(hj+\frac k{n^{\rho+2}})$. So $X^*(hj)$ is the average of $X(s)$ at a fine regular subgrid of $2n^\rho + 1$ locations in $[hj-\frac 1{n^2},hj+\frac 1{n^2}]$, centered around $hj$. Define $Y^*(hj)$, $\epsilon^*(hj)$ similarly. As $(X,W)$ is a stationary process, and $(X^*,W^*)$ is defined on a regular grid, based on averaging of $(X,Y)$ over a regular sub-grid, $(X^*,Y^*)$ is also stationary over the interior of $\calG_n$. 
	
	Let
	\(
	M=2n^\rho+1,\; \eta=n^{-(\rho+2)},\; x=h(j-j').
	\)
	Then
	\[
	\begin{aligned}
		K^*_{11}(hj,hj') & = \text{Cov}(X^*(hj),X^*(hj'))\\
		&=
		\frac{1}{M^2}
		\sum_{k,k'=-n^\rho}^{n^\rho}
		K_{11}\{x+(k-k')\eta\}  \\
		&=
		K_{11}(x)
		+
		\frac{1}{M^2}
		\sum_{m=1}^{2n^\rho}
		(M-m)\Delta_{m\eta}K_{11}(x),
	\end{aligned}
	\]
	where
	\[
	\Delta_a K_{11}(x)=K_{11}(x+a)-2K_{11}(x)+K_{11}(x-a).
	\]
	
	Let
	\[
	R_{jj'}
	=
	\frac{1}{M^2}
	\sum_{m=1}^{2n^\rho}
	(M-m)\Delta_{m\eta}K_{11}\{h(j-j')\}.
	\]
	Since
	\[
	\frac{1}{M^2}\sum_{m=1}^{2n^\rho}(M-m)\le 1,
	\]
	it suffices to bound \(\Delta_{m\eta}K_{11}\{h(j-j')\}\) uniformly over
	\(1\le m\le 2n^\rho\).
	
	If \(j=j'\), then \(x=h(j-j')=0\). By the Matérn local expansion,
	\[
	\Delta_{m\eta}K_{11}(0)
	=
	2\{K_{11}(m\eta)-K_{11}(0)\}
	=
	O\{(m\eta)^\alpha\}
	=
	O(h^{2\alpha})
	=
	o(h^\alpha),
	\]
	because \(m\eta\le 2n^{-2}=2h^2\).
	
	If \(j\ne j'\), then \(|x|=h|j-j'|\ge h\), while \(m\eta\le 2h^2=o(|x|)\).
	Using \(K_{11}''(t)=O(|t|^{\alpha-2})\) away from zero, Taylor's theorem gives
	\[
	\Delta_{m\eta}K_{11}(x)
	=
	O\{(m\eta)^2 |x|^{\alpha-2}\}
	\le
	O(h^4 h^{\alpha-2})
	=
	O(h^{\alpha+2})
	=
	o(h^\alpha).
	\]
	Therefore \(R_{jj'}=o(h^\alpha)\) uniformly in \(j,j'\).
	
	Since \(m\eta\le 2n^{-2}=2h^2\), the summation term is
	\(o(h^\alpha)\) uniformly in \(j,j'\). Hence
	\[
	K^*_{11}(hj,hj')
	=
	K_{11}\{h(j-j')\}+o(h^\alpha),
	\]
	uniformly in \(j,j'\).
	
	So, on the grid $\calG_n$, $K^*_{11}$ satisfies Assumption \ref{eq:K.assump} and has similar near-zero distance behavior. Similar, results hold for $K^*_{12}$ and $K^*_{22}$ and by Theorem \ref{th:beta.con}, the OLS estimator on first differences of $Y^*=X^*\beta + W^*$ on those of $X^*$ is consistent for $\beta$. As $Y^*$ is not observed, we will use $Z^*=Y^*+\epsilon^*$, which is simply averaging the observed $Z$ process on the finer sub-grid. We show that the extra-term $\frac{\sum_{i=1}^n (\epsilon^*_{i}-\epsilon^*_{i-1})(X^*_{i}-X^*_{i-1})}{\sum_{i=1}^n (X^*_{i}-X^*_{i-1})(X^*_{i}-X^*_{i-1})} \to 0$.
	
	Using the proof of Theorem \ref{th:beta.con}, $\frac 1{nh^{\alpha_{11}}} \sum_{i=1}^n (X^*_{i}-X^*_{i-1})(X^*_{i}-X^*_{i-1}) \to c_{11} \neq 0$. 
	As $X \perp \epsilon$, the numerator has mean $0$. So it is enough to show that $$\frac 1{n^2h^{2\alpha_{11}}} \text{Var}\left(\sum_{i=1}^n (\epsilon^*_{i}-\epsilon^*_{i-1})(X^*_{i}-X^*_{i-1})\right) \to 0.$$ 
	
	Let $A$ denote the Laplacian matrix corresponding to taking first differences on a grid. Note that as $\epsilon(s) \iid N(0,\tau^2)$, $\epsilon^*(hj) \iid N(0, \frac{\tau^2}{2n^\rho+1})$. As we can write $\sum_{i=1}^n (\epsilon^*_{i}-\epsilon^*_{i-1})(X^*_{i}-X^*_{i-1}) = X^{*T}A\epsilon^*$,  using the law of total variance, we have 
	$$
	\begin{aligned}
		\text{Var}(X^{*T}A\epsilon^*) &= 
		E[\text{Var}(X^{*T}A\epsilon^* \given X^*)] + \text{Var}(E[X^{*T}A\epsilon^*\given X^*]) \\
		&= \frac{\tau^2}{2n^\rho+1}  E(X^*A^2X^*) + 0\\
		&\leq \frac{\tau^2}{2n^\rho+1} \lambda_{\max}(A) E(X^{*T} A X^*) \\
		& \leq \frac{4\tau^2}{2n^\rho+1}  E(\sum_{i=1}^n (X^*_{i}-X^*_{i-1})(X^*_{i}-X^*_{i-1})) \\
		& = 4 \frac{\tau^2}{2n^\rho+1} nc_{11}h^{\alpha_{11}} \\
		& = O(n^{1-\alpha_{11}-\rho}).
	\end{aligned}
	$$
	In the above, we have bounded $\lambda(A)$ by $4$ using the Gershgorin circle theorem, and the penultimate equality comes from the proof of Theorem \ref{th:beta.con}. 
	Then $$\frac 1{n^2h^{2\alpha_{11}}} \text{Var}\left(\sum_{i=1}^n (\epsilon^*_{i}-\epsilon^*_{i-1})(X^*_{i}-X^*_{i-1})\right) = O(n^{1 - \alpha_{11} - \rho -2 + 2\alpha_{11}}) = O(n^{\alpha_{11} - \rho -1}) \to 0$$
	by choosing $\rho > \max(0,\alpha_{11} - 1)$ and we have the proof completed for this case. 
	
	For the case when $X$ is also observed with noise, i.e., we observe $\tilde X_i = X_i + \varepsilon_i$ where $\varepsilon_i$ is iid with zero mean and finite variance, we need to average both $Z$ and $\tilde X$ before taking differences. To make the error in $X$ sufficiently small, we need a larger $\rho=\max(0,\alpha_{11})$. With this choice of $\rho$, the additional terms coming from the noise in $X$ vanish asymptotically, and the OLS estimator regressing differences of $Z$ on differences of $\tilde X$ is consistent. 
	
	For the other direction, when $\alpha_{11} > \alpha_{22} + d$, we follow part of the proof of Theorem 6 of \cite{stein1999interpolation}.
	Note that 
	$\beta$ is not consistently estimable on the paths of $\{(Y(s),X(s)) : s\in \calD\}$ where $Y(s)=X(s)\beta+W(s)$. Hence, it is not consistently estimable on the measure generated by $Y(s_1), X(s_1),Y(s_2),X(s_2),\ldots$ for any countable sequence of locations $s_1, s_2, \ldots$. As the errors $\epsilon$'s and $\varepsilon$'s are independent of $(X,Y)$ with their distributions not depending on $\beta$,
	it is evident that $\beta$ is not consistently estimable on the measure generated by ${\cal Y}=\{Y(s_1),X(s_1),\epsilon(s_1),\varepsilon(s_1)Y(s_2),X(s_2),\epsilon(s_2),\varepsilon(s_2)\ldots\}$. As the $\sigma$-algebra generated by ${\cal Y}$ contains that generated by ${\cal Z} =\{Z(s_1),\tilde X(s_1),Z(s_2),\tilde X(s_2),\ldots\}$, $\beta$ cannot be identified on the measure generated by $\cal Z$ for any sequence of locations. 
\end{proof}

\section{Remaining proofs}\label{sec:otherproofs}

\subsection{Proofs for consistent estimability in higher dimension}\label{sec:proof2d}

Before proving Theorem \ref{th:suff.rd}, we first state and prove a generalization of Theorem \ref{th:beta.con} for $\mathbb R^d$.

\begin{theorem}\label{th:beta.con.Lap}
	Let $Z=(Z_1,Z_2)$ denote a bivariate stationary GRF on a set $\calD \in \mathbb R^d, d > 1,$ such that $\calD$ contains a $d$-dimensional open ball. Let the  covariance function of $Z$ be $C=(C_{k\ell})_{\{1 \leq k,\ell \leq 2\}}$ 
	which satisfies Assumption \ref{as:rd}. Further assume that the constants $c_{k\ell}$ and the exponents $\alpha_{k\ell}$ in $C$ satisfy Equation (\ref{eq:K12limit}) for some $\beta$. 
	Then the measures on the paths of the bivariate random fields $(Z_1,Z_2)$ are orthogonal for different values of $\beta$ if $\alpha_{11}-d < \alpha_{22} \le \alpha_{11}$. In particular, letting $m$ be an integer such that $\alpha_{11} < 4m$, $\mbox{Lap}_{n}^{(m)}(Z_1,Z_2) \to \beta$ in probability as $n\to\infty$.
\end{theorem}

Theorem \ref{th:beta.con.Lap} provides a general result on consistent estimability of the ratio of the coefficients of principal irregular terms for two correlated GRFs on $\mathbb R^d$ using discrete Laplacians. This result does not rely on any specific parametric form for the covariance functions and immediately leads to Theorem \ref{th:suff.rd} on consistent estimability of the regression slope under spatial confounding for processes in $\mathbb R^d$. 

The proof of Theorem \ref{th:beta.con.Lap} relies on several technical results on Laplacians of isotropic covariance functions and radially symmetric functions, which we state and prove first. 

\begin{lemma}\label{lem:lapmean}
	Let $(Z_1,Z_2)$ be a stationary GRF on $\mathbb R^d$ with covariance function $C=(C_{k\ell})$
	satisfying Assumption \ref{as:rd} with parameters $(c_{k\ell})$ and $(\alpha_{k\ell})$ with $\max$$_{k\ell} \,\alpha_{k\ell} < 4$.
	Let $Z^{(1)}=(Z_1^{(1)},Z_2^{(1)})$ where $Z_i^{(1)}=\Delta_h Z_i$, defined
	in the interior $\calG_n^{(1)}$ of the grid $\calG_n$.
	Then, for any $s_i \in \calG_n^{(1)}$,  $E(Z_k^{(1)}(s_i)Z_l^{(1)}(s_i)) = \gamma_{k\ell}
	c_{k\ell} h^{\alpha_{k\ell}-4} + o(h^{\alpha_{k\ell}-4})$, where $\gamma_{k\ell}
	$ is some constant depending on $\alpha_{k\ell}$ and $d$. 
\end{lemma}

\begin{proof} 
	As the first discrete Laplacian process of a stationary process is stationary on the interior $\calG_n^{(1)}$, it is enough to prove for one $s_i \in \calG_n^{(1)}$.
	We will use $i' \sim i$ to indicate that two grid locations $s_i$ and $s_{i'}$ in $\calG_n$ are adjacent, i.e., $\|s_i - s_{i'}\|=h$. 
	If $i' \sim i$  
	for $s_i \in \calG_n^{(1)}$,
	then $s_{i'}=s_i \pm he_g$ for some $g \in 1,\ldots,d$. For notational simplicity, we drop the subscript $k\ell$ and first prove a result for a stationary univariate GRF $X$ on $\calG_n$ with covariance $C$. By Assumption \ref{as:rd}, $C(u)=A(\|u\|) + B(\|u\|) + o(\|u\|^\alpha)$.
	As $\alpha <
	4$, for $t > 0$, we have $A(t) = a_0 + a_2 t^2
	+ o(t^\alpha)$ and $B(t) = ct^\alpha + o(t^\alpha)$. As $\|e_g\|=1$, we have
	$$
	\begin{aligned}
		E \left[  2X(0) - X(h e_g) - X(-h e_g) \right]^2 &=
		6C(0) - 4C(h e_g) - 4C(-h e_g) + 2C(2h e_g) \\
		&= - 8a_2 h^2 + 8a_2 h^2
		- 8c h^\alpha + 2^{\alpha+1} c h^\alpha +  o(h^\alpha)\\
		&= (2^{\alpha+1} - 8) c h^\alpha + o(h^\alpha).
	\end{aligned}
	$$
	
	Also, for any $g \neq g'$, as $\|e_g\|=1$ and $\|e_g \pm e_{g'}\|=\sqrt 2$, we have 
	$$
	\begin{aligned}
		& E \left[ \left( 2X(0) - X(h e_g) - X(-h e_g) \right) 
		\left( 2X(0) - X(h e_{g'}) - X(-h e_{g'}) \right) \right]\\
		\qquad &= \left[ 
		4C(0) - 2C(h e_g) - 2C(-h e_g)
		- 2C(h e_g') - 2C(-h e_g') \right. \\
		& \qquad 
		\left. + C(h(e_g - e_g')) + C(h(e_g + e_g')) + C(-h(e_g - e_g')) + C(-h(e_g + e_g'))\right] \\
		&= - 8a_2 h^2 + 8a_2 h^2
		- 8c h^\alpha + 2^{\frac\alpha 2 + 2} c h^\alpha +  o(h^\alpha)\\
		&= (2^{\frac \alpha 2+2} - 8) c h^\alpha + o(h^\alpha).
	\end{aligned}
	$$
	
	Returning to the setup of the Lemma, using the above results, we have 
	$$
	\begin{aligned}
		& E(Z^{(1)}_{ki}Z^{(1)}_{\ell i}) \\
		&= \frac 1{h^4} \sum_{g,g'=1}^d E \left[ \left( 2Z_k(0) - Z_k(h e_g) - Z_k(-h e_g) \right) 
		\left( 2Z_\ell(0) - Z_\ell(h e_{g'}) - Z_\ell(-h e_{g'}) \right) \right] \\
		& = \gamma_{k\ell} c_{k\ell} h^{\alpha_{k\ell}-4} + o(h^{\alpha_{k\ell}-4})
	\end{aligned}
	$$ 
	where $\gamma_{k\ell} = d (2^{\alpha_{k\ell}+1} - 8) + d(d-1) (2^{\frac {\alpha_{k\ell}}2+2} - 8)$. 
\end{proof}

\begin{lemma}\label{lem:lapbound}
	Let $(Z_1,Z_2)$ be a GRF on an $(n+1)^d$-sized regular grid $\calG_n \in [0,L]^d$ with a covariance function $C=(C_{k\ell})$ that
	satisfies Assumption \ref{as:rd} for some covariances $K_{k\ell}$ on $\mathbb R$, and constants $c_{k\ell}$ and $\alpha_{k\ell}$. Let $m$ be a positive integer such that  $\alpha_{k\ell} <4m$ for all $k,\ell$, and define $Z^{(m)}=(Z_1^{(m)},Z_2^{(m)})$ where $Z_i^{(m)}=\Delta_h^{(m)} Z_i$. Then on $\calG_n^{(m)}$, $Z^{(m)}$ is stationary with covariance $C^{(m)}=(C_{k\ell}^{(m)})$, such that
	$C_{k\ell}^{(m)}(u)$ satisfies Assumption \ref{as:rd} with the isotropic part being $K^{(m)}_{k\ell}$
	which satisfies Assumption \ref{eq:K.assump}
	with constant $c=c_{k\ell}M$ and exponent $\alpha_{k\ell}-4m$, where $M$ is a constant depending only on $\alpha_{k\ell}$ and $d$. 
\end{lemma}

\begin{proof} It is enough to prove this for $m=1$, which implies $\alpha_{k\ell} < 4$ for all $k,\ell$.
	For larger $m$, we can then apply the result recursively. Let $s_i \neq s_j \in \calG_n^{(1)}$ with $s_i - s_j = u$. Let $C^{(1)}(u) = \text{Cov} \left( \Delta_h Z_k (s_j), \Delta_h Z_\ell (s_j) \right)$.
	
	For a function $f: \mathbb R^d \to R$, define the directional discrete Laplacian (at lag $h$) along a direction $g \in \{1,2,\ldots,d\}$ as $\Delta_{h,g} f$. For notational simplicity, we drop the subscript $k\ell$ and first prove a result for a stationary univariate GRF $Z$ on $\calG_n$ with covariance $C$ satisfying
	$C(u)=C^*(u) + r(u)$ where $C^*(u)$ is the isotropic part, i.e., $ C^*(u)= K(\|u\|) = A(\|u\|) + B(\|u\|)$. Then we have 
	$$
	\begin{aligned}
		C^{(1)}(u) & = \text{Cov} \left( \Delta_h Z (s_i), \Delta_h Z (s_j) \right) \\ 
		& =  \sum_{g,g' = 1}^{d} \Delta_{h,g}\Delta_{h,g'} C(u) \\
		& = \sum_{g,g' = 1}^{d} \Delta_{h,g}\Delta_{h,g'} \left[K(\|u\|) + r(u)\right] \\
		& = \sum_{g,g' = 1}^{d} \left[ \frac{\partial^4 K(\|u\|)}{\partial u_g^2 \partial u_{g'}^2} + \left(\Delta_{h,g}\Delta_{h,g'} \left[K(\|u\|) + r(u)\right] - \frac{\partial^4 K(\|u\|)}{\partial u_g^2 \partial u_{g'}^2} \right) \right].  \\
	\end{aligned}
	$$
	Let $C^{*(1)}(u)=\sum_{g,g' = 1}^{d} \frac{\partial^4 K(\|u\|)}{\partial u_g^2 \partial u_{g'}^2}$ and 
	\[r^{(1)}(u)=\sum_{g,g' = 1}^{d} \left[\Delta_{h,g}\Delta_{h,g'} [K(\|u\|) + r(u)] - \frac{\partial^4 K(\|u\|)}{\partial u_g^2 \partial u_{g'}^2} \right].\]
	Then $C^{(1)}(u)=C^{*(1)}(u)+r^{(1)}(u)$, and it is enough to show $C^{*(1)}$ and $r^{(1)}$ are of the form as in Assumption \ref{as:rd} with parameter $\alpha-4$.  
	
	As $A$ is even and analytic, then $\sum_{g,g'} \frac{\partial^4 A(\|u\|)}{\partial u_g^2 \partial u_{g'}^2} = A^{(1)}(\|u\|)$ for some even analytic function $A^{(1)}$ on $\mathbb R$ (Technical Lemma \ref{lem:norm}). Similarly, by Technical Lemma \ref{lem:normalpha}, $\sum_{g,g'} \frac{\partial^4 B(\|u\|)}{\partial u_g^2 \partial u_{g'}^2} = B^{(1)}(\|u\|)$ where $B^{(1)}$ is a function on $\mathbb R$ which has the same properties as the function $B$ in Assumption \ref{eq:K.assump}, with some constant $c^{(1)}$ and exponent $\alpha-4$. Let $K^{(1)}(\|u\|)=A^{(1)}(\|u\|)+B^{(1)}(\|u\|)$, we have shown that $C^*(u)=K^{(1)}(\|u\|)$ satisfies  the conditions of Assumption \ref{as:rd}. 
	
	Next we will show that $r^{(1)}(u)=o(\|u\|^{\alpha-4})$ which will complete the proof. Let $B^*(u) = B(\|u\|)$. We can write 
	\[
	\begin{aligned}
		r^{(1)}(u)&=\sum_{g,g' = 1}^{d} \left[\left( \Delta_{h,g}\Delta_{h,g'} C(u) - \frac{\partial^4 C(u)}{\partial u_g^2 \partial u_{g'}^2} \right)+  \frac{\partial^4 r(u)}{\partial u_g^2 \partial u_{g'}^2} \right] \\
		&=\sum_{g,g' = 1}^{d} \left[\left( \Delta_{h,g}\Delta_{h,g'} B^*(u) - \frac{\partial^4 B^*(u)}{\partial u_g^2 \partial u_{g'}^2} \right)+ o(\|u\|^{\alpha-4})
		\right]
	\end{aligned}
	\]
	In the expression above, the  term $\frac{\partial^4 r(u)}{\partial u_g^2 \partial u_{g'}^2}$ is immediately $o(\|u\|^{\alpha-4})$ from Assumption \ref{as:rd}. Also, the $O(1)$ term from the analytic part of $C$ is absorbed into the $o(\|u\|^{\alpha-4})$ term as $\alpha<4$. 
	
	Hence, we focus on the first term. Let $\rho_{gg'}=\Delta_{h,g}\Delta_{h,g'} C(u) - \frac{\partial^4 C(u)}{\partial u_g^2 \partial u_{g'}^2}$. 
	
	We first express
	$\rho_{gg'}\,$
	in terms of $4^{th}$ order mixed partial derivatives of $B$.
	$$
	\begin{aligned}
		& \Delta_{h,g}\Delta_{h,g'} B^*(u) \\
		&= \frac{1}{h^4} \left[ 4 B^*(u) - 2 B^*(u + h e_g) - 2 B^*(u - h e_g) - 2 B^*(u + h e_{g'}) - 2 B^*(u - h e_{g'}) \right. \\
		& \qquad+ B^*(u + h (e_g + e_{g'})) + B^*(u - h (e_g + e_{g'})) \\
		& \left. \qquad+ B^*(u + h (e_g - e_{g'})) + B^*(u - h (e_g - e_{g'})) \right]\\
		&= \frac{1}{h^2} \left[
		2 \left( \frac{2 B^*(u) - B^*(u + h e_g) - B^*(u - h e_g)}{h^2} \right)
		\right. \\
		&\quad \left.
		- \left( \frac{2 B^*(u + h e_{g'}) - B^*(u + h e_{g'} + h e_g) - B^*(u + h e_{g'} - h e_g)}{h^2} \right)
		\right. \\
		&\quad \left.
		- \left( \frac{2 B^*(u - h e_{g'}) - B^*(u - h e_{g'} + h e_g) - B^*(u - h e_{g'} - h e_g)}{h^2} \right)
		\right] \\
		&= \left[
		\frac{
			-2 \frac{\partial^2 B^*(u)}{\partial u_g^2} 
			+ \frac{\partial^2 B^*(u + h u_{g'})}{\partial u_g^2}
			+ \frac{\partial^2 B^*(u - h e_{g'})}{\partial u_g^2}}{h^2} \right.\\
		& \quad \left. + \frac{\frac {h}6 \left( \frac{2\partial^3 B^*(\xi_{1-})}{\partial u_g^3} - \frac{2\partial^3 B^*(\xi_{1+})}{\partial u_g^3} + \frac{\partial^3 B^*(\xi_{2+})}{\partial u_g^3} - \frac{\partial^3 B^*(\xi_{2-})}{\partial u_g^3} + \frac{\partial^3 B^*(\xi_{3+})}{\partial u_g^3} - \frac{\partial^{3-} B^*(\xi_{3Blue{-}})}{\partial u_g^3}\right) 
		}{h^2}
		\right]\\
		&= 
		\frac{\partial^4 B^*(u)}{\partial u_g^2 \partial u_{g'}^2} + \left(\frac 12 \frac{\partial^4 B^*(\xi^*_1)}{\partial u_g^2 \partial u_{g'}^2} + \frac 12 \frac{\partial^4 B^*(\xi^*_2)}{\partial u_g^2 \partial u_{g'}^2} - \frac{\partial^4 B^*(u)}{\partial u_g^2 \partial u_{g'}^2} \right) + \\
		& \quad \frac { -2\frac{\partial^4 B^*(\xi_{1})}{\partial u_g^4}(\xi_{1+}-\xi_{1-})_g + \frac{\partial^4 B^*(\xi_{2})}{\partial u_g^4}(\xi_{2+}-\xi_{2-})_g + \frac{\partial^4 B^*(\xi_{3})}{\partial u_g^4}(\xi_{3+}-\xi_{3-})_g}{6h} \\
		&= \frac{\partial^4 B^*(u)}{\partial u_g^2 \partial u_{g'}^2} 
		+ o(\|u\|^{\alpha-4}).
	\end{aligned}
	$$
	Here $\xi_{1\pm} \in (u-he_g,u+he_g)$, $\xi_{2\pm} \in (u+he_{g'}-he_g,u+he_{g'}+he_g)$, $\xi_{3\pm} \in (u-he_{g'}-he_g,u-he_{g'}+he_g)$, $\xi_i \in (\xi_{i-},\xi_{i+})$ for $i=1,2,3$ and $\xi^*_1 \in (u,u+he_{g'})$, $\xi_2^* \in (u-he_{g'},u)$. 
	Hence, $|(\xi_{i+}-\xi_{i-})_g| \leq 2h$, and, 
	$\|\xi_i\| < 2\|u\|$ and $\|\xi^*_i\| < 2\|u\|$ for all $i$, implying that all 
	$4^{th}$ order mixed derivatives of $B$ at $\xi_i$ or $\xi_i^*$ are $o(\|u\|^{\alpha-4})$ by Technical Lemma \ref{lem:lapgeneral}.
	This leads to the $o(\|u\|^{\alpha-4})$ term in the expression above, proving the result. 
	
\end{proof}

\begin{proof}[Proof of Theorem \ref{th:beta.con.Lap}]
	Let $N=|\calG_n^{(m)}|=(n+1-2m)^d$.
	We write $\mbox{Lap}_n^{(m)}(Z_1,Z_2) = T_2/T_1$ where 
	$$T_\ell = \frac 1{Nh^{\alpha_{11}-4m}} Z_1^{(m)T}Z_\ell^{(m)} = \frac 1{Nh^{\alpha_{11}-4m}} \sum_{s_i \in \calG_n^{(m)}} Z^{(m)}_{1}(s_i)Z^{(m)}_{\ell}(s_i).$$
	
	By the assumptions of the theorem, $\alpha_{11} < 4$ implies $\alpha_{k\ell} < 4$ for all $k,\ell$. 
	It is enough to prove the result for $m=1$, i.e., $\alpha_{k\ell} < 4$ for all $k,l$, as the proof for higher $m$ can be obtained by recursive use of Lemmas \ref{lem:lapmean} and \ref{lem:lapbound}.
	By Lemma \ref{lem:lapmean}, $E(Z_{1}^{(1)}(s_i)Z^{(1)}_{\ell}(s_i)) = \gamma c_{1\ell} h^{\alpha_{1\ell}-4} + o(h^{\alpha_{1\ell}-4})$
	where $\gamma$ is some constant (depending only on $d$ and $\alpha_{1\ell}$).
	As, by (\ref{eq:K12limit}), $\alpha_{11}=\alpha_{12}$, this implies $$E(T_\ell) = \frac 1{Nh^{\alpha_{11}-4}} \left(\sum_{i \in \calG_n^{(1)}} \gamma c_{1\ell} h^{\alpha_{11}-4} + o(h^{\alpha_{11}-4}) \right) \to \gamma c_{1\ell}.$$
	As (\ref{eq:K12limit}) also implies $c_{12}=\beta c_{11}$, we have $E(T_1) \to \gamma c_{11}$ and $E(T_2) \to \gamma \beta c_{11}$. So it is enough to show that $Var(t_\ell) \to 0$ implying $T_2/T_1 \to \beta$. 
	
	Let $Z^{(1)}=(Z_1^{(1)},Z_2^{(1)})$.
	and $C^{(1)}=\mbox{Cov}(Z^{(1)})$ on $\calG_n^{(1)}$, with blocks $C^{(1)}_{k\ell}$ for $1 \leq k,l \leq 2$. 
	$$Var(Z_1^{(1)T} Z_2^{(1)}) =
	\mbox{trace}\left[\left(C_{12}^{(1)}\right)^2\right] + \mbox{trace}(C^{(1)}_{11}C^{(1)}_{22}).
	$$
	
	Note that as $C_{12}$ is symmetric (Assumption \ref{as:rd}), it is immediate that so is $C^{(1)}_{12}$.
	So $$
	\begin{aligned}
		\mbox{trace}\left[\left(C_{12}^{(1)}\right)^2\right] = \mbox{trace}\left[C_{12}^{(1)}C_{12}^{T(1)}\right] =
		\sum_{s_i,s_j \in \calG_n^{(1)}} \left(C^{(1)}_{12}(s_i-s_j)\right)^2.
	\end{aligned}
	$$
	
	We can separate this sum into $s_i=s_j$ and $s_i \neq s_j$. By Lemma \ref{lem:lapmean}, we have 
	\begin{equation}\label{eq:match}
		\begin{aligned}
			\sum_{s_i \in \calS}
			\left(C^{(1)}_{12}(0)\right)^2 \asymp n^d
			h^{2\alpha_{11}-8}.
		\end{aligned}
	\end{equation}
	
	From Lemma \ref{lem:lapbound},
	we have for $u \neq 0$,
	
	$$
	\begin{aligned}
		\left(C^{(1)}_{12}(u)\right)^2 &=
		A^{(1)}_{12}(\|u\|)^2+c^2_{12}M^2(\alpha_{11},d)\|u\|^{2\alpha_{11}-8}  + o(\|u\|^{2\alpha_{11}-8}) \\
		&= O(1) + c^2_{12}M^2(\alpha_{11},d)O\left(\|u\|^{2\alpha_{11}-8}\right) + o(\|u\|^{2\alpha_{11}-8}) \\
		& \asymp \|u\|^{2\alpha_{11}-8}. \\
	\end{aligned}
	$$
	In the above, we have used the fact that $A^{(1)}_{12}$ is analytic, implying it is $O(1)$ in a bounded interval, and is dominated by $\|u\|^{2\alpha_{11}-8}$ as $\alpha_{11}<4$.
	Summing over all $s_i \neq s_j \in \calG_n^{(1)}$, we have 
	$$
	\begin{aligned}
		\sum_{s_i \neq s_j \in \calS}
		\left(C^{(1)}_{12}(s_i-s_j)\right)^2 \asymp
		\sum_{s_i \neq s_j \in \calG_n^{(1)}}  \|s_i - s_j\|^{2\alpha_{11}-8}.
	\end{aligned}
	$$
	
	As $\alpha_{11} < 4$, we have
	by Technical Lemma \ref{lem: grid},
	that the above summation is 
	
	\begin{equation}\label{eq:gridsum}
		\begin{aligned}
			\sum_{s_i \neq s_j \in \calS}
			\left(C^{(1)}_{12}(s_i-s_j)\right)^2&\asymp
			\begin{cases}
				n^{2d-8+2\alpha_{11}}\, h^{2\alpha_{11}-8}, & 0<8-2\alpha_{11}<d,\\[4pt]
				n^d \log n \, h^{2\alpha_{11}-8}, & 8-2\alpha_{11}=d,\\[4pt] 
				n^d \, h^{2\alpha_{11}-8}, & 8-2\alpha_{11}>d
			\end{cases} \\[8pt]
			&\asymp \begin{cases}
				n^{2d}, & 4 > \alpha_{11}> 4 - d/2,\\[4pt]
				n^{2d} \log n, & 4-\alpha_{11}=\frac d2,\\[4pt] 
				n^{d+8-2\alpha_{11}}, & 0 <\alpha_{11} < 4 - d/2,
			\end{cases}
		\end{aligned}
	\end{equation}
	
	Combining with \eqref{eq:match},
	$\mbox{trace}\left[\left(C_{12}^{(1)}\right)^2\right]$ has the same rate as (\ref{eq:gridsum}). 
	
	Similarly, as $\alpha_{22} \leq \alpha_{11}$, we have $8 - \alpha_{11} - \alpha_{22}> 0$, and by application of Lemma \ref{lem:lapmean}, Lemma \ref{lem:lapbound} and Technical Lemma \ref{lem: grid}, we get 
	
	\begin{equation}\label{eq:prodtrace}
		\mbox{trace}\left[(C^{(1)}_{11}C^{(1)}_{22})\right] \asymp
		\begin{cases}
			n^{2d}, & 4 > \frac{\alpha_{11}+\alpha_{22}}2> 4 - d/2,\\[4pt]
			n^{d+8-\alpha_{11}-\alpha_{22}}\log n \, , & \frac{\alpha_{11}+\alpha_{22}}2=4-d/2,\\[4pt]
			n^{d+8-\alpha_{11}-\alpha_{22}}, & 0 <\frac{\alpha_{11}+\alpha_{22}}2 < 4 - d/2,
		\end{cases}
	\end{equation}
	
	Let $\rho=\max\{2d,d+8-2\alpha_{11}\}$ and $\rho'=\max\{2d,d+8-\alpha_{11}-\alpha_{22}\}$. 
	We have 
	\begin{equation}\label{eq:raterd}
		\begin{aligned}
			\mbox{Var}(T_2) &= \frac 1{4n^{2d}h^{2\alpha_{11}-8}} 
			\left(\mbox{trace}(C_{12}^{(1)2})+ \mbox{trace}(C_{11}^{(1)}C_{22}^{(1)}) \right) \\
			&=\frac 1{4n^{2d+8-2\alpha_{11}}}\left(O(n^{\rho}) + O(n^{\rho'})\right) \\
			&= O\left(n^{-\min\{d,8-2\alpha_{11}\}} + O(n^{-\min\{d-\alpha_{11}+\alpha_{22},8-2\alpha_{11}\}}\right)\\
			& \to 0 \mbox{ as } \alpha_{11} < d + \alpha_{22}, \alpha_{11} < 4
		\end{aligned}
	\end{equation}
	Similarly, $\mbox{Var}(T_1)=O\left(n^{-\min\{d,8-2\alpha_{11}\}}\right)$. Hence, $\mbox{Var}(T_\ell) \to 0$ if $\alpha_{22} \leq \alpha_{11} < \alpha_{22} + d$ and the proof is complete. 
	
\end{proof}

\begin{proof}[\textbf{Proof of Theorem \ref{th:suff.rd}}]
	
	Proof of Theorem \ref{th:suff.rd} follows from Theorem \ref{th:beta.con.Lap} exactly as the proof of Theorem \ref{th:suff} follows from Theorem \ref{th:beta.con}.
\end{proof}

\subsection{Remaining proofs of results on equivalence of multivariate GRFs}\label{sec:texequiv}

\begin{proof}[\textbf{Proof of Theorem \ref{th:techequiv}}] This theorem is a generalization of Theorem 1 of \cite{bachoc2022asymptotically}, relaxing the assumption that the components in the multivariate GRFs have the same smoothness or tail behavior of the spectral density. That is, instead of Condition 1 of \cite{bachoc2022asymptotically}, we assume our Condition \eqref{eq:cond1}, which is weaker. Much of the proof of this theorem is identical to the proof of Theorem 1 of \cite{bachoc2022asymptotically}, and we only prove the parts where their Condition 1 was needed, using our Condition (\ref{eq:cond1}). 
	
	We first show that an integral operator on $\calW_\calD(\Fz)$ defined by 
	$$(Vf)(\mu) = \intrd B(\mu,\omega)\Fz(\omega)f(\omega)d\omega$$ is well-defined for almost all $\mu \in \mathbb R^d$. Then $\calW_\calD(F^{(i)})$ is a complex, separable Hilbert space. We note that under (\ref{eq:cond1}), $\calW_\calD(F^{(i)})$ is equipped with the corresponding inner product 
	\begin{equation}\label{eq:metriclim}
		\langle u,v \rangle_{F^{(i)}}
		= \int_{\mathbb R^d} u(\omega)^*F^{(i)}(\omega)v(\omega)d\omega.
	\end{equation} 
	
	Denoting the $i^{th}$ row of $B(\mu,\omega)$ by $b_i(\mu,\omega)^*$ we have by Cauchy-Schwartz inequality,
	$$
	\begin{aligned}
		& \intrd |b_i(\mu,\omega)^* \Fz(\omega) f(\omega)| d\omega \\
		&\leq  \left(   \intrd b_i(\mu,\omega)^* \Fz(\omega) b_i(\mu,\omega) d\omega \intrd f(\omega)^* \Fz(\omega) f(\omega)  d\omega \right)^{\frac 12}
	\end{aligned}
	$$
	Using (\ref{eq:polya2}), $\Fz \leq c_2 \Phi$ and $\sup_\omega \|\Phi(\omega)\| < M$, the first integral on the right-hand side above is finite for almost all $\mu \in \mathbb R^d$. Also, the second integral on the right is well-defined as $f \in \calW_\calD(\Fz)$. So, the integral operator $V$ is well-defined on $\calW_\calD(\Fz)$. 
	
	We next show that for any $f \in \calW_\calD(\Fz)$, $\|Vf\|_{\Fz}$ is finite. \cite{bachoc2022asymptotically} used their Condition 1 for this result. We show that this can be proved by using the multiplicative property of $\|\cdot\|$ norm, the Cauchy-Schwartz inequality, the equivalence of $\|\cdot\|$ and Frobenius norms for fixed-dimensional matrices, and $B \in \calW^2_{\calD}(\Fz)$.  We have  
	$$
	\begin{aligned}
		\|Vf\|^2_{\Fz} &= \intrd (Vf)(\mu)^* \Fz(\mu) (Vf)(\mu) d\mu \\
		& = \intrd \intrd \intrd f(\omega)^*\Fz(\omega) B(\mu,\omega)^*\Fz(\mu) B(\mu,\lambda)\Fz(\lambda)f(\lambda) d\omega\, d\lambda\,  d\mu \\
		& \leq \intrd \intrd \intrd \left( \sqrt{f(\omega)^*\Fz(\omega) f(\omega)} \|\Fz(\omega)^{1/2}B(\mu,\omega)^*\Fz(\mu)^{1/2}\|\right. \\  
		& \qquad \left. \| \Fz(\mu)^{1/2} B(\mu,\lambda)\Fz(\lambda)^{1/2}\| 
		\sqrt{f(\lambda)^*\Fz(\lambda) f(\lambda)} \right) d\omega\, d\mu\, d\lambda \\
		& \leq \sqrt{\intrd \intrd \intrd  f(\omega)^*\Fz(\omega) f(\omega) \|\Fz(\mu)^{ 1/2}B(\mu,\lambda)\Fz(\lambda)^{1/2}\|^2 d\omega\, d\mu \, d\lambda} \,\times \\
		& \qquad \sqrt{\intrd \intrd \intrd  f(\lambda)^*\Fz(\lambda) f(\lambda) \|\Fz(\omega)^{ 1/2}B(\mu,\omega)^*\Fz(\mu)^{1/2}\|^2 d\omega\, d\mu \, d\lambda} \\
		& \leq C \|f\|^2_{\Fz} \|B\|^2_{2,\Fz} < \infty. 
	\end{aligned}$$
	Here $C$ is some constant depending only on $p$.
	
	The operator $V$ was shown in \cite{bachoc2022asymptotically} to be a Hilbert-Schmidt operator, with eigenvalues $\lambda_k$ and eigenfunctions $g_k$ such that $Vg_k =\lambda_k g_k$ and $\sum_k \lambda_k^2 < \infty$.
	Let $u_{k,n}$ denote a sequence of functions in $\calW_\calD$ with $u_{k,n} \to g_k$ in $\calW_\calD(\Fz)$. Then $u_{k,n} = \calF(\phi_{k,n})$ for some square-integrable $\phi_{k,n}$ that is zero outside of $\calD$. For functions $u,v$ in $\calW_\calD(\Fz)$, define $$q(u,v)=\intrd\intrd u^*(\omega)\Fz(\omega)B(\omega,\mu)\Fz(\mu)v(\mu)  d\omega\, d\mu.$$
	\cite{bachoc2022asymptotically} used their Condition 1 to show that $q(u_{k,n},u_{j,n}) \to q(g_k,g_j)$. We prove the same result below using our assumptions. Note that $|q(u,v)| \leq \|u\|_{\Fz} \|v\|_{\Fz} \|B\|_{2,\Fz}$. We then have 
	$$
	\begin{aligned}
		|q(u_{k,n},u_{j,n}) - q(g_{k},g_{j})| & \leq  |q(u_{k,n} - g_{k},u_{j,n})| +  |q( g_k,u_{j,n}-g_j)| \\
		& \to 0 \mbox{ as } n \to \infty.
	\end{aligned}
	$$
	Here the limit follows as $\|u_{k,n} - g_k\|_{\Fz} \to 0$, $\|u_{j,n} - g_j\|_{\Fz} \to 0$, and $\|B\|_{2,\Fz}$, $\|g_k\|_{\Fz}$ and $\|g_j\|_{\Fz}$ are all finite. 
	
	Let $\calL_{2,p}(\calD)$ denote the Hilbert space of functions $h=(h_1,\ldots,h_p)^T$ from $\calD \to \mathbb C^p$ such that each $h_i$ is  square integrable on $\calD$. The space $\calL_{2,p}(\calD)$ is equipped with the inner product $\langle h,g\rangle_{\calD}=\sum_{i=1}^p \int_\calD \overline{h_i(t)}g_i(t)dt$. For $i = 0,1$, let $\calB^{(i)}$ denote the operator on $\calL_{2,p}(\calD)$ defined as 
	\begin{equation}\label{eq:opl2p}
		\calB^{(i)}(f)(s)=\int_\calD C^{(i)}(s-u)f(u)du.
	\end{equation}
	Let $\{h_k\}_k$ denote an orthonormal basis of $\calL_{2,p}(\calD)$ composed of the eigenfunctions of $\calB^{(0)}$ with eigenvalues $\rho_k$. Then \cite{bachoc2022asymptotically} uses their Condition (1) to prove that $\rho_k > 0$ for every $k$. This is also true assuming our condition \eqref{eq:cond1} as $F^{(i)} > c_1 \Phi$ for $i = 0,1$ and $\Phi$ is a diagonal matrix with strictly positive entries $\phi^2_i$ where $\phi_i \in \calW_\calD$. This proves that all the parts of the proof of Theorem 1 of \cite{bachoc2022asymptotically} which relied on their Condition (1) can be proved using the weaker condition (\ref{eq:equivgeneral}). The rest of the proof of Theorem \ref{th:techequiv} is identical to that of Theorem 1 of \cite{bachoc2022asymptotically}.
	
\end{proof}

\subsection{Proofs of results for specific covariance functions}\label{sec:proofexamples}

\begin{proof}[Proof of Corollary \ref{cor:powexp}] We first prove consistent estimability when $\delta_X < \delta_W + 1$. Let $(K_{k\ell})_{1 \leq k,l \leq 2}$ be the covariance function of $(X,W)$. By the statement of the corollary and using the expansion of the power exponential covariance near zero, each $K_{11}$, $K_{22}$ and $K_{12}$ respectively satisfies Assumption 1 with parameters $\alpha_{11}=\delta_X < 1 + \alpha_{22}$ where $\alpha_{22}=\delta_{W}$. Also, $\alpha_{12} > \alpha_{11}$. We can then directly apply Theorem \ref{th:suff} to establish consistent estimability. 
	
	Let $f_X(\omega)$ and $f_W(\omega)$ denote the spectral densities of $X$ and $W$ for $\omega \in \mathbb R$. By Theorem 1.1 of \cite{nolan2020univariate}, these densities are continuous.
	Also, by Theorem 1.2 of \cite{nolan2020univariate}, $f_X(\omega) \asymp (1+|\omega|)^{-\delta_X-1}$ and $f_W(\omega) \asymp (1+|\omega|)^{-\delta_W-1}$ as $\omega \to \pm \infty$. Finally, $f_X(\omega)$ and $f_W(\omega)$ are bounded away from $0$ and $\infty$ as $\omega \to 0$ \citep[Equation 2.2 of][]{garoni2002d}. Hence, there exists constants $c$ and $C$ such that 
	$$
	\begin{aligned}
		c (1+|\omega|)^{-\delta_X-1} &\leq   f_X(\omega) \leq C (1+|\omega|)^{-\delta_X-1}, \\
		c (1+|\omega|)^{-\delta_W-1} &\leq   f_W(\omega) \leq C (1+|\omega|)^{-\delta_W-1}. \\
	\end{aligned}
	$$
	Applying Equation \ref{eq:maternbound}, we have shown that $f_X$ and $f_W$ satisfy (\ref{eq:bound}). 
	
	For large enough $M$, as $\delta_X - \delta_W > 1$, we have 
	$$\int_{M}^\infty \frac{f_X(\omega)}{f_W(\omega)} d\omega \asymp \int_{M}^\infty \frac{1}{(1+|w|)^{\delta_W - \delta_X}}\, d\omega < \infty.$$
	By Theorem \ref{th:nonidgen}, $\beta$ is not consistently estimable. 
\end{proof}

\begin{proof}[Proof of Corollary \ref{cor:cauchy}] 
	For a small distance $h$, the Cauchy correlation function satisfies, $C(h) \asymp (1+\|h\|^\delta)^{-\kappa} = 1 - \kappa \|h\|^\delta + o(\|h\|^\delta)$ \citep[Equation (2.3) of ][]{lim2009gaussian}. Also, for large frequencies $\omega \in \mathbb R^d$, the spectral density $f(\omega) \asymp \|\omega\|^{-\delta-d}$ \citep[Equation (3.13) of ][]{lim2009gaussian}. Finally, note that the generalized Cuachy covariance at large distances $h \in \mathbb R^d$ behaves like $\|h\|^{-\kappa\delta}$. When $\kappa\delta > d$, $C(h)$ is integrable in $\mathbb R^d$, which implies continuity of the spectral density. In this scenario, the spectral density is also bounded away from $0$ and $\infty$ as $\omega \to 0$ \citep[Proposition (3.3) of ][]{lim2009gaussian} implying that we can write $ c (1+\|\omega\|)^{-\delta-d} < f(\omega) < C (1+\|\omega\|)^{-\delta-d}$ for some constants $c,C$ for all $\omega$. The rest of the proof is then exactly similar to the proofs of Corollaries \ref{cor:matern} and \ref{cor:powexp}.
\end{proof}

\begin{proof}[Proof of Corollary \ref{cor:lmc}]
	The proof for this corollary relies on the fact that the smoothness of a process that is a linear combination of independent component processes is the same as the smoothness of the roughest component process. Formally, as the covariance of each $U_r$ satisfies Assumption \ref{eq:K.assump} with $\delta_r$ being the exponent of the principal irregular term, the covariance of a linear combination of $U_r$ will also satisfy Assumption \ref{eq:K.assump} and the exponent of the principal irregular term will be the minimum $\delta_r$ from the $U_r$'s supporting the linear combination. This immediately allows the application of Theorem \ref{th:suff.rd} to establish consistent estimability of $\beta$ when $\delta_X < \delta_W + d$ and $\delta_{XW} > \delta_X$. 
	
	As each $U_r$ is either from the Matérn, power exponential, or generalized Cauchy family, its spectral density is well-behaved in the sense of being continuous, bounded away from $0$ and $\infty$ at low frequencies and having polynomial decay at high frequencies. So the spectral density of each of $X$ and $W$, being simply the linear combination of the component spectral densities, is also well behaved and satisfies condition \eqref{eq:bound}. Additionally, the spectral density of the coregionalization process, being a linear combination of independent component processes, has the same order of decay as the slowest decaying spectral density among the component processes. Hence, the spectral densities of $X$ and $W$ have algebraic decay with rates $\delta_X+d$ and $\delta_W+d$, respectively. When $\delta_X > \delta_W + d$, (\ref{eq:nec}) is satisfied and $\beta$ cannot be consistently estimable. 
\end{proof}

\subsection{Technical results on distances and radially symmetric functions}\label{sec:tl}

\begin{tl}\label{lem:norm} Let $A$ denote an even analytic function on $\mathbb R$ and $u \in \mathbb R^d$. Then $\sum_{g,g'} \frac{\partial^4 A(\|u\|)}{\partial u_g^2 \partial u_{g'}^2}=A^{(1)}(\|u\|)$ where $A^{(1)}$ is also an even analytic function on $\mathbb R$. 
\end{tl}

\begin{proof} Since $ A $ is an even analytic function on $ \mathbb{R}^d $, it can be expressed as a power series in terms of $r= \|u\|^2 $, that is,
	\[
	A(\|u\|) = \sum_{k=0}^{\infty} c_k \|u\|^{2k} = \sum_{k=0}^{\infty} c_k r^k.
	\]
	Computing the first derivative with respect to $ u_g $:
	\[
	\frac{\partial A}{\partial u_g} = \sum_{k=1}^{\infty} c_k \cdot 2k r^{k-1} u_g.
	\]
	Differentiating again:
	\[
	\frac{\partial^2 A}{\partial u_g^2} = \sum_{k=2}^{\infty} c_k \cdot 2k \left[ (2k-2) r^{k-2} u_g^2 + r^{k-1} \right].
	\]
	
	Now differentiating $ \frac{\partial^2 A}{\partial u_g^2} $ with respect to $ u_{g'} $:
	\[
	\frac{\partial^3 A}{\partial u_g^2 \partial u_{g'}} =
	\sum_{k=2}^{\infty} c_k \cdot 2k(2k-2) \left[ (2k-4) r^{k-3} u_g^2 u_{g'}  
	+ r^{k-2}u_{g'}  + \delta_{gg'} 2 r^{k-2} u_g \right] .
	\]
	$$
	\begin{aligned}
		\frac{\partial^4 A}{\partial u_g^2 \partial u_{g'}^2} &=
		\sum_{k=2}^{\infty} c_k 2k (2k-2) 
		\Big((2k-4) \left[ (2k-6) r^{k-4} u_g^2 u_{g'}^2 + r^{k-3}u_g^2 + r^{k-3}4\delta_{gg'}u_g^2 + \right.\\
		& \qquad \qquad \left. + r^{k-3}u_{g'}^2  
		\right]  + r^{k-2}(2\delta_{gg'} +1)\Big).    
	\end{aligned}
	$$
	Summing over $ g, g' $, we obtain
	$$
	\begin{aligned}
		\sum_{g,g'} \frac{\partial^4 A}{\partial u_g^2 \partial u_{g'}^2} &=
		\sum_{k=2}^{\infty} c_k 2k (2k-2) 
		\Big((2k-4) \left[ (2k-6) r^{k-2}  + 2dr^{k-2} + 4r^{k-2} 
		\right]  \\
		& \qquad\qquad + (2d+d^2)r^{k-2} \Big).\\
		& = \sum_{k=0}^\infty c'_k r^k = \sum_{k=0}^\infty c'_k \|u\|^{2k} = A^{(1)}(\|u\|). 
	\end{aligned}
	$$
	This completes the proof.
\end{proof}

\begin{tl}\label{lem:normalpha} Suppose $B$
	is the function as in Assumption \ref{eq:K.assump} for
	some constant $c$ and exponent $\alpha$
	and $u \in \mathbb R^d$. 
	Then $\sum_{g,g'} \frac{\partial^4 B(\|u\|)}{\partial u_g^2 \partial u_{g'}^2} = B^{(1)}(\|u\|)$ where $B^{(1)}$ also is of the same form of $B$ as in Assumption \ref{eq:K.assump} for some constant $c_1$ and exponent $\alpha_1$ where $c_1$ depends only on $c$, $\alpha$ and $d$,  and $\alpha_1=\alpha-4$.
\end{tl}

\begin{proof} Note that $B(t) = ct^\alpha + o(t^\alpha)$ for some $c,\alpha$.
	Let $r = \|u\|^2$ and $f(u)=\|u\|^\alpha = r^{\frac \alpha 2}$. Then for any $u \neq 0$, we have 
	
	$$
	\begin{aligned}
		\frac{\partial f}{\partial u_g} &= \alpha r^{\frac \alpha 2-1} u_g.\\
		\frac{\partial^2 f}{\partial u_g^2} &= \alpha \left[( \alpha-2) r^{\frac \alpha 2-2} u_g^2 + r^{\frac \alpha 2-1} \right].
		\\
		\frac{\partial^3 f}{\partial u_g^2 \partial u_{g'}} &= \alpha( \alpha-2) \left[ (\alpha-4) r^{\frac \alpha 2-3} u_g^2 u_{g'} + r^{\frac \alpha 2-2} 2\delta_{gg'} u_g  + r^{\frac \alpha 2-2}u_{g'} \right].\\
		\frac{\partial^4 f}{\partial u_g^2 \partial u_{g'}^2} &= \alpha( \alpha-2) \left[ (\alpha-4) \left( (\alpha-6) r^{\frac \alpha 2-4} u_g^2 u_{g'}^2 + r^{\frac \alpha 2-3} u_g^2 + r^{\frac \alpha 2-3}4\delta_{gg'}u_g^2  + r^{\frac \alpha 2-3}u^2_{g'} \right)\right.\\
		& \qquad\qquad\qquad \left. + \, r^{\frac \alpha 2-2} (2\delta_{gg'} + 1)  \right].\\
	\end{aligned}
	$$
	Summing over all $g,g'$, we have, $$
	\begin{aligned}
		\sum_{g,g'} \frac{\partial^4 f}{\partial u_g^2 \partial u_{g'}^2} &= \alpha(\alpha-2) \left[ (\alpha-4) \left( (\alpha-6) + 2d + 4 \right) + 2d + d^2  \right] \|u\|^{\alpha-4} \\
		&= M(\alpha,d) \|u\|^{\alpha-4}. 
	\end{aligned}
	$$
	
	Let $R(t)=B(t) - ct^\alpha$, then by Assumption \ref{eq:K.assump},
	the $k^{th}$ derivative of $R(t)$ is $o(t^{\alpha-k})$. Then by Technical Lemma \ref{lem:lapgeneral},
	\[\sum_{g,g'} \frac{\partial^4 R(\|u\|)}{\partial u_g^2 \partial u_{g'}^2} = o(\|u\|^{\alpha-4}) = R^{(1)}(\|u\|).
	\]
	Defining $B^{(1)}(t)=cM(\alpha,d)t^{\alpha-4} + R^{(1)}(t)$, we have our result.
\end{proof}

\begin{tl}\label{lem:lapgeneral}
	Let $ f:\mathbb{R}\rightarrow \mathbb{R} $ be a four times differentiable function, and let $ u\in \mathbb{R}^d $. Define $ r=\|u\|=\sqrt{u_1^2+\dots+u_d^2} $. Then for any $u \neq0$, the following identity holds:
	\[
	\sum_{g,g'=1}^{d}\frac{\partial^4 f(\|u\|)}{\partial u_g^2 \partial u_{g'}^2}
	= f^{(4)}(r) 
	+ \frac{2(d-1)}{r}f^{(3)}(r) 
	+ \frac{(d-1)(d-3)}{r^2}f''(r)
	- \frac{(d-1)(d-3)}{r^3}f'(r).
	\]
\end{tl}

\begin{proof}
	Let $ r=\|u\| $. First note that
	\[
	\frac{\partial f(r)}{\partial u_g} = f'(r)\frac{u_g}{r}.
	\]
	Differentiating again with respect to $u_g$, we obtain
	\[
	\frac{\partial^2 f(r)}{\partial u_g^2} 
	= f''(r)\frac{u_g^2}{r^2} + f'(r)\frac{r^2 - u_g^2}{r^3}.
	\]
	Summing this second derivative over $g=1,\dots,d$, we have the radial Laplacian:
	\[
	\sum_{g=1}^{d}\frac{\partial^2 f(r)}{\partial u_g^2} 
	= f''(r) + \frac{d-1}{r}f'(r).
	\]
	
	Define
	\[
	h(r)=f''(r)+\frac{d-1}{r}f'(r).
	\]
	Then the expression we aim to compute becomes
	\[
	\sum_{g,g'=1}^{d}\frac{\partial^4 f(r)}{\partial u_g^2 \partial u_{g'}^2}
	=\sum_{g'=1}^{d}\frac{\partial^2 h(r)}{\partial u_{g'}^2}.
	\]
	Differentiating $ h(r) $, we have
	\[
	\frac{\partial h(r)}{\partial u_{g'}}=h'(r)\frac{u_g}{r},\quad
	\frac{\partial^2 h(r)}{\partial u_{g'}^2}=h''(r)\frac{u_{g'}^2}{r^2}+h'(r)\frac{r^2-u_{g'}^2}{r^3}.
	\]
	Summing over $ g'=1,\dots,d $, and noting $\sum_{g'=1}^{d}u_{g'}^2=r^2$, we get
	\[
	\sum_{g'=1}^{d}\frac{\partial^2 h(r)}{\partial u_{g'}^2}=h''(r)+(d-1)\frac{h'(r)}{r}.
	\]
	We now substitute back the definition of $h(r)$:
	\[
	h'(r)=f'''(r)+(d-1)\left(\frac{f''(r)}{r}-\frac{f'(r)}{r^2}\right),
	\]
	and
	\[
	h''(r)=f^{(4)}(r)+(d-1)\frac{d}{dr}\left(\frac{f''(r)}{r}-\frac{f'(r)}{r^2}\right).
	\]
	
	Evaluating this explicitly, we have
	\[
	h''(r)=f^{(4)}(r)+(d-1)\left(\frac{f'''(r)r - f''(r)}{r^2}-\frac{f''(r)r^2 - 2r f'(r)}{r^4}\right).
	\]
	Simplifying the expression carefully, we arrive at the identity
	\[
	\sum_{g,g'=1}^{d}\frac{\partial^4 f(r)}{\partial u_g^2 \partial u_{g'}^2}
	=f^{(4)}(r) 
	+\frac{2(d-1)}{r}f^{(3)}(r) 
	+\frac{(d-1)(d-3)}{r^2}f''(r)
	-\frac{(d-1)(d-3)}{r^3}f'(r).
	\]
	This completes the proof.
\end{proof}

\begin{tl}\label{lem: grid}
	Let $d\in\mathbb{N}$ and $\calI_n=\{1,2,\ldots,n\}^d\subset\mathbb{R}^d$. For $\alpha>0$ define
	\[
	S_{n,\alpha}=\sum_{\substack{u,v\in\calI_n\\u\neq v}}\frac{1}{\|u-v\|^\alpha}.
	\]
	Then, as $n\to\infty$,
	\[
	S_{n,\alpha}\asymp
	\begin{cases}
		n^{2d-\alpha}, & 0<\alpha<d,\\[4pt]
		n^d\log n, & \alpha=d,\\[4pt]
		n^d, & \alpha>d,
	\end{cases}
	\]
	with constants depending only on $d$ and $\alpha$.
\end{tl}

\begin{proof}
	Write $h\in\mathbb{Z}^d\setminus\{0\}$ and let $N_n(h)$ be the number of pairs $(u,v)\in\calI_n^2$ with $v-u=h$. A direct count gives
	\[
	N_n(h)=\prod_{i=1}^d\bigl(n-|h_i|\bigr)_+\quad\text{where}\quad (x)_+=\max\{x,0\}.
	\]
	Hence
	\[
	S_{n,\alpha}=\sum_{\substack{h\in\mathbb{Z}^d\\h\neq 0}}\frac{N_n(h)}{\|h\|^\alpha}.
	\]
	For the upper bound, $N_n(h)\le n^d$ and $N_n(h)=0$ if $\|h\|_\infty\ge n$, so
	\[
	S_{n,\alpha}\le n^d\sum_{\substack{h\in\mathbb{Z}^d\\0<\|h\|_\infty\le n}}\frac{1}{\|h\|^\alpha}.
	\]
	For the lower bound, if $\|h\|_\infty\le n/2$ then $N_n(h)\ge (n/2)^d$, so
	\[
	S_{n,\alpha}\ge \Bigl(\frac{n}{2}\Bigr)^{\!d}\sum_{\substack{h\in\mathbb{Z}^d\\0<\|h\|_\infty\le n/2}}\frac{1}{\|h\|^\alpha}.
	\]
	Therefore there exist constants $c_1,c_2>0$ depending only on $d$ such that
	\[
	c_1\,n^d\sum_{0<\|h\|_\infty\le n/2}\frac{1}{\|h\|^\alpha}
	\;\le\;
	S_{n,\alpha}
	\;\le\;
	c_2\,n^d\sum_{0<\|h\|_\infty\le n}\frac{1}{\|h\|^\alpha}.
	\]
	Compare the lattice sums with shell counts. There exist constants $a_d,b_d,c_d>0$ depending only on $d$ such that for all integers $m\ge 1$,
	\[
	a_d\sum_{r=1}^{m} r^{d-1-\alpha}
	\;\le\;
	\sum_{0<\|h\|_\infty\le m}\frac{1}{\|h\|^\alpha}
	\;\le\;
	b_d\sum_{r=1}^{\lceil c_d m\rceil} r^{d-1-\alpha}.
	\]
	By the integral test,
	\[
	\sum_{r=1}^{m} r^{d-1-\alpha}\asymp
	\begin{cases}
		m^{\,d-\alpha}, & 0<\alpha<d,\\[4pt]
		\log m, & \alpha=d,\\[4pt]
		1, & \alpha>d.
	\end{cases}
	\]
	Taking $m$ proportional to $n$ and combining with the two sided bound for $S_{n,\alpha}$ yields
	\[
	S_{n,\alpha}\asymp
	\begin{cases}
		n^{2d-\alpha}, & 0<\alpha<d,\\[4pt]
		n^d\log n, & \alpha=d,\\[4pt]
		n^d, & \alpha>d,
	\end{cases}
	\]
	which completes the proof.
\end{proof}

\subsection{Proofs for non-stationary and non-Gaussian examples of Section \ref{sec:nsandng}}\label{sec:pfns}

\begin{proof}[\textbf{Proof of Theorem \ref{th:ns}}] For simplicity of the expressions we take $L=1$. More generally, the expressions will involve a fixed $L>0$ which does not change the rates. The error of the difference-based OLS estimator is given by $\frac {U_n}{D_n}$ where
	
	\[
	\begin{aligned}
		U_n &= \frac 1{nh^{\alpha_{11}}} \sum_{i=1}^n (X(hi) - X(h(i-1)))(W(hi) - W(h(i-1))), \\
		D_n &= \frac 1{nh^{\alpha_{11}}} \sum_{i=1}^n (X(hi) - X(h(i-1))^2.
	\end{aligned}
	\]
	By the constraints on $\alpha_{11}$, using Lemmas \ref{lem:nsmean} and \ref{lem:nsvar}, we have $E U_n \to 0$, $\var(U_n) \to 0$, $E D_n \to -2\int_{0}^{1} \Gamma_{\alpha_{11}}(s,s)\,ds \neq 0$, $\var(D_n) \to 0$ implying $U_n/D_n \to 0$. 
\end{proof}

\begin{lemma}\label{lem:nsmean}
	Let $K_\alpha$ be as in (\ref{eq:nsclass}) for $L=1$, and let $s_i = hi$. Define $\Delta_{h}K_\alpha(s_{i},s_{i})
	= K_\alpha(s_{i},s_{i}) - K_\alpha(s_{i},s_{i-1})
	- K_\alpha(s_{i-1},s_{i}) + K_\alpha(s_{i-1},s_{i-1})$ for $i=1,\ldots,n$. Let
	$
	S_{n} 
	= \frac{1}{n h^{\alpha_{11}}} \sum_{i=1}^{n} \Delta_{h}K_{\alpha}(s_{i},s_{i})
	$ and $\alpha_{11} < 2$. 
	(i) If $\alpha > \alpha_{11}$, then $S_{n} \to 0$. 
	(ii) If $\alpha = \alpha_{11}$, then
	$
	S_{n} \longrightarrow -2 \int_{0}^{1} \Gamma_{\alpha}(s,s)\,ds.
	$
\end{lemma}

\begin{proof}
	Let $K=K_\alpha$ and write $K = K^{\mathrm{an}} + H + r$ with 
	\[
	K^{\mathrm{an}}(s,s') = \Gamma_{0}(s,s') + \Gamma_{2}(s,s')(s-s')^{2},
	\qquad
	H(s,s') = \Gamma_{\alpha}(s,s')\,|s-s'|^{\alpha},
	\]
	and $r$ being the remainder part. 
	
	First, consider the analytic part. Let $x_{i} = (i-\tfrac12)h$ for $i=1,\ldots,n$. A Taylor expansion of $K^{\mathrm{an}}$ around $(x_{i},x_{i})$ gives
	\[
	\Delta_{h} K^{\mathrm{an}}(s_{i},s_{i})
	= h^{2}\,\partial_1 \partial_2 K^{\mathrm{an}}(x_{i},x_{i}) + O(h^{3})
	= O(h^{2}),
	\]
	uniformly in $i$ (as $K^{an}$ is analytic). Hence
	\[
	\frac{1}{n h^{\alpha_{11}}} \sum_{i=1}^{n} \Delta_{h}K^{\mathrm{an}}(s_{i},s_{i})
	\leq h^{2-\alpha_{11}} \sup_i \partial_1 \partial_2K^{an}(x_i,x_i) +o(h^{2-\alpha_{11}}) = O\big(h^{2-\alpha_{11}}\big) \to 0
	\]
	since $\alpha_{11}<2$.
	
	Next, we treat the irregular term. On the diagonal, $H(s_{i},s_{i}) = H(s_{i-1},s_{i-1}) = 0$ and $|s_{i}-s_{i-1}| = h$. Then
	\[
	\begin{aligned}
		\Delta_{h}H(s_{i},s_{i})
		&= H(s_{i},s_{i}) - H(s_{i},s_{i-1})
		- H(s_{i-1},s_{i}) + H(s_{i-1},s_{i-1}) \\
		&= - H(s_{i},s_{i-1}) - H(s_{i-1},s_{i}) \\
		&= -\Gamma_{\alpha}(s_{i},s_{i-1})\,h^{\alpha}
		-\Gamma_{\alpha}(s_{i-1},s_{i})\,h^{\alpha}\\
		&= -2\,\Gamma_{\alpha}(x_{i},x_{i})\,h^{\alpha}
		+ O\bigl(h^{\alpha+1}\bigr),
	\end{aligned}
	\]
	where the last line uses analyticity of $\Gamma_{\alpha}$.
	
	Hence
	\[
	\frac{\Delta_{h}H(s_{i},s_{i})}{h^{\alpha_{11}}}
	=
	-2\,\Gamma_{\alpha}(x_{i},x_{i})\, h^{\alpha-\alpha_{11}}
	+ O\big(h^{\alpha-\alpha_{11}+1}\big).
	\]
	
	For the remainder term, note that for $|s-s'| \in \{0,h\}$ we have
	\[
	r(s,s') = o(|s-s'|^{\alpha'}) = o(h^{\alpha'}),
	\]
	so $\Delta_{h} r(s_{i},s_{i}) = o(h^{\alpha'})$ uniformly in $i$, and therefore
	\[
	\frac{\Delta_{h} r(s_{i},s_{i})}{h^{\alpha_{11}}}
	= o\big(h^{\alpha'-\alpha_{11}}\big),
	\]
	again uniformly in $i$.
	
	If $\alpha > \alpha_{11}$, then $\alpha' = \min(\alpha,2) > \alpha_{11}$, so both
	$
	\Gamma_{\alpha}(x_{i},x_{i})\, h^{\alpha-\alpha_{11}}
	$ and the $o(h^{\alpha'-\alpha_{11}})$ term
	tend to zero uniformly in $i$. Summing over $i$ and dividing by $n$ preserves convergence to zero, which gives $S_{n} \to 0$.
	
	If $\alpha = \alpha_{11} < 2$, then $\alpha' = \alpha_{11}$ and
	\[
	\frac{\Delta_{h} H(s_{i},s_{i})}{h^{\alpha_{11}}}
	= -2\,\Gamma_{\alpha_{11}}(x_{i},x_{i}) + O(h),
	\qquad
	\frac{\Delta_{h} r(s_{i},s_{i})}{h^{\alpha_{11}}} = o(1),
	\]
	uniformly in $i$. Thus
	\[
	S_{n}
	= -\frac{2}{n} \sum_{i=1}^{n} \Gamma_{\alpha_{11}}(x_{i},x_{i}) + o(1).
	\]
	Since $\Gamma_{\alpha_{11}}$ is continuous on $[0,1]^{2}$, the Riemann sum converges:
	\[
	\frac{1}{n} \sum_{i=1}^{n} \Gamma_{\alpha_{11}}(x_{i},x_{i})
	\longrightarrow \int_{0}^{1} \Gamma_{\alpha_{11}}(s,s)\,ds,
	\]
	which proves the claimed limit in the case $\alpha = \alpha_{11}$.
\end{proof}

\begin{lemma}\label{lem:nsvar}
	Let $K_{\alpha_k}$, $k=1,2$ be functions as of the form (\ref{eq:nsclass}) for $L=1$. Define
	\[
	T_{n} 
	= \frac{1}{n^{2} h^{2\alpha_{11}}}
	\sum_{i=1}^{n}\sum_{j=1}^{n}
	\bigl(\Delta_{h} K_{\alpha_1}(s_{i},s_{j})\bigr)\bigl(\Delta_{h} K_{\alpha_2}(s_{i},s_{j})\bigr),
	\]
	where $h = 1/n$ and $s_i = ih$. If $\alpha_{1} + \alpha_{2}> 2\alpha_{11} - 1$ and $\alpha_{11}<2$, then $T_{n} \to 0$.
\end{lemma}

\begin{proof}
	Write, for $k=1,2$,
	\[
	K_{\alpha_k} = K_k^{\mathrm{an}} + H_k + r_k,
	\]
	with
	\[
	K_k^{\mathrm{an}}(s,s') = \Gamma_{0}(s,s') + \Gamma_{2}(s,s')(s-s')^{2},
	\qquad
	H_k(s,s') = \Gamma_{\alpha_k}(s,s')\,|s-s'|^{\alpha_k},
	\]
	and $r_k$ is the same remainder form as in the mean result. All the $\Gamma$’s are real analytic on $[0,1]^2$. Thus
	\[
	\Delta_h K_{\alpha_k}
	= \Delta_h K_k^{\mathrm{an}} + \Delta_h H_k + \Delta_h r_k.
	\]
	
	Expanding the product
	\[
	\Delta_h K_{\alpha_1}\,\Delta_h K_{\alpha_2}, 
	\]
	we get a sum of nine interaction terms. We prove for the case where $\alpha_k <2$ for $k=1,2$ where the interaction of the analytic terms $H_k$ will dominate this product. The other cases follow similarly as $\alpha_{11}<2$. As in the proof of Lemma \ref{lem:nsmean}, 
	$\Delta_h K_k^{\mathrm{an}}(s_i,s_j) = O(h^{2})$ uniformly in $i,j$, and each $\Delta_h r_k$ carries an extra small-$o$ factor relative to $\Delta_h H_k$ from the assumptions on the remainders. Hence every interaction involving at least one factor $\Delta_h K_k^{\mathrm{an}}$ or $\Delta_h r_k$ is of strictly higher order in $h$ than the pure $\Delta_h H_1\,\Delta_h H_2$ term. It therefore suffices to show that the $H_1H_2$ contribution to $T_n$ tends to zero.
	
	We split this sum into diagonal and off-diagonal contributions.
	
	\medskip
	\noindent\emph{Diagonal terms $i=j$.}
	
	On the diagonal we have $|s_i - s_{i-1}| = h$ and, as in the proof of Lemma \ref{lem:nsmean},
	\[
	\Delta_h H_k(s_i,s_i) = O\bigl(h^{\alpha_k}\bigr), \qquad k=1,2,
	\]
	uniformly in $i$. Therefore
	\[
	\Delta_h H_1(s_i,s_i)\,\Delta_h H_2(s_i,s_i)
	= O\bigl(h^{\alpha_1+\alpha_2}\bigr),
	\]
	and there are $n$ such diagonal indices. Hence
	\[
	\begin{aligned}
		\frac{1}{n^{2} h^{2\alpha_{11}}}
		\sum_{i=1}^{n} \Delta_h H_1(s_i,s_i)\,\Delta_h H_2(s_i,s_i)
		&= O\!\left(
		\frac{n h^{\alpha_1+\alpha_2}}{n^{2} h^{2\alpha_{11}}}
		\right)
		= O\bigl(h^{\alpha_1+\alpha_2 - 2\alpha_{11} + 1}\bigr).
	\end{aligned}
	\]
	Under the condition $\alpha_1+\alpha_2 > 2\alpha_{11}-1$, this exponent is positive, so the diagonal contribution to $T_n$ tends to zero.
	
	\medskip
	\noindent\emph{Off–diagonal terms $i\neq j$.}
	
	For $i \neq j$, set $t_{ij} = |s_i - s_j| = |i-j|h \in [h,1]$ and define midpoints
	\[
	x_i = s_i - \tfrac{h}{2}, 
	\qquad
	x_j = s_j - \tfrac{h}{2}.
	\]
	Then the four arguments in $\Delta_h H_k(s_i,s_j)$ are
	\[
	(x_i \pm \tfrac{h}{2},\, x_j \pm \tfrac{h}{2}).
	\]
	Since $H_k$ is real analytic on $\{(s,s') : s\neq s'\}$, it admits a two-dimensional Taylor expansion around $(x_i,x_j)$:
	\[
	H_k(x_i+\delta_1,x_j+\delta_2)
	=
	H_k(x_i,x_j) + L_k(\delta_1,\delta_2) + R_k(\delta_1,\delta_2),
	\]
	where $L_k$ is the linear part in $(\delta_1,\delta_2)$, and the remainder satisfies
	\[
	\bigl|R_k(\delta_1,\delta_2)\bigr|
	\le C_k \bigl(|\delta_1|+|\delta_2|\bigr)^{2}\,
	t_{ij}^{\alpha_k-2}
	\]
	for some constant $C_k$ independent of $i,j,h$. 
	
	Evaluating this expansion at the four corners $(\delta_1,\delta_2)\in\{\pm h/2\}\times\{\pm h/2\}$ and forming the alternating sum that defines $\Delta_h$, the constant term $H_k(x_i,x_j)$ and all the linear terms in $L_k$ cancel by symmetry. Hence
	\[
	\Delta_h H_k(s_i,s_j)
	=
	R_k\!\bigl(\tfrac{h}{2},\tfrac{h}{2}\bigr)
	- R_k\!\bigl(\tfrac{h}{2},-\tfrac{h}{2}\bigr)
	- R_k\!\bigl(-\tfrac{h}{2},\tfrac{h}{2}\bigr)
	+ R_k\!\bigl(-\tfrac{h}{2},-\tfrac{h}{2}\bigr),
	\]
	so, using the bound on $R_k$,
	\[
	\bigl|\Delta_h H_k(s_i,s_j)\bigr|
	\le 4 \max_{|\delta_1|,|\delta_2|\le h/2} \bigl|R_k(\delta_1,\delta_2)\bigr|
	\le C_k' h^{2} t_{ij}^{\alpha_k-2},
	\]
	for a constant $C_k'$ independent of $i,j,h$.
	
	Hence, for $i\neq j$,
	\[
	\bigl|\Delta_h H_1(s_i,s_j)\,\Delta_h H_2(s_i,s_j)\bigr|
	\le C\,h^{4} t_{ij}^{\alpha_1+\alpha_2-4},
	\]
	for some $C$ independent of $i,j,h$.
	
	Summing over all off–diagonal pairs,
	\[
	\sum_{i\neq j} \Delta_h H_1(s_i,s_j)\,\Delta_h H_2(s_i,s_j)
	\asymp 
	h^{4} \sum_{i\neq j} t_{ij}^{\alpha_1+\alpha_2-4}
	.
	\]
	Now $t_{ij} = |i-j|h$, and
	\[
	\sum_{i\neq j} t_{ij}^{\alpha_1+\alpha_2-4}
	= \sum_{i\neq j} \bigl(|i-j|h\bigr)^{\alpha_1+\alpha_2-4}
	= 2\sum_{m=1}^{n-1} (n-m)\,(mh)^{\alpha_1+\alpha_2-4}.
	\]
	This yields
	\[
	\sum_{i\neq j} t_{ij}^{\alpha_1+\alpha_2-4}
	\asymp n \sum_{m=1}^{n-1} (mh)^{\alpha_1+\alpha_2-4}
	\asymp n\,h^{-1} \int_{h}^{1} t^{\alpha_1+\alpha_2-4}\,dt,
	\]
	so, using $n = 1/h$,
	\[
	\sum_{i\neq j} \Delta_h H_1(s_i,s_j)\,\Delta_h H_2(s_i,s_j)
	\asymp 
	h^{4} \cdot \frac{1}{h^{2}} \int_{h}^{1} t^{\alpha_1+\alpha_2-4}\,dt
	\asymp
	h^{2} \int_{h}^{1} t^{\alpha_1+\alpha_2-4}\,dt
	.
	\]
	
	Therefore, the off–diagonal contribution to $T_n$ satisfies
	\[
	\frac{1}{n^{2} h^{2\alpha_{11}}}
	\sum_{i\neq j} \Delta_h H_1(s_i,s_j)\,\Delta_h H_2(s_i,s_j)
	\asymp 
	h^{4-2\alpha_{11}} \int_{h}^{1} t^{\alpha_1+\alpha_2-4}\,dt.
	\]
	Let $\gamma = \alpha_1+\alpha_2-4$. Then
	\[
	\int_{h}^{1} t^{\gamma}\,dt
	=
	\begin{cases}
		O(1), & \gamma > -1 \ (\alpha_1+\alpha_2>3),\\[0.2em]
		O\!\bigl(\log(1/h)\bigr), & \gamma = -1 \ (\alpha_1+\alpha_2=3),\\[0.2em]
		O\bigl(h^{\gamma+1}\bigr) = O\bigl(h^{\alpha_1+\alpha_2-3}\bigr), & \gamma < -1 \ (\alpha_1+\alpha_2<3).
	\end{cases}
	\]
	In all three cases, since $\alpha_{11}<2$ implies $4-2\alpha_{11}>0$, the factor $h^{4-2\alpha_{11}}$ drives this off–diagonal contribution to zero as $h\to 0$. For example, in the case $\alpha_1+\alpha_2<3$ we obtain the sharper bound
	\[
	h^{4-2\alpha_{11}} \int_{h}^{1} t^{\gamma}\,dt
	\asymp h^{\alpha_1+\alpha_2-2\alpha_{11}+1},
	\]
	which tends to zero under the condition $\alpha_1+\alpha_2>2\alpha_{11}-1$.
	
	Combining the diagonal and off–diagonal bounds, and recalling that all mixed terms involving $K^{\mathrm{an}}$ or $r_k$ are of strictly higher order in $h$, we conclude that
	\[
	T_n \longrightarrow 0
	\]
	whenever $\alpha_1+\alpha_2>2\alpha_{11}-1$ and $\alpha_{11}<2$.
\end{proof}

\begin{proof}[\textbf{Proof of Corollary \ref{cor:ns_warped}}] 
	
	Proof of Consistency (Part (a)):
	
	We first do the case where $f$ is monotonic (without loss of generality, increasing) and analytic. The piecewise analytic case is proved afterwards. 
	
	Define the analytic extension of derivative of $f$ to $[0,L]^2$ as 
	\[
	\begin{aligned}
		H(s,s') &= \frac {f(s) - f(s')}{s-s'} \mbox{ if } s \neq s' \\
		&= f'(s) \mbox{ if } s=s'. 
	\end{aligned} 
	\]
	Then $H$ is analytic on $[0,L]^2$. Using the monotonicity of $f$, we have $H(s,s') \geq 0$ for all $s,s'$, and thus $|f(s) - f(s')| = H(s,s') |s-s'|$ for all $s,s'$.  
	
	As $K_{k\ell}$ satisfies Assumption \ref{eq:K.assump}, for $t \geq 0$, $K(t) = a_{k\ell,0} + a_{k\ell,2}t^2 + c_{k\ell}t^{\alpha_{k\ell}} + r_{kl}(t)$ where $r^{(m)}_{kl}(t)=c_m o(t^{\alpha'_{k\ell}-m})$ with $\alpha'_{k\ell} = \min\{\alpha_{k\ell},2\}$. 
	
	Let $(K^*_{k\ell})$ denote the covariance function of $(X^*,W^*)$. Then for any $s,s'$, we have,
	\[
	\begin{aligned}
		K_{kl}^*(s,s') &= K_{kl}(f(s),f(s'))\\
		&= a_{kl,0} + a_{kl,2} H(s,s')^2 |s-s'|^2 + c_{kl} H(s,s')^{\alpha_{k\ell}} |s-s'|^{\alpha_{k\ell}} + r_{kl}\bigl(H(s,s')|s-s'|\bigr).
	\end{aligned}
	\]
	Thus $K^*_{k\ell}$ is of the form (\ref{eq:nsclass}) with $\int \Gamma_{\alpha_{11}}(s,s)ds = \int H(s,s)ds = \int f'(s)ds > 0$ (as $f$ is increasing and not constant). The result follows from Theorem \ref{th:ns}.  
	
	Now consider the piecewise case. It suffices to do this for the two-piece case, as the same idea generalizes. Let there by some $\tau \in [0,1]$ such that $f$ is analytic monotonic in $[0,\tau-\epsilon]$ and in $[\tau+\epsilon,1]$ for any $\epsilon > 0$. 
	
	To handle the breakpoint, split the index set using a buffer around $\tau$:
	\[
	\begin{aligned}
		I_1 &= \{ i : s_{i+1} \le \tau - h \},\\
		I_2 &= \{ i : s_i \ge \tau + h \},\\
		I_3 &= \{0,\dots,n-1\} \setminus (I_1 \cup I_2).
	\end{aligned}
	\]
	Because the grid is regular and the buffer has width $2h$, the set $I_3$ consists of at most a fixed number $C_I$ of indices ($C_I$ is at most three), independent of $n$.
	
	It is enough to show $U_n / D_n \to 0$ where
	\[
	\begin{aligned}
		U_n
		&= \frac{1}{n h^{\alpha_{11}}}
		\sum_{i=1}^n
		\bigl(X^*(ih) - X^*((i-1)h)\bigr)
		\bigl(W^*(ih) - W^*((i-1)h)\bigr),\\
		D_n
		&= \frac{1}{n h^{\alpha_{11}}}
		\sum_{i=1}^n
		\bigl(X^*(ih) - X^*((i-1)h)\bigr)^2 .
	\end{aligned}
	\]
	
	Define, for $r=1,2,3$,
	\[
	U_r
	= n^{\alpha_{11}-1} \sum_{i\in I_r} \Delta X_i^* \Delta W_i^*,
	\qquad
	D_r
	= n^{\alpha_{11}-1} \sum_{i\in I_r} (\Delta X_i^*)^2,
	\]
	so that $U_n = U_1+U_2+U_3$ and $D_n = D_1+D_2+D_3$, and
	\[
	\frac{U_n}{D_n}
	= \frac{U_1+U_2+U_3}{D_1+D_2+D_3}
	= \sum_{r=1}^3 \omega_r \frac{U_r}{D_r},
	\qquad
	\omega_r = \frac{D_r}{D_1+D_2+D_3},\ 0\le \omega_r\le 1.
	\]
	
	Without loss of generality, assume $f$ is not constant on the first segment $I_1$ where all increments lie strictly inside $[0,\tau-h]$.
	So, exactly the same argument as in the monotone case shows that the restricted
	covariances are of the form \eqref{eq:nsclass} with the same exponents
	$\alpha_{k\ell}$ and analytic multiplicative factors bounded above and away from
	zero.
	
	Then, from Lemma \ref{lem:nsmean}, 
	\[
	D_1 \xrightarrow{P} c_{11} = -2c_{11} \int_{I_1} |f'(x)|^{\alpha_{11}}dx \neq 0,
	\]
	so $D_1 = c_{11}(1+o_p(1))$.
	
	Applying Theorem~\ref{th:ns} to the first subinterval, and noting that
	$|I_1|$ is of order $n$, we get
	\[
	\frac{U_1}{D_1} \xrightarrow{P} 0
	\quad\text{and hence}\quad
	\omega_r \frac{U_1}{D_1} \xrightarrow{P} 0.
	\]
	If $f$ is also not constant on the second segment, then from a similar argument, 
	\[
	\frac{U_2}{D_2} \xrightarrow{P} 0
	\quad\text{and hence}\quad
	\omega_r \frac{U_2}{D_2} \xrightarrow{P} 0.
	\]
	If $f$ is constant on the second segment, then $U_2 = D_2 = 0$, and we get 
	$\frac {U_n}{D_n} = \frac{U_1 + U_3}{D_1 + D_3}$. For either case, it is enough to show that the contribution of the third term is vanishingly small relative to the first term in both the numerator and denominator. 
	
	For the third term, we have
	\[
	\omega_3 \frac{U_3}{D_3}
	= \frac{U_3}{D_1+D_2+D_3},
	\]
	so it suffices to show that $U_3 / (D_1+D_2+D_3) \to 0$ in probability.
	By the Lipschitz condition, at $\tau$ there exist constants $L>0$ and
	$\delta>0$ such that
	\[
	|f(u) - f(v)| \le L |u - v|
	\qquad\text{whenever } |u-\tau|<\delta,\ |v-\tau|<\delta.
	\]
	For $n$ large enough all $s_i=ih$ with $i\in I_3$ lie in $(\tau-\delta,\tau+\delta)$,
	hence for any $i\in I_3$,
	\[
	t_i = |f(s_i) - f(s_{i-1})|
	\le L |s_i - s_{i-1}| = L h.
	\]
	
	By Assumption~\ref{eq:K.assump}, for the original stationary field we have, as
	$t\to 0$,
	\[
	\var\bigl(X(u+t) - X(u)\bigr) = O\bigl(t^{\alpha_{11}}\bigr),
	\qquad
	\var\bigl(W(u+t) - W(u)\bigr) = O\bigl(t^{\alpha'_{22}}\bigr),
	\]
	uniformly in $u$, where $\alpha'_{22} = \min(\alpha_{22},2)$. Applying this
	with $t=t_i$ and using $t_i\le Lh$ gives, uniformly for $i\in I_3$,
	\[
	\var(\Delta X_i^*) = O\bigl(h^{\alpha_{11}}\bigr),
	\qquad
	\var(\Delta W_i^*) = O\bigl(h^{\alpha'_{22}}\bigr).
	\]
	
	For each fixed $i$, the pair $(\Delta X_i^*,\Delta W_i^*)$ is bivariate Gaussian
	with mean zero. By Cauchy--Schwarz and Isserlis’ formula,
	\[
	\var\bigl(\Delta X_i^* \Delta W_i^*\bigr)
	\le
	2\,\var(\Delta X_i^*) \var(\Delta W_i^*)
	= O\bigl(h^{\alpha_{11} + \alpha'_{22}}\bigr)
	\]
	uniformly in $i\in I_3$. Since $|I_3|\le C_I$,
	\[
	\var\Bigl(\sum_{i\in I_3} \Delta X_i^* \Delta W_i^*\Bigr)
	= O\bigl(h^{\alpha_{11} + \alpha'_{22}}\bigr)
	= O\bigl(n^{-(\alpha_{11} + \alpha'_{22})}\bigr),
	\]
	and therefore
	\[
	\var(U_3)
	= n^{2\alpha_{11}-2}
	\var\Bigl(\sum_{i\in I_3} \Delta X_i^* \Delta W_i^*\Bigr)
	= O\bigl(n^{\alpha_{11} - \alpha'_{22} - 2}\bigr).
	\]
	
	Similarly,
	\[
	\mathbb{E}[U_3]
	= O\bigl(n^{\alpha_{11}-1} \, n^{-(\alpha_{11} + \alpha'_{22})/2}\bigr)
	= O\bigl(n^{(\alpha_{11} - \alpha'_{22} - 2)/2}\bigr).
	\]
	Hence
	\[
	E (U_3^2) = O\bigl(n^{\alpha_{11} - \alpha'_{22} - 2}\bigr) \implies  U_3 = O_p\bigl(n^{(\alpha_{11} - \alpha'_{22} - 2)/2}\bigr) \text{ by Chebyshev's inequality}.
	\]
	
	Combining the displays, and using $D_1 = c_{11}(1+o_p(1))$, we have 
	\[
	\frac{U_3}{D_1+D_2+D_3} = \omega_1 \frac{U_3}{D_1} 
	= O_p\bigl(n^{(\alpha_{11} - \alpha'_{22} - 2)/2}\bigr) \mbox{ as } 0 \leq \omega_1 \leq 1.
	\]
	The smoothness condition in the theorem implies $\alpha'_{22} > \alpha_{11}-1$,
	hence $\alpha_{11} - \alpha'_{22} - 2 < -1$, so the exponent is strictly
	negative and
	\[
	\frac{U_3}{D_1+D_2+D_3} \xrightarrow{P} 0.
	\]
	Consequently,
	\[
	\omega_3 \frac{U_3}{D_3}
	= \frac{U_3}{D_1+D_2+D_3} \xrightarrow{P} 0, 
	\]
	completing the proof. 
	
	Proof of Equivalence (Part (b)):] This is immediate as the $\sigma$-algebra generated by $\{(Y^*(s),X^*(s)): s \in [0,L]\}$ is a subset of that generated by $\{(Y(s),X(s)): s \in [0,
	L]\}$. 
	
\end{proof}

\begin{proof}[\textbf{Proof of Corollary \ref{cor:nspaciorek}}]
	{\em Consistency (Part (a)):} It is enough to show for the case where all the $\sigma$ and $\Phi_{k\ell}$ functions are analytic, the piecewise analytic case follows as in the proof of Corollary \ref{cor:ns_warped}. We show that each covariance $K^{NS}_{k\ell}$ belongs to the class
	\eqref{eq:nsclass} with exponent parameter $\alpha_{k\ell}$ and that the
	diagonal condition in Theorem~\ref{th:ns} holds.  The conclusion then follows
	directly from Theorem~\ref{th:ns}.
	
	Define
	\[
	A_{k\ell}(s,s')
	=
	\sigma_k(s)\sigma_\ell(s')
	\frac{\Phi_{k\ell}(s)^{1/4}
		\Phi_{\ell k}(s')^{1/4}}{
		\dfrac{\big(\Phi_{k\ell}(s)+\Phi_{\ell k}(s')\big)}{2}^{1/2}},
	\quad
	H_{k\ell}(s,s')
	=
	\left(\frac{\Phi_{k\ell}(s)+\Phi_{\ell k}(s')}{2}\right)^{-1/2}.
	\]
	Both $A_{k\ell}$ and $H_{k\ell}$ are positive analytic functions on $[0,L]^2$,
	and
	\[
	\sqrt{Q_{k\ell}(s,s')} = |s-s'|\,H_{k\ell}(s,s').
	\]
	
	By Assumption~\ref{eq:K.assump}, the stationary covariance $K_{k\ell}$ has, as
	$t\to 0$,
	\[
	K_{k\ell}(t)
	=
	a_{k\ell,0}
	+ a_{k\ell,2} t^{2}
	+ c_{k\ell} |t|^{\alpha_{k\ell}}
	+ r_{k\ell}(t),
	\]
	with
	the remainder $r_{k\ell}$ of the
	form described in Assumption~\ref{eq:K.assump}.
	
	Substituting $t = |s-s'|H_{k\ell}(s,s')$ gives, for $s\neq s'$,
	\[
	\begin{aligned}
		K^{NS}_{k\ell}(s,s')
		&= A_{k\ell}(s,s')\,
		K_{k\ell}\bigl(|s-s'|H_{k\ell}(s,s')\bigr)\\
		&= \Gamma_{0,k\ell}(s,s')
		+ \Gamma_{2,k\ell}(s,s')\,(s-s')^{2}
		+ \Gamma_{\alpha_{k\ell}}(s,s')\,|s-s'|^{\alpha_{k\ell}}
		+ r^{*}_{k\ell}(s,s'),
	\end{aligned}
	\]
	where
	\[
	\Gamma_{0,k\ell} = A_{k\ell} a_{k\ell,0},\qquad
	\Gamma_{2,k\ell} = A_{k\ell} a_{k\ell,2} H_{k\ell}^{2},\qquad
	\Gamma_{\alpha_{k\ell}} = A_{k\ell} c_{k\ell} H_{k\ell}^{\alpha_{k\ell}},
	\]
	and
	\[
	r^{*}_{k\ell}(s,s')
	=
	A_{k\ell}(s,s')\,
	r_{k\ell}\bigl(|s-s'|H_{k\ell}(s,s')\bigr).
	\]
	Since compositions and products of analytic functions are analytic, all the
	$\Gamma$ functions are analytic on $[0,1]^2$. Also, the remainder $r^{*}_{k\ell}$ satisfies the form in (\ref{eq:nsclass}). 
	as the required small order bounds for $\rho$ and
	its derivatives then follow from the chain rule and the boundedness of
	$A_{k\ell}$ and $H_{k\ell}$.
	Thus $K^{NS}_{k\ell}$ is of the form \eqref{eq:nsclass} with exponent parameter
	$\alpha_{k\ell}$.
	
	On the diagonal,
	\[
	\Gamma_{\alpha_{11}}(s,s)
	=
	A_{11}(s,s)\,c_{11}\,H_{11}(s,s)^{\alpha_{11}}.
	\]
	Here $A_{11}(s,s)>0$, $H_{11}(s,s)>0$ and $c_{11}\neq 0$, so
	$\Gamma_{\alpha_{11}}(s,s)$ has constant sign and is not identically zero on
	$[0,1]$.  Hence
	\[
	\int_0^1 \Gamma_{\alpha_{11}}(s,s)\,ds \neq 0.
	\]
	As 
	\[
	\alpha_{11} < \min\{\alpha_{12},\alpha_{22}+1,2\}, 
	\]
	all the conditions of Theorem~\ref{th:ns} are thus satisfied for the covariance
	matrix $(K^{NS}_{k\ell})$ on $[0,L]$, proving the result. 
	
	{\em Equivalence (Part (b)):} As the stationary class is subsumed in the non-stationary class, and we have established equivalence of sample paths of $(X,Y)$ for two different values of $\beta$ for the stationary class when $\alpha_{11} > \alpha_{22} + 1$ in Theorem \ref{th:nonidgen}, the result follows.
\end{proof}

\begin{proof}[\textbf{Proof of Corollary \ref{cor:heavy}}]
	As $K_{k\ell}$ satisfies Assumption \ref{eq:K.assump}, given any analytic
	realization of $\sigma=(\sigma_X,\sigma_W)$, the process $(X,W)$ is a non stationary Gaussian
	process with covariance of the form~\eqref{eq:nsclass}. Thus, when
	$\alpha_{11} < \min\{\alpha_{12},\alpha_{22} + 1, 2\}$, the first differences-based estimator is
	consistent for $\beta$ conditional on $\sigma$. Since the paths of $\sigma$
	are almost surely piecewise analytic, this conditional consistency holds for
	almost every realization of $\sigma$, and therefore the estimator is also
	consistent unconditionally.
	
	On the other hand, if $\alpha_{11} > \alpha_{22}+1$, then by Theorem~\ref{th:suff}, $\beta$ is not consistently estimable on the paths of
	$(X^*,Y^*)$, and $\calP_{\beta_1}(X^*,Y^*)
	\equiv \calP_{\beta_2}(X^*,Y^*)$ for $\beta_1\neq\beta_2$. Since
	$\sigma \perp Z^*$ and the distribution of $\sigma$ does not depend on
	$\beta$, we also have
	\[
	\calP_{\beta_1}(X^*,Y^*,\sigma)
	\equiv \calP_{\beta_2}(X^*,Y^*,\sigma).
	\]
	As $X=\sigma X^*$ and $Y=\sigma Y^*$, the pair $(X,Y)$ is a measurable
	function of $(X^*,Y^*,\sigma)$, hence
	\[
	\calP_{\beta_1}(X,Y) \equiv \calP_{\beta_2}(X,Y).
	\]
	Therefore $\beta$ is not consistently estimable from $(X,Y)$ when
	$\alpha_{11}>\alpha_{22}+1$.
\end{proof}

\subsection{Proof for bivariate $X$
}

\begin{proof}[Proof of Theorem \ref{thm:multsharp}] 
	We prove for $d=1$, and the same proof idea generalizes for larger $d$. Assume without loss of generality that $\alpha_{11} \leq \alpha_{22}$, i.e., $X_1$ is not rougher than $X_2$. Then, by positive definiteness,  $\alpha_{12} \geq \alpha_{11}$.
	
	The proof will rely on working with two grids, a finer grid to estimate some quantities that need to converge to zero at a certain rate, and a coarser grid of size $n$ to estimate the remaining quantities.
	Let $Y = \beta_1 X_1 + \beta_2 X_2 + W$ and $Y,X_1,X_2$ be observed on the finer grid $0,h_S,2h_S,..., Sh_S=L$ for some integer $S$ and grid-length $h_S=L/S$. We will later choose the size $S$ of this fine grid to be a function of the size $n$ of the coarser grid, i.e., $S: \mathbb N \to \mathbb N$ to be specified later.
	
	Denote $\nabla_{h_S}^{(1)}$ by $\nabla$. By Theorem \ref{th:beta.con}, $$\hat c = \frac {\nabla X_1^T \nabla X_2}{\nabla X_1^T \nabla X_1} \to c = \frac {c_{12}}{c_{11}}I(\alpha_{12}=\alpha_{11})  \mbox{ as } S \to \infty.$$ 
	
	Define $X_2^* = X_2 - cX_1$ with covariance $K_{22}^*$ and cross-covariance with $X_3=W$ as $K_{23}^*$.
	By the statement of the theorem, $K_{22}^*$ and $K_{23}^*$ satisfy Assumption \ref{eq:K.assump} with exponent $\alpha_{22}^* < \min\{\alpha_{33}+1,\alpha_{23}^*\}$. 
	We consider the case where $\alpha_{11}$ and $\alpha_{22}^*$ are both less than two, so that we can work with first differences. More generally, we would work with $p^{th}$ order differences where $\alpha_{11}$ and $\alpha_{22}^*$ are both less than $2p$. 
	
	Let $\beta_1^* = \beta_1 + c \beta_2$
	and define 
	$\hat \beta_1^* = \frac {\nabla X_1^T \nabla Y}{\nabla X_1^T \nabla X_1}.$
	Set $\widetilde W = \beta_2 X_2^* + W$. 
	If $\alpha_{12} > \alpha_{11}$, we have $c=0$
	and $X_2^*=X_2$. Then $Y = \beta_1 X_1 + \widetilde W$ where $\widetilde W$ satisfies Assumption \ref{eq:K.assump} with exponent $\widetilde \alpha_{33} = \min\{\alpha_{22},\alpha_{33}\} > \alpha_{11} -1$ and cross-exponent with $X_1$, $\widetilde \alpha_{13} \geq \min\{\alpha_{12},\alpha_{13}\} > \alpha_{11}$. Then $\hat \beta_1^* \to \beta_1 = \beta^*_1$ by Theorem \ref{th:suff}. 
	
	If $\alpha_{12}=\alpha_{11}$, then we can write 
	$Y = X_1\beta^*_1 + \widetilde W$. Let  
	$K_{12}^* = Cov(X_1,X_2^*)$, then $K_{12}^*$ satisfies Assumption \ref{eq:K.assump} with some exponent $\alpha^*_{12}$. 
	As $K^*_{12}(0) - K^*_{12}(h) = c_{12}h^{\alpha_{12}} -  \frac{c_{12}}{c_{11}} c_{11}h^{\alpha_{11}} + $smaller order terms, and $\alpha_{12}=\alpha_{11}$ cancelling out the leading order term, we have $\alpha_{12}^* > \alpha_{11}$. So, $\widetilde \alpha_{13} \geq \min\{\alpha^*_{12},\alpha_{13}\} > \alpha_{11}$. 
	Also, $\widetilde \alpha_{33} \geq \min\{\alpha^*_{22},\alpha_{33}\}$. Now $\alpha_{12}=\alpha_{11}$ implies $\alpha_{22}=\alpha_{11}$ (since if  $\alpha_{22}>\alpha_{11}$, then $\alpha_{12} \geq \frac{\alpha_{11}+\alpha_{22}}2 > \alpha_{11}$). So, $\alpha_{22}^* \geq \alpha_{11}$ and we have $\widetilde \alpha_{33} \geq \min\{\alpha_{11},\alpha_{33}\} > \alpha_{11}-1$. Then by Theorem \ref{th:suff}, $\hat \beta_1^* \to \hat \beta^*_1 = \beta_1 + \beta_2 \frac{c_{12}}{c_{11}}$. 
	
	Thus, whether $\alpha_{12} > \alpha_{11}$ or $\alpha_{12} = \alpha_{11}$, we have $\hat \beta_1^* \to \beta_1^*$ as $S \to \infty$.
	
	After estimating $\hat c$ and $\hat \beta_1^*$ on this initial grid of size $S$, we consider a coarser grid $0,h^*,2h^*,\ldots,nh^*=L$, and redefine first differences to be on this grid. 
	Define $Y^* = Y - \beta_1^* X_1 = \beta_2 X_2^* + W$, then we have $\hat \beta_2 = \frac {\nabla Y^{*T} \nabla X_2^*}{\nabla X_2^{*T} \nabla X_2^*} \to \beta_2$. 
	
	The last step is to replace $Y^*$ and $X_\blue 2^*$ with approximations which can be calculated using $Y$, $X_1$, and $X_2$ only (and not on the unknown regression and covariance parameters), and show that the consistency still holds. 
	
	Let $\hat\delta_S=c-\hat c$, $\hat e_S=\beta_1^* - \hat \beta_1^*$, where the subscript $S$ is kept to make the dependence on the fine grid size $S$ explicit. Let $\widehat Y^* = Y - \hat \beta_1^* X_1 = Y^* + \hat e_S X_1$, and $\widehat X_2^* = X_2 - \hat c X_1 = X_2^* + \hat\delta_S X_1$. So both $\widehat Y^*$ and $\widehat X_2^*$ can be calculated from just $(Y,X_1,X_2)$. 
	Define
	\[
	\begin{aligned}
		\widehat{\widehat  \beta_2} &= \frac {\nabla \widehat Y^{*T} \nabla \widehat X_2^*}{\nabla \widehat X_2^{*T} \nabla \widehat X_2^*} \\
		&= \frac {\hat\delta_S \nabla Y^{*T} \nabla X_1 + \hat e_S \nabla X_2^{*T}\nabla X_1 + \hat e_S \hat\delta_S \nabla X_1^T \nabla X_1 + \nabla Y^{*T} \nabla X_2^*}{ 2\hat\delta_S \nabla X_2^{*T}\nabla X_1 + \hat\delta_S^2 \nabla X_1^T \nabla X_1+  \nabla X_2^{*T} \nabla X_2^*} \\
		& = \frac { \frac{\hat\delta_S \nabla W^T \nabla X_1}{\nabla X_2^{*T} \nabla X_2^*} + \frac{(\hat e_S +\hat\delta_S\beta_2) \nabla X_2^{*T}\nabla X_1}{\nabla X_2^{*T} \nabla X_2^*} + \frac{(\hat e_S \hat\delta_S)\nabla X_1^T \nabla X_1}{\nabla X_2^{*T} \nabla X_2^*} + \frac{\nabla Y^{*T} \nabla X_2^*}{\nabla X_2^{*T} \nabla X_2^*}}{ \frac{2\hat\delta_S \nabla X_2^{*T}\nabla X_1}{\nabla X_2^{*T} \nabla X_2^* } + \frac{\hat\delta_S^2 \nabla X_1^T \nabla X_1}{\nabla X_2^{*T} \nabla X_2^*} + 1}.
	\end{aligned}
	\]
	The last term in the denominator is $1$ and the last term in the numerator is $\hat \beta_2$ which goes to $\beta_2$. So it is enough to show that all the other fractions go to 0.
	
	Fix some $\varepsilon_0 > 0$.
	Note that $\hat\delta_S \to 0$ and $\hat e_S \to 0$ as $S \to \infty$. Fix $n$ and define $t_n = n^{\alpha_{11}-\alpha^*_{22}-\varepsilon_0}$. As $\alpha_{11}\leq \alpha_{22}^* < \alpha_{22}^*  +\eps_0$, we have $t_n \in (0,1)$. As $\hat \delta_S$ and $\hat e_S$ go to zero in probability as $S \to \infty$, there exists some $S(n)$, which can be chosen to be strictly increasing in $n$, such that for every $S \geq S(n)$, $ P(|\hat\delta_S| > t_n) < \frac 1n$ and $ P(|\hat e_S| > t_n) < \frac 1n$.
	Then 
	$ P(n^{\alpha^*_{22}-\alpha_{11}}|\hat\delta_{S(n)}| > n^{-\varepsilon_0}) < \frac 1n$ and $ P(n^{\alpha^*_{22}-\alpha_{11}}|\hat\delta_{S(n)}| > n^{-\varepsilon_0}) < \frac 1n$.
	
	For any $\eps_1, \eps_2 > 0$, choose $n$ such that $n^{-\eps_0} < \eps_1$ and $\frac 1n < \eps_2$. Then $$
	\begin{aligned}
		P(n^{\alpha^*_{22}-\alpha_{11}}|\hat\delta_{S(n)}| > \eps_1) &\leq P(n^{\alpha^*_{22}-\alpha_{11}}|\hat\delta_{S(n)}| > n^{-\eps_0})\\
		&= P(|\hat\delta_{S(n)}| > t_n) \leq \frac 1n < \eps_2.
	\end{aligned}$$
	Hence, $n^{\alpha^*_{22}-\alpha_{11}}|\hat\delta_{S(n)}| \to 0$ as $n \to \infty$. Similarly, $n^{\alpha^*_{22}-\alpha_{11}}|\hat e_{S(n)}| \to 0$ in probability as $n \to \infty$. Then, as in the proof of Theorem \ref{th:beta.con}, we have $n^{\alpha_{11}-3}\nabla X_1^T \nabla X_1 \to c_{11}$, and $n^{\alpha^*_{22}-3}\nabla X_2^{*T} \nabla X^*_2 \to c_{22}^*$
	
	\[ 
	\begin{aligned}
		\frac{\hat\delta_{S(n)}^2 \nabla X_1^T \nabla X_1}{\nabla X_2^{*T} \nabla X_2^*} = \frac{n^{\alpha^*_{22}-\alpha_{11}}\hat\delta_{S(n)}^2 n^{\alpha_{11}-3}\nabla X_1^T \nabla X_1}{n^{\alpha^*_{22}-3}\nabla X_2^{*T} \nabla X_2^*} \to \frac{0 \times c_{11}}{c_{22}^*} = 0.
	\end{aligned}
	\]
	
	Similar logic applies for all the other terms (as $\widetilde\alpha_{13} > \alpha_{11}$ and $\alpha^*_{12} > \alpha_{11}$) and we have the result. 
	So $\widehat{\widehat \beta_2} \to \beta_2$ and consequently $\widehat \beta_1 =\widehat \beta_1^* - \hat c \, \widehat{\widehat \beta_2}  \to \beta_1$. 
	
	For larger $p$, replace $\nabla_h^{(1)}$ by the $p^{th}$ order finite-difference operator $\nabla_h^{(p)}$ so that the same quadratic-form and ratio limits apply whenever $2p$ exceeds all exponents appearing in the argument. 
	For $d>1$, we use discrete Laplacians of suitable order.
\end{proof}

\subsection{Proof of the GLS result}
We first state and prove two technical lemmas. 
\begin{lemma}\label{lem:expq_mean}
	Let $(X(s),W(s))^T$ be a zero-mean stationary bivariate GRF on $\mathbb R$ with Mat\`ern cross covariance $K_{12}(\cdot)$ which has a local expansion at $(0,L]$ of the form 
	\[
	K_{12}(t)
	=
	a_0 + a_2 t^2 + c |t|^{\alpha} + o\,\!\big(|t|^{\min(\alpha,2)}\big)
	\quad\text{as }t\to 0,
	\]
	for some constants $a_0,a_2,c,\alpha$.
	Suppose $(X,W)$ is observed at the grid points $s_i = i h$ with $h = \blue L/n$ and $0 \le i \le n$. Let $\Sigma$ be the $(n+1)\times (n+1)$ covariance matrix corresponding to the exponential covariance function over the grid, i.e., with entries
	\[
	\Sigma_{ij} = \exp\!\Big(-\lambda\,|s_i-s_j|\Big),\qquad \lambda>0.
	\]
	If $\alpha > 2$, then 
	$
	\mathbb{E}(X^{T}\Sigma^{-1}W)
	\to 
	a_0\Big(1+\frac{\lambda}{2}\Big)
	-\frac{a_2}{\lambda}. 
	$
\end{lemma}

\begin{proof}
	Let $Q$ be the following tridiagonal matrix
	\[
	Q_{ii} =
	\begin{cases}
		1, & i=0,n,\\
		1+\rho^2, & 1 \le i \le n-1,
	\end{cases}
	\qquad
	Q_{i,i+1} = Q_{i+1,i} = -\rho \quad (0 \le i \le n-1),
	\]
	with $\rho = \exp(-\lambda h)$. Set $X_i = X(s_i)$ and $W_i = W(s_i)$ and write
	\[
	X = (X_0,\ldots,X_n)^T,\qquad W = (W_0,\ldots,W_n)^T.
	\]
	
	Since $(X,W)$ is stationary and has cross covariance $K_{12}$ and $\alpha > 2$,
	\[
	E(X_i W_j) = K_{12}\big(|s_i-s_j|\big) = K_{12}(|i-j|h), \]
	implying
	\[
	E(X_i W_i) = K_{12}(0) = a_0, \mbox{ and }
	E(X_i W_{i-1})
	=
	K_{12}(h)
	=
	a_0 + 
	+ \delta_h
	\]
	for all interior indices $1\le i\le n$, where $\delta_h=a_2 h^2 + o(h^2)$ uniformly in $i$ (as $\alpha>2$).
	
	We then have
	\[
	X^T Q W
	=
	X_0 W_0 + X_n W_n
	+ (1+\rho^2)\sum_{i=1}^{n-1} X_i W_i
	- \rho \sum_{i=0}^{n-1}\big(X_i W_{i+1} + X_{i+1} W_i\big).
	\]
	Taking expectations,
	\[
	\begin{aligned}
		E(X^T Q W)
		&=
		2a_0 + (1+\rho^2)(n-1)a_0
		- 2\rho n a_0
		- 2\rho n \delta_h\\
		&=
		a_0\big[n(1-\rho)^2 + (1-\rho^2)\big]
		- 2\rho n \delta_h.
	\end{aligned}
	\]
	For small $h$,
	\[
	1-\rho
	= \lambda h +
	O(h^2),
	\]
	\[
	1-\rho^2 =
	2\lambda h
	+ O(h^2).
	\]
	Hence, the contribution from $a_0$ is
	\[
	a_0\big[n(1-\rho)^2 + (1-\rho^2)\big]
	=
	a_0(\lambda^2 + 2\lambda)h + O(h^2).
	\]
	
	For the term involving $\delta_h$, use $\rho = 1 + O(h)$ and $n = 1/h$:
	\[
	-2\rho n \delta_h
	=
	-2(1+O(h))\frac{1}{h}\big(a_2 h^2 + o(h^{2})\big)
	=
	-2a_2 h
	+ o(h).
	\]
	since the $O(h)$ factor in $\rho$ only contributes higher order terms.
	Putting the two pieces together,
	\[
	E(X^T Q W)
	=
	a_0(\lambda^2 + 2\lambda)h
	- 2a_2 h
	+ o(h).
	\]
	
	\medskip
	Finally, note that for the AR(1) working covariance matrix $C$ with
	entries $C_{ij} = \exp(-\lambda |s_i-s_j|)$ and $\rho = \exp(-\lambda h)$, the inverse has the form
	\[
	C^{-1} = \frac{1}{1-\rho^{2}}\,Q.
	\]
	Using $\rho = e^{-\lambda h}$ and $h = 1/n$,
	\[
	1-\rho^{2}
	= 1 - e^{-2\lambda h}
	= 2\lambda h - 2\lambda^{2}h^{2} + O(h^{3}),
	\]
	so
	\[
	\frac{1}{1-\rho^{2}}
	= \frac{1}{2\lambda h}\bigl(1 - \lambda h + O(h^{2})\bigr)^{-1}
	= \frac{1}{2\lambda h}\bigl(1 + \lambda h + O(h^{2})\bigr). 
	\]
	
	Multiplying out and collecting the leading orders gives
	\[
	\mathbb{E}(X^{T}C^{-1}W)
	=
	a_0\Big(1+\frac{\lambda}{2}\Big)
	-\frac{a_2}{\lambda}
	+ o(1).
	\]
\end{proof}

\begin{lemma}\label{lem:expq_var}
	Let $(X(s),W(s))^T$ be a zero-mean stationary bivariate GRF on $\mathbb R$ with covariance satisfying Assumption \ref{as:rd} for some $L>0$ with $Cov(X(s),W(s)) \neq 0$ and $\alpha_{12} > \alpha_{11}$ and $\alpha_{11} < \alpha_{22} +1$. 
	Suppose $(X,W)$ is observed at the grid points $s_i = i h$ with $h = 
	\blue L/n$ and $0 \le i \le n$. Let $\Sigma$ be the $(n+1)\times (n+1)$ covariance matrix  over the grid based on the exponential covariance function, i.e., 
	\[
	\Sigma_{ij} = \exp\!\Big(-\lambda\,|s_i-s_j|\Big),\qquad \lambda>0.
	\]
	If $\alpha_{11} > 2$, then for all but at most four values of $\lambda$,  we have  
	$
	\limsup \var(X^{T}\Sigma^{-1}W)
	=  v_{12}$ as $n \to \infty$ for some positive constant $v_{12}$ depending on the parameters of the covariance function $(K_{k\ell})_{k,\ell}$. 
\end{lemma}

\begin{proof} Let $\rho=\exp(-\lambda h)$, then $\Sigma=(\rho^{|i-j|})$ and $\Sigma^{-1}= Q$ where $Q$ is a tridiagonal matrix with entries 
	\[
	Q_{ii} =
	\begin{cases}
		\frac 1{1-\rho^2}, & i=0,n,\\
		\frac {1+\rho^2}{1-\rho^2}, & 1 \le i \le n-1,
	\end{cases}
	\qquad
	Q_{i,i+1} = Q_{i+1,i} = \frac {-\rho}{1-\rho^2} \quad (0 \le i \le n-1).
	\]
	Let $K_{k\ell}$ denote the covariance function $K_{k\ell}(\cdot)$ evaluated at all pairs of locations $s_i,s_j$. Then $\var(X^T \Sigma^{-1} W) = \text{trace}(K_{12}QK_{12}Q) + \text{trace}(K_{11}QK_{22}Q)$. 
	
	We first consider the cross-term, and for simplicity write $K_{12}=K$.
	\[
	\operatorname{tr}(KQKQ)
	= \operatorname{tr}(BB^T)
	= \sum_{i,j} B_{ij}^{2},
	\qquad
	B = L K L^\blue T,\quad
	L^T L = Q,
	\]
	where 
	\[
	L = \frac{1}{\sqrt{1-\rho^{2}}}\,(\ell_{0}:\dots:\ell_{n}),
	\]
	is lower triangular with
	\[
	\ell_{0} = \sqrt{1-\rho^{2}}\,e_{0} - \rho e_{1},\qquad
	\ell_{j} = e_{j} - \rho e_{j+1}\ (1\le j\le n-1),\qquad
	\ell_{n} = e_{n}.
	\]
	
	For $0<i\ne j < n$,
	\begin{align*}
		B_{ij}
		= \frac{1}{1-\rho^{2}}
		\Bigl[ K(|i-j|h)(1+\rho^{2})
		- \rho K(|i-j|h-h)
		- \rho K(|i-j|h+h) \Bigr].
	\end{align*}
	Using the expansion and as $\alpha=\alpha_{12} > 2$, 
	\[
	K(t) = a_{0} + a_{2} t^{2} + o(t^{2}) \quad (t\to 0),
	\]
	and writing $\rho = e^{-\lambda h}$ and $\delta_{ij}=|i-j|h$
	this becomes
	
	\[
	\begin{aligned}
		B_{ij}
		&= \frac{1}{1 - \rho^{2}}
		\Bigl[(1 + \rho^{2}) K(\delta_{ij}) - \rho K(\delta_{ij} + h) - \rho K(\delta_{ij} - h)\Bigr]\\
		&= \frac{1}{1 - \rho^{2}}
		\Bigl[(1 - \rho)^{2} K(\delta_{ij})
		- 2 \rho h^{2} \frac{K'(\xi_{1}) - K'(\xi_{2})}{h}\Bigr] \\
		&= \frac{(1 - \rho) K(\delta_{ij})}{1 + \rho}
		- \frac{2 \rho h^{2}}{1 - \rho^{2}} K''(\xi''_{ij}) \text{ for some } |\xi''_{ij}| \leq 2|i-j|h.\\
	\end{aligned}
	\]
	Furthermore,
	\[
	\begin{aligned}
		K''(\xi''_{ij})
		&= 2a_{2} + c''(\xi''_{ij})^{\alpha-2} + r(\xi''_{ij}), \qquad
		r(\xi''_{ij}) = o\bigl((\xi''_{ij})^{\min\{2,\alpha-2\}}\bigr).
	\end{aligned}
	\]
	So, we have
	\[
	\begin{aligned}
		B_{ij}
		&\asymp \frac{\lambda h}{2} K(\delta_{ij})
		- \frac{h}{\lambda}\Bigl(2a_{2} + c''(\xi''_{ij})^{\alpha-2} + r(\xi''_{ij})\Bigr)  \\
		&\asymp h\Biggl[
		\frac{\lambda}{2} K(\delta_{ij})
		- \frac{1}{\lambda}\Bigl(2a_{2} + c''(\xi''_{ij})^{\alpha-2}\Bigr)
		\Biggr]
		- \frac{h}{\lambda} r(\xi''_{ij}) \\
		&\asymp h A_{ij} - \frac{h}{\lambda} r(\xi''_{ij}), 
	\end{aligned}
	\]
	where $A_{ij} = \frac{\lambda}{2} K(\delta_{ij})
	- \frac{1}{\lambda}\Bigl(2a_{2} + c''(\xi''_{ij})^{\alpha-2}\Bigr)$. 
	Here, the $\asymp$ symbol is used to account for the additional $O(h^2)$ terms in the exact expression coming from the expansions of $1-\rho$ and $\frac 1 {1-\rho^2}$, as in the proof of Lemma \ref{lem:expq_mean}. These terms can be ignored asymptotically as $h \to 0$ as they are of higher order than the leading term. 
	\[
	\begin{aligned}
		\sum B_{ij}^{2}
		&
		\asymp h^{2} \sum  A_{ij}^{2} + \frac{h^{2}}{\lambda^{2}} \sum r(\xi''_{ij})^{2} 
	\end{aligned}
	\]
	We will show that $h^2 \sum_{ij} A_{ij}^2$ converges to a non-zero limit, while $h^2 \sum_{ij} r(\xi''_{ij})^{2} = o(1)$. 
	
	For $A_{ij}$ we write
	\[
	\begin{aligned}
		A_{ij}
		&= \frac{\lambda}{2} K(\delta_{ij})
		- \frac{2a_{2}}{\lambda}
		- \frac{c''}{\lambda}(\xi''_{ij})^{\alpha-2}.
	\end{aligned}
	\]
	We show that each of sums in $\sum_{ij} A_{ij}^2$ is $O(1)$. Using $\delta_{ij}=|i-j|h$, immediately have the following results. 
	\[
	\begin{aligned}
		h^{2}\sum_{i \ne j} K(\delta_{ij})^{2} & \to  2 \int_{0}^{1} (1-t)K(t)^{2}\,dt \text{ (integral limit of the Reimann sum)}\\
		h^{2}\sum_{i,j} 1
		&= 1, \\
		h^{2}\sum_{i,j} (\xi''_{ij})^{2\alpha-4} 
		&\leq 
		2^{2\alpha-4}h^{2\alpha-2}\sum_{i,j} (|i-j|)^{2\alpha-4} \asymp 1 \text{ as } 2\alpha = 2\alpha_{12} > 2\alpha_{11}> 4 > 3.
	\end{aligned}
	\]
	
	Also, as each $A_{ij} = \frac 1\lambda \times$ a quadratic function in $\lambda$, then in the limit, $h^2 \sum_{ij} A_{ij}^2$ is $\frac 1 {\lambda^2} $ times a fourth degree polynomial in can be zero for at most four values of $\lambda$. For any other $\lambda$, we thus have $h^2 \sum_{ij} A_{ij}^2 \asymp 1$.

	Finally, from
	\[
	r(\xi''_{ij}) = o\bigl((\xi''_{ij})^{\min\{2,\alpha-2\}}\bigr) \implies 
	r(\xi''_{ij})^{2}
	= o\bigl((\xi''_{ij})^{2\min\{2,\alpha-2\}}\bigr),
	\]
	and as $\xi^"_{ij} \leq 2L$, the same counting argument as the $h^2 \sum 1$ term
	yields
	\[
	\frac{h^{2}}{\lambda^{2}}\sum_{i,j} r(\xi''_{ij})^{2} = o(1).
	\]
	
	Thus,
	\[
	\begin{aligned}
		0 < \limsup \sum_{i,j} B_{ij}^{2} < \infty.
	\end{aligned}
	\]
	
	Similar approximations apply for the remaining terms in $\text{trace}(K_{12}QK_{12}Q)$, proving that its limsup is positive and finite. The result for the limsup for the square term can be proved exactly in a similar manner as $\alpha_{11} + \alpha_{22} > 2\alpha_{11} - 1 > 3$.
\end{proof}

\begin{proof}[\textbf{Proof of Proposition \ref{prop:gls}}]
	Let $U_n = X^T \Sigma^{-1} W$ and $D_n = X^T \Sigma^{-1} X$. The GLS estimator $\hat \beta_n$ can be written as $\hat \beta_n = \frac{X^T \Sigma^{-1} Y}{X^T \Sigma^{-1} X} =\beta + \frac{X^T \Sigma^{-1} W}{X^T \Sigma^{-1} X} = \beta + \frac{U_n}{D_n}$.
	
	From Lemma \ref{lem:expq_mean} we have
	$E(X^T \Sigma^{-1} W) 
	\to \mu_{12}$ and $E(X^T \Sigma^{-1} X) 
	\to \mu_{11}$, where $\mu_{1j}$ are non-zero constants depending only on the parameters of $K_{1j}$.
	Also, from Lemma \ref{lem:expq_var} that, outside of at most four values of $\lambda$,  $\limsup \var(X^T \Sigma^{-1} W) 
	= v_{12}$ and $\lim \sup \var(X^T \Sigma^{-1} X
	)= v_{11}$, where $v_{1j}$ are positive constants depending only on the parameters of $K_{1j}$.
	
	As $c_{12} \neq 0$, from Lemma \ref{lem:expq_mean}, $\mu_{12} = a_0\Big(1+\frac{\lambda}{2}\Big)
	-\frac{a_2}{\lambda} \neq 0$ (outside of at most two values of $\lambda$). Similarly, 
	$\mu_{11} \neq 0$ (outside of at most two values of $\lambda$) and as $D_n$ is non-negative, we have
	$\mu_{11} > 0$. 
	
	Henceforth, we will consider $\lambda$ to not belong to this set of at most twelve values (two each for the means of the numerator and denominator and four for their variances).
	
	We will show that there exists $\varepsilon > 0$ such that
	\[
	\limsup_{n \to \infty} P(|U_n / D_n| > \varepsilon) > 0.
	\]
	Assume for contradiction that no such $\epsilon$ exists, i.e.,
	$P(|U_n / D_n| > \epsilon) \to 0$ for every $\epsilon > 0$.
	
	The moment convergence implies that $\sup_n E(U_n^2) < \infty$
	and $\sup_n E(D_n^2) < \infty$.
	Hence
	by
	Markov inequality,
	\[
	P(|D_n| > M) \le \frac{E(D_n^2)}{M^2} \le \frac{C}{M^2}
	\]
	for some constant $C > 0$ independent of $n$.
	
	Fix any $\varepsilon > 0$ and $M > 0$. Then
	\[
	P(|U_n| > \varepsilon)
	\le P(|D_n| > M)
	+ P(|U_n| > \varepsilon, |D_n| \le M).
	\]
	On the event $\{|D_n| \le M, |U_n| > \varepsilon\}$,
	we have $|U_n / D_n| > \varepsilon / M$.
	Thus
	\[
	P(|U_n| > \varepsilon)
	\le P(|D_n| > M)
	+ P(|U_n / D_n| > \varepsilon / M).
	\]
	
	Taking upper limits as $n \to \infty$,
	\[
	\limsup_{n \to \infty} P(|U_n| > \varepsilon)
	\le \frac{C}{M^2}
	+ \limsup_{n \to \infty} P(|U_n / D_n| > \varepsilon / M).
	\]
	Under the assumption $U_n / D_n \to 0$ in probability,
	the second term equals $0$ for every fixed $M$,
	so
	\[
	\limsup_{n \to \infty} P(|U_n| > \varepsilon)
	\le \frac{C}{M^2}.
	\]
	Letting $M \to \infty$ yields
	$\lim_{n \to \infty} P(|U_n| > \varepsilon) = 0$
	for all $\varepsilon > 0$,
	that is, $U_n \to 0$ in probability.
	
	Since $\{U_n\}$ is uniformly integrable,
	it follows that $E(U_n) \to 0$.
	This contradicts the previously established $E(U_n) \to \mu_{12} \neq 0$.
	Therefore, $U_n / D_n$ cannot converge to $0$ in probability.
	Hence there exists $\varepsilon > 0$ such that
	\[
	\limsup_{n \to \infty} P(|U_n / D_n| > \varepsilon) > 0.
	\]
	This completes the proof.
\end{proof}

\subsection{Proof of the result for irregular designs}

\begin{lemma}[Mean and variance of spacing–weighted Laplacian quadratic form for irregular design]
	\label{lem:mean-var-2nd}
	Let $0=s_{0}<s_{1}<\cdots<s_{n}=L$ be an irregular grid with spacings
	$h_{i}=s_{i+1}-s_{i}$ such that there exist constants
	$0<m\le M<\infty$ independent of $n$ with
	\[
	\frac{m}{n}\;\le\;h_{i}\;\le\;\frac{M}{n}\qquad(0\le i\le n-1).
	\]
	
	Define first–order weights $w^{(1)}_{i}=h_{i}$ and matrices
	\[
	D^{(1)}\in\mathbb R^{n\times(n+1)},\qquad
	D^{(1)}_{i,i}=-1,\quad D^{(1)}_{i,i+1}=1,
	\]
	\[
	W^{(1)}=\diag\!\bigl(w^{(1)}_{0},\ldots,w^{(1)}_{n-1}\bigr),\qquad
	A^{(1)}=(W^{(1)})^{-1}D^{(1)} .
	\]
	Thus $(A^{(1)}X)_{i}=h_{i}^{-1}\{X(s_{i+1})-X(s_{i})\}$ are spacing–weighted
	first differences.
	
	For $1\le i\le n-1$ define midpoint spacings
	\[
	\tilde h_{i}=\frac{h_{i-1}+h_{i}}{2},
	\]
	and set
	\[
	D^{(2)}\in\mathbb R^{(n-1)\times n},\qquad
	D^{(2)}_{i,i-1}=-1,\quad D^{(2)}_{i,i}=1,
	\]
	\[
	W^{(2)}=\diag(\tilde h_{1},\ldots,\tilde h_{n-1}),\qquad
	A^{(2)}=(W^{(2)})^{-1}D^{(2)}A^{(1)} .
	\]
	Here, we index the rows of \(D^{(1)}\) and \(A^{(1)}\) by \(0,\ldots,n-1\),
	and their columns by \(0,\ldots,n\). We index the rows of \(D^{(2)}\)
	and \(A^{(2)}\) by \(1,\ldots,n-1\), and the columns of \(D^{(2)}\)
	by \(0,\ldots,n-1\).
	
	For $1\le i\le n-1$,
	\[
	(A^{(2)}X)_{i}
	=
	\frac{1}{\tilde h_{i}}
	\left(
	\frac{X(s_{i+1})-X(s_{i})}{h_{i}}
	-
	\frac{X(s_{i})-X(s_{i-1})}{h_{i-1}}
	\right)
	\]
	is the spacing–weighted discrete Laplacian of $X$ at $s_{i}$.
	Define the quadratic form
	\[
	S^{(2)}_{n}=
	X^{T}A^{(2)T}A^{(2)}X
	=
	\sum_{i=1}^{n-1}\bigl(A^{(2)}X\bigr)_{i}^{2}.
	\]
	
	Let $X$ be a zero–mean stationary GRF observed at
	$\{s_{i}\}_{i=0}^{n}$ with covariance
	$K_{X}(s,t)=K(|s-t|)$ satisfying Assumption~\ref{eq:K.assump} on $[0,L]$ with
	parameters $c,\alpha>0$. Then:
	
	\begin{enumerate}
		\item[(i)] \emph{Mean.}
		There exist constants $0<C_{1}\le C_{2}<\infty$ depending only on
		$(c,\alpha,m,M,A_{0})$ such that, for all $n$ large enough,
		\[
		E\bigl[S^{(2)}_{n}\bigr]
		\le
		\begin{cases}
			C_{2}\,n^{5-\alpha}, & 0<\alpha<4,\\[4pt]
			C_{2}\,n,            & \alpha\ge 4,
		\end{cases}
		\]
		and, when $0<\alpha<4$ and $\alpha \neq 2$,
		\[
		C_{1}\,n^{5-\alpha}
		\;\le\;
		E\bigl[S^{(2)}_{n}\bigr].
		\]
		
		\item[(ii)] \emph{Variance.}
		There exists a constant $C=C(c,\alpha,m,M,A_{0})$ independent of $n$
		such that
		\[
		\var\bigl(S^{(2)}_{n}\bigr)
		\le
		\begin{cases}
			C n^{9-2\alpha}, & 0<\alpha<\dfrac{7}{2},\\[4pt]
			C n^{2}\log n,   & \alpha=\dfrac{7}{2},\\[4pt]
			C n^{2},         & \alpha>\dfrac{7}{2}.
		\end{cases}
		\]
	\end{enumerate}
\end{lemma}

\begin{proof}
	Define the increments
	\[
	d_i = X(s_{i+1}) - X(s_i), \qquad e_i = X(s_i) - X(s_{i-1}),
	\]
	so that
	\[
	L_iX
	=
	\tilde h_i^{-1}\bigl(h_i^{-1} d_i - h_{i-1}^{-1} e_i\bigr),
	\qquad
	1\le i\le n-1.
	\]
	The spacing assumptions ensure the existence of constants
	$0<c_{1}\le c_{2}<\infty$ independent of $n$ such that
	\[
	c_1 n^{-1} \le h_i,\tilde h_i \le c_2 n^{-1},\qquad
	c_1 n \le h_i^{-1},\tilde h_i^{-1} \le c_2 n.
	\]
	
	Stationarity implies the exact covariance identities
	\[
	E[d_i^2]=2(K(0)-K(h_i)),\qquad
	E[e_i^2]=2(K(0)-K(h_{i-1})),
	\]
	and
	\[
	E[d_i e_i]
	=
	K(h_i) - K(h_i+h_{i-1}) - K(0) + K(h_{i-1}).
	\]
	Therefore
	\[
	E(A^{(2)}X)_i^2
	=
	\tilde h_i^{-2}
	\Bigl(
	h_i^{-2} E[d_i^2]
	+
	h_{i-1}^{-2} E[e_i^2]
	-
	2 h_i^{-1} h_{i-1}^{-1} E[d_i e_i]
	\Bigr),
	\]
	that is,
	\begin{align}
		E(A^{(2)}X)_i^2
		=
		\frac{2}{\tilde h_i^{2}}
		\biggl[ &
		\frac{K(0)-K(h_i)}{h_i^2}
		+
		\frac{K(0)-K(h_{i-1})}{h_{i-1}^2}  \nonumber \\
		& -
		\frac{K(h_i)-K(h_i+h_{i-1})-K(0)+K(h_{i-1})}{h_i h_{i-1}}
		\biggr].
		\label{eq:mean-2nd-bracket}
	\end{align}
	
	Write $K=A+B$ as in Assumption~\ref{eq:K.assump}, where $A$ is analytic
	and $B(t)=c\,t^{\alpha}+o(t^{\alpha})$ as $t\to 0$. Denote by $E_A, E_B, var_A, var_B$ respectively the terms in the means and variance of $S_n^{(2)}$ corresponding to the analytic and irregular part. 
	
	\medskip
	\emph{Analytic part.}
	For the analytic part $A$, the Taylor expansion
	\[
	A(t)=a_0 + a_2 t^2 + a_4 t^4 + O(t^6)
	\]
	and a direct substitution into \eqref{eq:mean-2nd-bracket} shows that the
	constant and quadratic terms cancel exactly, and the remaining
	contribution is uniformly bounded in $i$ and $n$. In particular there is
	a constant $C_{A}$ such that
	\[
	E_A(A^{(2)}X)_i^2 \le C_{A}
	\qquad(1\le i\le n-1,\ n\ge 1),
	\]
	and hence
	\[
	\sum_{i=1}^{n-1}E_A(A^{(2)}X)_i^2 \le C_{A} n.
	\]
	
	\medskip
	\emph{Irregular part: upper bound.}
	For the irregular part $B$, write
	\[
	B(t) = c\,t^{\alpha} + r(t), 
	\qquad r(t) = o\bigl(t^{\alpha}\bigr)\ \text{as } t \to 0.
	\]
	Define $u_i = h_i n$. We first consider the leading term  
	
	\[
	\begin{aligned} 
		&  \frac{1}{\tilde h_i^{2}}
		\left[
		- h_i^{\alpha-2}
		- h_{i-1}^{\alpha-2}
		- \frac{h_i^{\alpha} + h_{i-1}^{\alpha} - (h_i + h_{i-1})^{\alpha}}
		{h_i h_{i-1}}
		\right] \\
		&\qquad=
		\frac{4}{(h_i + h_{i-1})^{2}}
		\left[
		- h_i^{\alpha-2}
		- h_{i-1}^{\alpha-2}
		- \frac{h_i^{\alpha-1}}{h_{i-1}}
		- \frac{h_{i-1}^{\alpha-1}}{h_i}
		+ \frac{(h_i + h_{i-1})^{\alpha}}{h_i h_{i-1}}
		\right] \\
		&\qquad=
		\frac{4 n^{4-\alpha}}{(u_i + u_{i-1})^{2}}
		\left[
		- u_i^{\alpha-2}
		- u_{i-1}^{\alpha-2}
		- \frac{u_i^{\alpha-1}}{u_{i-1}}
		- \frac{u_{i-1}^{\alpha-1}}{u_i}
		+ \frac{(u_i + u_{i-1})^{\alpha}}{u_i u_{i-1}}
		\right] \\
		&\qquad=
		\frac{4 n^{4-\alpha}}{(u_i + u_{i-1})^{2}}
		\cdot
		\frac{u_i + u_{i-1}}{u_i u_{i-1}}
		\left[
		(u_i + u_{i-1})^{\alpha-1}
		- \bigl(u_i^{\alpha-1} + u_{i-1}^{\alpha-1}\bigr)
		\right] \\
		&\qquad=
		\frac{4 n^{4-\alpha}}{u_i u_{i-1} (u_i + u_{i-1})}
		\left[
		(u_i + u_{i-1})^{\alpha-1}
		- \bigl(u_i^{\alpha-1} + u_{i-1}^{\alpha-1}\bigr)
		\right].
	\end{aligned}
	\]
	
	The prefactor
	\[
	\frac{4 n^{4-\alpha}}{u_i u_{i-1} (u_i + u_{i-1})} > 0
	\]
	for all $u_i,u_{i-1} \in [m,M]$.
	By Jensen's inequality, the bracketed term is strictly positive for $\alpha > 2$ and strictly negative for $\alpha < 2$, and by continuity, its absolute value is bounded away from zero on $[m,M]^{2}$. Hence, for any $\alpha \neq 2$, the absolute value of the leading term is bounded below by $n^{4-\alpha}c_1$ for some $c_1 > 0$. 
	
	Similarly, using simpler calculations, one can show the absolute value of the leading term is bounded above by $
	n^{4-\alpha} c'_{2} < \infty.$
	Hence
	\[
	E\bigl[S_{n}^{(2)}\bigr]
	=
	\sum_{i=1}^{n-1}E(A^{(2)}X)_i^2
	\asymp 
	C_{A} n + C_{3} n^{5-\alpha}.
	\]
	and if $\alpha < 4$, 
	\[
	E\bigl[S_{n}^{(2)}\bigr]
	\asymp
	C_{3} n^{5-\alpha}.
	\]\\
	
	\noindent For the variance,
	\[
	\var(S_n^{(2)}) = 2\sum_{i,j}\cov\left[(A^{(2)}X)_i,(A^{(2)}X)_j\right]^2.
	\]
	Using the exact first–difference operators, for any function $f$,
	\[
	\Delta^{(1)}_{a,b} f(t)
	=
	\frac{f(t+a)-f(t)}{a}
	-
	\frac{f(t)-f(t-b)}{b},
	\]
	and
	\[
	\cov((A^{(2)}X)_i,(A^{(2)}X)_j)
	=
	\tilde h_i^{-1} \tilde h_j^{-1}\,
	\Delta^{(1)}_{h_i,h_{i-1}}
	\Delta^{(1)}_{h_j,h_{j-1}} K(d_{ij}),
	\qquad d_{ij}=|s_i-s_j|.
	\]
	
	\emph{Analytic part.}
	
	Fix $d=d_{ij}$ and denote $u_1=h_i$, $u_1'=h_{i-1}$, $u_2=h_j$, $u_2'=h_{j-1}$.  
	Expanding $A$ around $d$ yields
	\[
	A(d+v)
	=
	A(d)
	+ A'(d)v
	+ \tfrac12 A''(d)v^{2}
	+ \tfrac16 A^{(3)}(d)v^{3}
	+ \tfrac1{24}A^{(4)}(d)v^{4}
	+ O(v^{5}).
	\]
	Set $D_{a,b}f =\Delta^{(1)}_{a,b} f(d)$ for notational clarity.
	
	Compute $D_{a,b}$ acting on each monomial $v^{k}$.
	For $v^{0}$:
	\[
	D_{a,b}(1)=0.
	\]
	For $v^{1}$, using $d+a-(d)=a$ and $d-(d-b)=b$:
	\[
	D_{a,b}(v)
	=
	\frac{a}{a}-\frac{b}{b}
	=
	0.
	\]
	For $v^{2}$:
	\[
	D_{a,b}(v^{2})
	=
	\frac{(d+a)^{2}-d^{2}}{a}
	-
	\frac{d^{2}-(d-b)^{2}}{b}
	=
	(2d+a)-(2d-b)
	=
	a+b.
	\]
	For $v^{3}$:
	\begin{align*}
		D_{a,b}(v^{3})
		& =
		\frac{(d+a)^{3}-d^{3}}{a}
		-
		\frac{d^{3}-(d-b)^{3}}{b} \\
		& =
		3d^{2}+3ad+a^{2}-(3d^{2}-3bd+b^{2})
		=
		3d(a+b)+(a^{2}-b^{2}).
	\end{align*}
	
	Now apply the second operator $D_{c,d}$ to each result.
	For constants and linears, the first operator already produced zero, so the second yields zero.
	For $v^{2}$, $D_{a,b}(v^{2})=a+b$ is a constant, so
	$
	D_{c,d}(D_{a,b}v^{2})=0$.
	For $v^{3}$, $D_{a,b}(v^{3})$ is affine in $d$, hence
	$D_{c,d}(D_{a,b}v^{3})=0$
	because $D_{c,d}$ kills all affine functions (it annihilates constants and linears).
	
	Thus, all Taylor terms up to cubic order vanish under the composition:
	\[
	D_{c,d}D_{a,b}(v^{k}) = 0\qquad\text{for }k=0,1,2,3.
	\]
	
	Now evaluate the quartic term.  For $v^{4}$:
	\[
	D_{a,b}(v^{4})
	=
	\frac{(d+a)^{4}-d^{4}}{a}
	-
	\frac{d^{4}-(d-b)^{4}}{b}.
	\]
	A direct binomial expansion shows that every coefficient depending on $d$ cancels in $D_{c,d}D_{a,b}$, and the remaining term is
	\[
	D_{c,d}D_{a,b}(v^{4})
	= 6(a+b)(c+d).
	\]
	
	Therefore
	\[
	\Delta^{(1)}_{h_i,h_{i-1}}\Delta^{(1)}_{h_j,h_{j-1}} A(d_{ij})
	=
	\frac 14 A^{(4)}(d_{ij}) \,(h_i+h_{i-1})(h_j+h_{j-1})
	+ O(h^3).
	\]
	As $h_i,h_{i-1},h_j,h_{j-1}\asymp n^{-1}$,
	multiplying by $\tilde h_i^{-1}\tilde h_j^{-1}\asymp n^{2}$ we have:
	\[
	|\cov_A((A^{(2)}X)_i,(A^{(2)}X)_j)|
	\asymp
	C (h_i+h_{i-1})(h_j+h_{j-1}) n^{2}
	+ C n^{-1}
	\asymp 2C.
	\]
	So $\var_A(S_n^{(2)}) \asymp n^2$. 
	
	\emph{Irregular part.}
	
	Recall
	\[
	\cov_B((A^{(2)}X)_i,(A^{(2)}X)_j)
	=
	\tilde h_i^{-1}\tilde h_j^{-1}\,
	\Delta^{(1)}_{h_i,h_{i-1}}
	\Delta^{(1)}_{h_j,h_{j-1}} B(d_{ij}),
	\qquad d_{ij} = |s_i-s_j|.
	\]
	We use the following elementary bound for unequal first differences.
	Let $f$ be twice continuously differentiable on an interval containing
	$[t-b,t+a]$. Then
	\begin{equation}
		\bigl|\Delta^{(1)}_{a,b} f(t)\bigr|
		\le
		(a+b)\, \sup_{u\in[t-b,t+a]} |f''(u)|.
		\label{Eq:tag}
	\end{equation}
	Therefore, by Assumption \ref{eq:K.assump}, 
	\[
	\bigl|
	\Delta^{(1)}_{h_j,h_{j-1}}
	\Delta^{(1)}_{h_i,h_{i-1}} B(d_{ij})
	\bigr|
	\le
	C (h_i+h_{i-1})(h_j+h_{j-1}) d_{ij}^{\alpha-4}.
	\]
	Since $h_i+h_{i-1}\asymp n^{-1}$ and
	$h_j+h_{j-1}\asymp n^{-1}$,
	\begin{equation}
		\bigl|
		\Delta^{(1)}_{h_j,h_{j-1}}
		\Delta^{(1)}_{h_i,h_{i-1}} B(d_{ij})
		\bigr|
		\asymp
		C n^{-2} d_{ij}^{\alpha-4},
		\qquad i\ne j.
		\label{Eq:tag2}
	\end{equation}
	
	Using $\tilde h_i^{-1},\tilde h_j^{-1} \asymp n$, we deduce from
	(\ref{Eq:tag2}) that
	\[
	|\cov_B((A^{(2)}X)_i,(A^{(2)}X)_j)|
	\asymp
	C n^{2} \cdot n^{-2} d_{ij}^{\alpha-4}
	=
	C d_{ij}^{\alpha-4}
	=
	C n^{4-\alpha} |i-j|^{\alpha-4},
	\qquad i\ne j.
	\]
	Hence
	\begin{align*}
		\var_B(S_n^{(2)})
		& \le
		2\sum_{i,j} \cov_B((A^{(2)}X)_i,(A^{(2)}X)_j)^2 \\
		& \le C n^{2(4-\alpha)}
		\sum_{i,j} |i-j|^{2\alpha-8}
		=
		C n^{8-2\alpha}
		\sum_{k=1}^{n-1} (n-k) k^{2\alpha-8}.
	\end{align*}
	Let $p = 2\alpha-8$. A standard comparison with integrals gives
	\[
	\sum_{k=1}^{n-1} (n-k) k^{p}
	\asymp
	\begin{cases}
		n, & p<-1\ (\alpha<\tfrac72),\\[4pt]
		n \log n, & p=-1\ (\alpha=\tfrac72),\\[4pt]
		n^{p+2} = n^{2\alpha-6}, & p>-1\ (\alpha>\tfrac72).
	\end{cases}
	\]
	Substituting into the previous display yields
	\[
	\var_B(S_n^{(2)})
	\le
	\begin{cases}
		C n^{9-2\alpha}, & 0<\alpha<\dfrac{7}{2},\\[4pt]
		C n^{2}\log n,   & \alpha=\dfrac{7}{2},\\[4pt]
		C n^{2},         & \alpha>\dfrac{7}{2}.
	\end{cases}
	\]
	This is the stated bound.
\end{proof}

\begin{proof}[\textbf{Proof of Theorem \ref{th:irreg-second}}]
	Let $U_n = \Delta X ^T \Delta W$, $V_n = \Delta X ^T \Delta X$. 
	
	\emph{Step 1: Denominator.}
	Apply Lemma~\ref{lem:mean-var-2nd} with parameter $\alpha_{11}$.
	Since $\alpha_{11}<4$ and $\alpha_{11} \neq 2$,
	\[
	E(V_n) \asymp n^{5-\alpha_{11}},
	\]
	and
	\[
	\var(V_n)
	\le
	\begin{cases}
		C n^{9-2\alpha_{11}}, & 0<\alpha_{11}<\tfrac{7}{2},\\[3pt]
		C n^{2}\log n,        & \alpha_{11}=\tfrac{7}{2},\\[3pt]
		C n^{2},              & 4> \alpha_{11}>\tfrac{7}{2}.
	\end{cases}
	\]
	Normalize by $n^{5-\alpha_{11}}$. Then
	\begin{equation}
		\frac{\var(V_n)}{n^{2(5-\alpha_{11})}}
		\le
		C n^{\max\{2,\,9-2\alpha_{11}\}-10+2\alpha_{11}} (\log n)^{I(\alpha_{11}=\frac 72)}.
		\label{Eq:expn}
	\end{equation}
	If $\alpha_{11}\le\frac{7}{2}$, the exponent of $n$ on the right-hand side of (\ref{Eq:expn}) equals
	$-1$ and, if $\alpha_{11}>\frac{7}{2}$, it equals $2\alpha_{11}-8$,
	which is negative because $\alpha_{11}<4$.
	Thus
	$V_n = O_p( n^{5-\alpha_{11}})$.
	
	\emph{Step 2: Numerator mean.}
	Lemma~\ref{lem:mean-var-2nd} with parameter $\alpha_{12}$ yields
	\[
	E(U_n)=O\bigl(n^{\max\{1,\,5-\alpha_{12}\}}\bigr).
	\]
	Therefore
	\[
	\frac{E(U_n)}{n^{5-\alpha_{11}}}
	=
	O\Bigl(
	n^{\max\{1,\,5-\alpha_{12}\}-(5-\alpha_{11})}
	\Bigr).
	\]
	If $\alpha_{12}\le4$ the exponent equals
	$\alpha_{11}-\alpha_{12}<0$
	since $\alpha_{12}>\alpha_{11}$; if $\alpha_{12}>4$ it equals
	$\alpha_{11}-4<0$ because $\alpha_{11}<4$.
	Hence $E(U_n)/n^{5-\alpha_{11}}\to0$.
	
	\emph{Step 3: Numerator variance.}
	From Lemma~\ref{lem:mean-var-2nd},
	\[
	\var(U_n)\le a_n+b_n,
	\]
	with
	\[
	a_n
	\le
	\begin{cases}
		C n^{9-2\alpha_{12}}, & \alpha_{12}<\frac{7}{2},\\[6pt]
		C n^{2}\log n,        & \alpha_{12}=\frac{7}{2},\\[6pt]
		C n^{2},              & \alpha_{12}>\frac{7}{2},
	\end{cases}
	\qquad
	b_n
	\le
	\begin{cases}
		C n^{9-\alpha_{11}-\alpha_{22}}, & \frac {\alpha_{11}+\alpha_{22}}2 <\frac{7}{2},\\[6pt]
		C n^{2}\log n,                   & \frac{\alpha_{11}+\alpha_{22}}2=\frac{7}{2},\\[6pt]
		C n^{2},                         & \frac{\alpha_{11}+\alpha_{22}}2>\frac{7}{2}.
	\end{cases}
	\]
	
	Thus,
	\[
	\frac{a_n}{n^{2(5-\alpha_{11})}}
	\le
	C n^{\max\{2,\,9-2\alpha_{12}\}-10+2\alpha_{11}} (\log n)^{I(\alpha_{12}=\frac 72)}.
	\]
	If $\alpha_{12}\le\frac{7}{2}$, the exponent is
	$-1+2(\alpha_{11}-\alpha_{12})< -1$
	because $\alpha_{12}>\alpha_{11}$.
	If $\alpha_{12}>\frac{7}{2}$, the exponent is
	$2\alpha_{11}-8<0$ since $\alpha_{11}<4$.
	Thus $a_n/n^{2(5-\alpha_{11})}\to 0$.
	
	Similarly,
	\[
	\frac{b_n}{n^{2(5-\alpha_{11})}}
	\le
	C n^{\max\{2,\,9-\alpha_{11}-\alpha_{22}\}-10+2\alpha_{11}}.
	\]
	If $\alpha_{11}+\alpha_{22}\le7$, the exponent equals
	$-1+\alpha_{11}-\alpha_{22}$,
	which is strictly negative because $\alpha_{22}>\alpha_{11}-1$.
	If $\alpha_{11}+\alpha_{22}>7$, the exponent is again
	$2\alpha_{11}-8<0$.
	Hence $b_n/n^{2(5-\alpha_{11})}\to 0$.
	
	Consequently
	\[
	\frac{\var(U_n)}{n^{2(5-\alpha_{11})}}\to0
	\qquad\mathbb\Rightarrow\qquad
	\frac{U_n}{n^{5-\alpha_{11}}}\xrightarrow{P}0.
	\]
	
	\emph{Step 4: Ratio.}
	Combining the numerator and denominator,
	\[
	\widehat\beta_n-\beta
	=
	\frac{U_n}{V_n}
	=
	\frac{U_n/n^{5-\alpha_{11}}}{V_n/n^{5-\alpha_{11}}}
	\xrightarrow{P}0,
	\]
	because the numerator converges to $0$ in probability and the denominator
	converges in probability to $c_V\in(0,\infty)$.
	This proves $\widehat\beta_n\to\beta$ in probability under
	$\alpha_{11}<4$, $\alpha_{11} \neq 0$,  $\alpha_{12}>\alpha_{11}$ and
	$\alpha_{22}>\alpha_{11}-1$.
\end{proof}

\newpage
\section{Additional Figures and Tables}
In this section, we provide the additional figures and tables for the numerical experiments.
\begin{figure}[!h]
	\centering
	\includegraphics[width=0.8\linewidth]{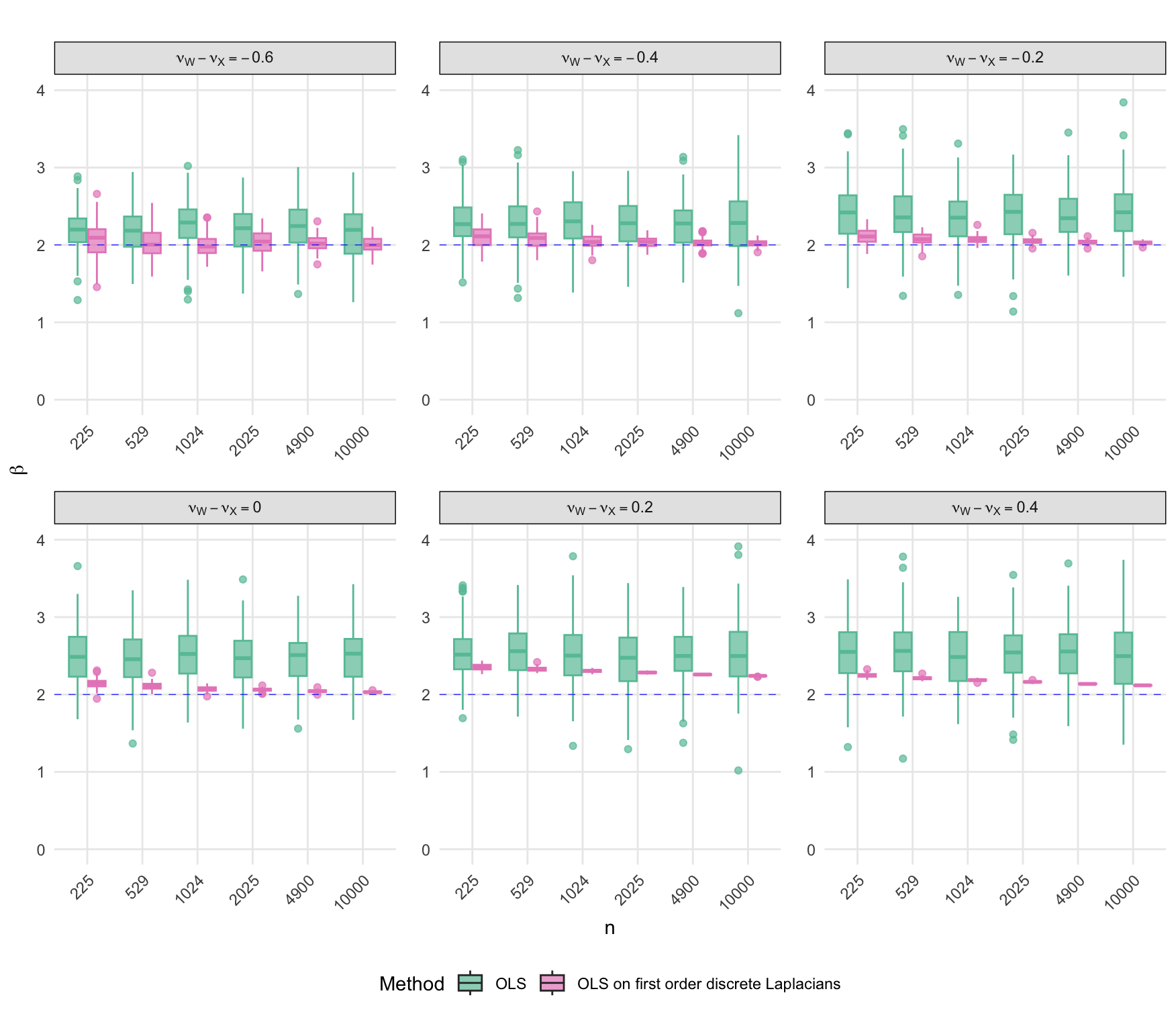}
	\caption{Estimates of $\beta$ for regression between GRF $Y=X\beta+W$ and $X$ in $\mathbb R^2$ when both the exposure $X$ and the confounder $W$ have Matérn covariances with smoothnesses $\nu_X$ and $\nu_W$ respectively.}
	\label{fig:mat2d}
\end{figure}

\begin{table}[!h]
	\centering
	\caption{\label{tab:rmse_table}Root Mean Squared Error (RMSE) for estimation of $\beta$ by different methods under spatial confounding in $1$-dimensional domain.}
	\centering
	\begin{tabular}[t]{rrrrrr}
		\toprule
		$\nu_X$ & $\nu_W - \nu_X$ & $n$ & OLS$_\blue n(X,Y)$ & OLS$_\blue n^{(1)}(X,Y)$ & OLS$_\blue n^{(2)}(X,Y)$\\
		\midrule
		0.7 & -0.6 & 100 & 0.28 & 0.53 & 1.00\\
		0.7 & -0.6 & 500 & 0.31 & 0.54 & 1.03\\
		0.7 & -0.6 & 1000 & 0.30 & 0.57 & 1.17\\
		0.7 & -0.6 & 2000 & 0.35 & 0.60 & 1.15\\\addlinespace
		0.7 & -0.3 & 100 & 0.61 & 0.28 & 0.38\\
		0.7 & -0.3 & 500 & 0.61 & 0.16 & 0.20\\
		0.7 & -0.3 & 1000 & 0.58 & 0.16 & 0.22\\
		0.7 & -0.3 & 2000 & 0.58 & 0.12 & 0.18\\\addlinespace
		0.7 & 0.0 & 100 & 0.53 & 0.19 & 0.16\\
		0.7 & 0.0 & 500 & 0.56 & 0.10 & 0.06\\
		0.7 & 0.0 & 1000 & 0.55 & 0.07 & 0.05\\
		0.7 & 0.0 & 2000 & 0.53 & 0.05 & 0.03\\\addlinespace
		0.7 & 0.3 & 100 & 0.63 & 0.17 & 0.10\\
		0.7 & 0.3 & 500 & 0.57 & 0.09 & 0.05\\
		0.7 & 0.3 & 1000 & 0.58 & 0.07 & 0.03\\
		0.7 & 0.3 & 2000 & 0.59 & 0.05 & 0.02\\\addlinespace
		1.2 & -0.6 & 100 & 0.68 & 0.37 & 0.70\\
		1.2 & -0.6 & 500 & 0.68 & 0.38 & 0.87\\
		1.2 & -0.6 & 1000 & 0.70 & 0.34 & 0.97\\
		1.2 & -0.6 & 2000 & 0.71 & 0.36 & 0.93\\
		\addlinespace
		1.2 & -0.3 & 100 & 0.80 & 0.38 & 0.31\\
		1.2 & -0.3 & 500 & 0.86 & 0.30 & 0.21\\
		1.2 & -0.3 & 1000 & 0.84 & 0.32 & 0.18\\
		1.2 & -0.3 & 2000 & 0.74 & 0.31 & 0.13\\\addlinespace
		1.2 & 0.0 & 100 & 0.67 & 0.32 & 0.16\\
		1.2 & 0.0 & 500 & 0.73 & 0.28 & 0.06\\
		1.2 & 0.0 & 1000 & 0.73 & 0.28 & 0.04\\
		1.2 & 0.0 & 2000 & 0.73 & 0.28 & 0.04\\\addlinespace
		1.2 & 0.3 & 100 & 0.63 & 0.31 & 0.12\\
		1.2 & 0.3 & 500 & 0.72 & 0.28 & 0.05\\
		1.2 & 0.3 & 1000 & 0.73 & 0.27 & 0.04\\
		1.2 & 0.3 & 2000 & 0.73 & 0.26 & 0.03\\
		\bottomrule
	\end{tabular}
\end{table}

\begin{table}[!t]
	\centering
	\caption{\label{tab:bias_table}Biases of the estimators of $\beta$ by different methods under spatial confounding in $1$-dimensional domain.}
	\centering
	\begin{tabular}[t]{rrrrrr}
		\toprule
		$\nu_X$ & $\nu_W - \nu_X$ & $n$ & OLS$_\blue n(X,Y)$ & OLS$_\blue n^{(1)}(X,Y)$ & OLS$_\blue n^{(2)}(X,Y)$\\
		\midrule
		0.7 & -0.6 & 100 & -0.010 & 0.018 & -0.050\\
		0.7 & -0.6 & 500 & -0.016 & -0.052 & -0.093\\
		0.7 & -0.6 & 1000 & -0.014 & -0.077 & -0.167\\
		0.7 & -0.6 & 2000 & -0.036 & -0.047 & -0.074\\\addlinespace
		0.7 & -0.3 & 100 & 0.465 & 0.155 & 0.083\\
		0.7 & -0.3 & 500 & 0.450 & 0.083 & 0.020\\
		0.7 & -0.3 & 1000 & 0.446 & 0.094 & 0.079\\
		0.7 & -0.3 & 2000 & 0.426 & 0.045 & 0.006\\\addlinespace
		0.7 & 0.0 & 100 & 0.360 & 0.163 & 0.105\\
		0.7 & 0.0 & 500 & 0.376 & 0.083 & 0.040\\
		0.7 & 0.0 & 1000 & 0.359 & 0.065 & 0.031\\
		0.7 & 0.0 & 2000 & 0.359 & 0.048 & 0.019\\\addlinespace
		0.7 & 0.3 & 100 & 0.372 & 0.154 & 0.096\\
		0.7 & 0.3 & 500 & 0.351 & 0.088 & 0.044\\
		0.7 & 0.3 & 1000 & 0.361 & 0.068 & 0.029\\
		0.7 & 0.3 & 2000 & 0.363 & 0.052 & 0.021\\\addlinespace
		1.2 & -0.6 & 100 & 0.007 & 0.038 & 0.057\\
		1.2 & -0.6 & 500 & 0.069 & 0.026 & 0.126\\
		1.2 & -0.6 & 1000 & -0.012 & 0.006 & -0.083\\
		1.2 & -0.6 & 2000 & -0.020 & -0.048 & -0.092\\
		\addlinespace
		1.2 & -0.3 & 100 & 0.380 & 0.263 & 0.092\\
		1.2 & -0.3 & 500 & 0.318 & 0.212 & 0.041\\
		1.2 & -0.3 & 1000 & 0.376 & 0.231 & 0.028\\
		1.2 & -0.3 & 2000 & 0.329 & 0.217 & 0.022\\\addlinespace
		1.2 & 0.0 & 100 & 0.340 & 0.239 & 0.117\\
		1.2 & 0.0 & 500 & 0.315 & 0.226 & 0.043\\
		1.2 & 0.0 & 1000 & 0.306 & 0.228 & 0.031\\
		1.2 & 0.0 & 2000 & 0.336 & 0.216 & 0.027\\\addlinespace
		1.2 & 0.3 & 100 & 0.352 & 0.257 & 0.111\\
		1.2 & 0.3 & 500 & 0.318 & 0.227 & 0.050\\
		1.2 & 0.3 & 1000 & 0.344 & 0.226 & 0.036\\
		1.2 & 0.3 & 2000 & 0.336 & 0.209 & 0.025\\
		\bottomrule
	\end{tabular}
\end{table}

\begin{table}[!h]
	\centering
	\caption{\label{tab:variance_table} Standard deviations of the estimators of $\beta$ by different methods under spatial confounding in $1$-dimensional domain.}
	\centering
	\begin{tabular}[t]{rrrrrr}
		\toprule
		$\nu_X$ & $\nu_W - \nu_X$ & $n$ & OLS$_\blue n(X,Y)$ & OLS$_\blue n^{(1)}(X,Y)$ & OLS$_\blue n^{(2)}(X,Y)$\\
		\midrule
		0.7 & -0.6 & 100 & 0.279 & 0.525 & 0.999\\
		0.7 & -0.6 & 500 & 0.314 & 0.533 & 1.030\\
		0.7 & -0.6 & 1000 & 0.295 & 0.567 & 1.157\\
		0.7 & -0.6 & 2000 & 0.348 & 0.594 & 1.150\\\addlinespace
		0.7 & -0.3 & 100 & 0.396 & 0.233 & 0.367\\
		0.7 & -0.3 & 500 & 0.407 & 0.142 & 0.203\\
		0.7 & -0.3 & 1000 & 0.366 & 0.132 & 0.210\\
		0.7 & -0.3 & 2000 & 0.393 & 0.111 & 0.175\\\addlinespace
		0.7 & 0.0 & 100 & 0.385 & 0.097 & 0.114\\
		0.7 & 0.0 & 500 & 0.416 & 0.048 & 0.050\\
		0.7 & 0.0 & 1000 & 0.422 & 0.035 & 0.037\\
		0.7 & 0.0 & 2000 & 0.389 & 0.026 & 0.026\\\addlinespace
		0.7 & 0.3 & 100 & 0.504 & 0.063 & 0.041\\
		0.7 & 0.3 & 500 & 0.444 & 0.021 & 0.011\\
		0.7 & 0.3 & 1000 & 0.459 & 0.015 & 0.006\\
		0.7 & 0.3 & 2000 & 0.466 & 0.013 & 0.003\\\addlinespace
		1.2 & -0.6 & 100 & 0.682 & 0.373 & 0.694\\
		1.2 & -0.6 & 500 & 0.674 & 0.380 & 0.861\\
		1.2 & -0.6 & 1000 & 0.701 & 0.335 & 0.962\\
		1.2 & -0.6 & 2000 & 0.709 & 0.354 & 0.922\\
		\addlinespace
		1.2 & -0.3 & 100 & 0.698 & 0.272 & 0.291\\
		1.2 & -0.3 & 500 & 0.794 & 0.214 & 0.209\\
		1.2 & -0.3 & 1000 & 0.747 & 0.226 & 0.177\\
		1.2 & -0.3 & 2000 & 0.658 & 0.221 & 0.129\\\addlinespace
		1.2 & 0.0 & 100 & 0.576 & 0.216 & 0.105\\
		1.2 & 0.0 & 500 & 0.659 & 0.173 & 0.044\\
		1.2 & 0.0 & 1000 & 0.659 & 0.171 & 0.032\\
		1.2 & 0.0 & 2000 & 0.644 & 0.173 & 0.023\\\addlinespace
		1.2 & 0.3 & 100 & 0.528 & 0.170 & 0.040\\
		1.2 & 0.3 & 500 & 0.641 & 0.166 & 0.011\\
		1.2 & 0.3 & 1000 & 0.647 & 0.152 & 0.006\\
		1.2 & 0.3 & 2000 & 0.647 & 0.146 & 0.004\\
		\bottomrule
	\end{tabular}
\end{table}

\begin{table}[!t]
	\centering
	\caption{\label{tab:mse_table_2d}Root Mean Squared Error (RMSE) for estimation of $\beta$ by different methods under spatial confounding in $2$-dimensional domain.}
	\centering
	\begin{tabular}[t]{rrrrr}
		\toprule
		$\nu_X$ & $\nu_W - \nu_X$ & $n$ & OLS$_\blue n(X,Y)$ & Lap$_\blue n^{(1)}(X,Y)$\\
		\midrule
		1 & -0.6 & 225 & 0.34 & 0.23\\
		1 & -0.6 & 529 & 0.36 & 0.19\\
		1 & -0.6 & 1024 & 0.42 & 0.14\\
		1 & -0.6 & 2025 & 0.38 & 0.15\\
		1 & -0.6 & 4900 & 0.38 & 0.10\\
		1 & -0.6 & 10000 & 0.38 & 0.10\\\addlinespace
		1 & -0.4 & 225 & 0.44 & 0.18\\
		1 & -0.4 & 529 & 0.45 & 0.14\\
		1 & -0.4 & 1024 & 0.45 & 0.10\\
		1 & -0.4 & 2025 & 0.42 & 0.08\\
		1 & -0.4 & 4900 & 0.40 & 0.06\\
		1 & -0.4 & 10000 & 0.50 & 0.05\\\addlinespace
		1 & -0.2 & 225 & 0.55 & 0.15\\
		1 & -0.2 & 529 & 0.56 & 0.11\\
		1 & -0.2 & 1024 & 0.51 & 0.09\\
		1 & -0.2 & 2025 & 0.53 & 0.06\\
		1 & -0.2 & 4900 & 0.56 & 0.05\\
		1 & -0.2 & 10000 & 0.57 & 0.04\\\addlinespace
		1 & 0.0 & 225 & 0.62 & 0.15\\
		1 & 0.0 & 529 & 0.61 & 0.12\\
		1 & 0.0 & 1024 & 0.65 & 0.08\\
		1 & 0.0 & 2025 & 0.60 & 0.06\\
		1 & 0.0 & 4900 & 0.59 & 0.05\\
		1 & 0.0 & 10000 & 0.62 & 0.03\\\addlinespace
		1 & 0.2 & 225 & 0.66 & 0.36\\
		1 & 0.2 & 529 & 0.67 & 0.33\\
		1 & 0.2 & 1024 & 0.65 & 0.30\\
		1 & 0.2 & 2025 & 0.61 & 0.28\\
		1 & 0.2 & 4900 & 0.63 & 0.26\\
		1 & 0.2 & 10000 & 0.68 & 0.24\\
		\addlinespace
		1 & 0.4 & 225 & 0.68 & 0.25\\
		1 & 0.4 & 529 & 0.71 & 0.21\\
		1 & 0.4 & 1024 & 0.61 & 0.19\\
		1 & 0.4 & 2025 & 0.66 & 0.16\\
		1 & 0.4 & 4900 & 0.69 & 0.14\\
		1 & 0.4 & 10000 & 0.67 & 0.12\\
		\bottomrule
	\end{tabular}
\end{table}

\begin{table}[!h]
	\centering
	\caption{\label{tab:bias_table_2d}Biases of the estimators of $\beta$ by different methods under spatial confounding in $2$-dimensional domain.}
	\centering
	\begin{tabular}[t]{rrrrr}
		\toprule
		$\nu_X$ & $\nu_W - \nu_X$ & $n$ & OLS$_\blue n(X,Y)$ & Lap$_\blue n^{(1)}(X,Y)$\\
		\midrule
		1 & -0.6 & 225 & 0.195 & 0.068\\
		1 & -0.6 & 529 & 0.192 & 0.017\\
		1 & -0.6 & 1024 & 0.259 & -0.004\\
		1 & -0.6 & 2025 & 0.184 & 0.038\\
		1 & -0.6 & 4900 & 0.217 & 0.019\\
		1 & -0.6 & 10000 & 0.149 & 0.005\\\addlinespace
		1 & -0.4 & 225 & 0.294 & 0.099\\
		1 & -0.4 & 529 & 0.290 & 0.074\\
		1 & -0.4 & 1024 & 0.292 & 0.041\\
		1 & -0.4 & 2025 & 0.272 & 0.035\\
		1 & -0.4 & 4900 & 0.260 & 0.023\\
		1 & -0.4 & 10000 & 0.295 & 0.016\\\addlinespace
		1 & -0.2 & 225 & 0.396 & 0.109\\
		1 & -0.2 & 529 & 0.391 & 0.082\\
		1 & -0.2 & 1024 & 0.354 & 0.070\\
		1 & -0.2 & 2025 & 0.377 & 0.050\\
		1 & -0.2 & 4900 & 0.397 & 0.036\\
		1 & -0.2 & 10000 & 0.447 & 0.027\\\addlinespace
		1 & 0.0 & 225 & 0.488 & 0.140\\
		1 & 0.0 & 529 & 0.466 & 0.107\\
		1 & 0.0 & 1024 & 0.539 & 0.070\\
		1 & 0.0 & 2025 & 0.478 & 0.060\\
		1 & 0.0 & 4900 & 0.470 & 0.043\\
		1 & 0.0 & 10000 & 0.492 & 0.031\\
		\addlinespace
		1 & 0.2 & 225 & 0.530 & 0.353\\
		1 & 0.2 & 529 & 0.542 & 0.329\\
		1 & 0.2 & 1024 & 0.493 & 0.304\\
		1 & 0.2 & 2025 & 0.446 & 0.284\\
		1 & 0.2 & 4900 & 0.510 & 0.258\\
		1 & 0.2 & 10000 & 0.519 & 0.241\\
		\addlinespace
		1 & 0.4 & 225 & 0.547 & 0.247\\
		1 & 0.4 & 529 & 0.558 & 0.211\\
		1 & 0.4 & 1024 & 0.488 & 0.186\\
		1 & 0.4 & 2025 & 0.508 & 0.163\\
		1 & 0.4 & 4900 & 0.551 & 0.137\\
		1 & 0.4 & 10000 & 0.489 & 0.118\\
		\bottomrule
	\end{tabular}
\end{table}

\begin{table}[!h]
	\centering
	\caption{\label{tab:variance_table_2d}Standard deviations of the estimators of $\beta$ by different methods under spatial confounding in $2$-dimensional domain.}
	\centering
	\begin{tabular}[t]{rrrrr}
		\toprule
		$\nu_X$ & $\nu_W - \nu_X$ & $n$ & OLS$_\blue n(X,Y)$ & Lap$_\blue n^{(1)}(X,Y)$\\
		\midrule
		1 & -0.6 & 225 & 0.279 & 0.224\\
		1 & -0.6 & 529 & 0.305 & 0.193\\
		1 & -0.6 & 1024 & 0.326 & 0.137\\
		1 & -0.6 & 2025 & 0.331 & 0.148\\
		1 & -0.6 & 4900 & 0.311 & 0.098\\
		1 & -0.6 & 10000 & 0.348 & 0.097\\\addlinespace
		1 & -0.4 & 225 & 0.322 & 0.145\\
		1 & -0.4 & 529 & 0.342 & 0.119\\
		1 & -0.4 & 1024 & 0.338 & 0.089\\
		1 & -0.4 & 2025 & 0.318 & 0.073\\
		1 & -0.4 & 4900 & 0.302 & 0.060\\
		1 & -0.4 & 10000 & 0.409 & 0.044\\\addlinespace
		1 & -0.2 & 225 & 0.379 & 0.099\\
		1 & -0.2 & 529 & 0.400 & 0.074\\
		1 & -0.2 & 1024 & 0.364 & 0.051\\
		1 & -0.2 & 2025 & 0.377 & 0.038\\
		1 & -0.2 & 4900 & 0.395 & 0.028\\
		1 & -0.2 & 10000 & 0.361 & 0.023\\\addlinespace
		1 & 0.0 & 225 & 0.377 & 0.065\\
		1 & 0.0 & 529 & 0.396 & 0.049\\
		1 & 0.0 & 1024 & 0.367 & 0.031\\
		1 & 0.0 & 2025 & 0.360 & 0.020\\
		1 & 0.0 & 4900 & 0.352 & 0.017\\
		1 & 0.0 & 10000 & 0.379 & 0.009\\\addlinespace
		1 & 0.2 & 225 & 0.392 & 0.039\\
		1 & 0.2 & 529 & 0.395 & 0.029\\
		1 & 0.2 & 1024 & 0.425 & 0.018\\
		1 & 0.2 & 2025 & 0.421 & 0.011\\
		1 & 0.2 & 4900 & 0.372 & 0.006\\
		1 & 0.2 & 10000 & 0.445 & 0.004\\
		\addlinespace
		1 & 0.4 & 225 & 0.406 & 0.027\\
		1 & 0.4 & 529 & 0.434 & 0.019\\
		1 & 0.4 & 1024 & 0.371 & 0.012\\
		1 & 0.4 & 2025 & 0.420 & 0.009\\
		1 & 0.4 & 4900 & 0.410 & 0.005\\
		1 & 0.4 & 10000 & 0.464 & 0.003\\
		\bottomrule
	\end{tabular}
\end{table}

\end{document}